\documentclass{article}

\usepackage{mathtools}
\usepackage{amsthm}
\usepackage{nicefrac}

\usepackage[smallerops]{newtx}
\usepackage[scale=1]{tx-ds}

\usepackage[english]{babel}
\usepackage{csquotes}
\usepackage[
    backend=biber,
    style=numeric,
    citestyle=numeric,
]{biblatex}
\DeclareNameFormat{sortname}{family-given}
\AtEveryBibitem{%
    \clearlist{location}%
    \clearfield{note}%
    \clearfield{urldate}%
    \clearfield{urlyear}%
    \ifentrytype{software}{}{%
        \clearfield{url}%
    }
}
\renewbibmacro{in:}{}
\usepackage[margin=25mm]{geometry}
\usepackage[bottom, stable]{footmisc}

\usepackage{authblk}
\usepackage[hidelinks]{hyperref}

\usepackage{enumitem}
\setlist[enumerate]{
    nosep,
    label=\textit{\roman*}),
}

\theoremstyle{plain}
\newtheorem{theorem}{Theorem}[section]
\newtheorem{lemma}[theorem]{Lemma}
\newtheorem{proposition}[theorem]{Proposition}
\newtheorem{corollary}[theorem]{Corollary}

\theoremstyle{definition}
\newtheorem{definition}[theorem]{Definition}
\newtheorem{hypothesis}[theorem]{Hypothesis}
\theoremstyle{remark}
\newtheorem{remark}[theorem]{Remark}

\makeatletter
\newcounter{step}
\newcommand*{\step}[1]{%
    \stepcounter{step}%
    \textbf{Step \arabic{step}: #1.\enspace}%
    \@ifnextchar\par\@gobble\relax%
}
\makeatother

\newcommand*{\beq}[1]{\begin{equation} \label{#1}}
\newcommand*{\eeq}{\end{equation}}
\newcommand*{\quadtext}[1]{\quad \text{#1} \quad}
\newcommand*{\qquadtext}[1]{\qquad \text{#1} \qquad}
\newcommand*{\defeq}{\coloneqq}
\newcommand*{\eqdef}{\eqqcolon}

\let\Re\relax
\let\Im\relax
\DeclareMathOperator{\Re}{Re}
\DeclareMathOperator{\Im}{Im}
\DeclareMathOperator{\supp}{supp}
\DeclareMathOperator{\diam}{diam}

\DeclarePairedDelimiter{\floor}{\lfloor}{\rfloor}
\DeclarePairedDelimiter{\ceil}{\lceil}{\rceil}
\DeclarePairedDelimiter{\abs}{\lvert}{\rvert}
\DeclarePairedDelimiter{\norm}{\lVert}{\rVert}
\DeclarePairedDelimiterX{\inprod}[2]{\langle}{\rangle}{#1, #2}
\DeclareMathOperator{\AC}{AC}
\newcommand*{\loc}{_{\mathrm{loc}}}
\newcommand*{\quasi}[1]{#1^{[1]}}

\NewDocumentCommand{\diff}{s o m}{%
    \mathop{}\!\mathrm{d}%
    \IfNoValueTF{#2}{}{^{#2}}%
    \IfBooleanTF{#1}{\mkern -1mu #3}{#3}%
}

\DeclareMathOperator{\tr}{tr}
\newcommand*{\transpose}{\mathsf{T}}
\newenvironment{smallpmatrix}{%
    \bigl( \begin{smallmatrix}
}{\end{smallmatrix} \bigr)}

\newcommand*{\sigmaalgebra}{\(\upsigma\)-algebra}
\newcommand*{\charf}[1]{\mathbb{1}_{#1}}
\newcommand*{\eqlaw}{\overset{\mathrm{law}}{=}}
\newcommand*{\probto}[1][]{\xrightarrow[#1]{\mathbb{P}}}
\newcommand*{\lawto}[1][]{\xrightarrow[#1]{\mathrm{law}}}

\DeclarePairedDelimiterXPP{\bprob}[1]{\mathbb{P}}{[}{]}{}{#1}
\DeclarePairedDelimiterXPP{\pprob}[1]{\mathbb{P}}{(}{)}{}{#1}
\newcommand*{\expect}{\mathop{\mbox{}\mathbb{E}}}
\DeclarePairedDelimiterXPP{\bexpect}[1]{\mathbb{E}}{[}{]}{}{#1}
\newcommand*{\given}[1][]{\nonscript\>#1\vert\nonscript\>\mathopen{}}
\DeclarePairedDelimiterX{\quadvar}[1]{\langle}{\rangle}{#1}
\DeclarePairedDelimiterX{\crossvar}[2]{\langle}{\rangle}{#1, #2}
\makeatletter
\newsavebox{\@brx}
\newcommand*{\llangle}[1][]{\savebox{\@brx}{\(\m@th{#1\langle}\)}%
    \mathopen{\copy\@brx\mkern2mu\kern-0.9\wd\@brx\usebox{\@brx}}}
\newcommand*{\rrangle}[1][]{\savebox{\@brx}{\(\m@th{#1\rangle}\)}%
    \mathclose{\copy\@brx\mkern2mu\kern-0.9\wd\@brx\usebox{\@brx}}}
\makeatother
\newcommand*{\subGbracket}[2][]{%
    \llangle[#1] #2\rrangle[#1]%
}

\DeclareMathOperator{\erf}{erf}
\DeclareMathOperator{\Ai}{Ai}
\DeclareMathOperator{\Bi}{Bi}
\DeclareMathOperator{\SAi}{SAi}
\newcommand*{\intsol}[1]{\Phi_{#1}}

\newcommand*{\normal}[2]{\mathcal{N}(#1, #2)}
\DeclareMathOperator{\uniform}{Unif}

\newcommand*{\HBM}[1][\beta]{\mathcal{B}_{#1}}
\makeatletter
\newcommand*{\ABM}{\@ifstar\@ABM\@@ABM}
\newcommand*{\@ABM}{\tilde{X}_{\beta,\shift}}
\newcommand*{\@@ABM}{X_\beta}
\makeatother
\newcommand*{\UHP}{\mathbb{H}}
\newcommand*{\clUHP}{\overline{\mathbb{H}}_\infty}
\newcommand*{\RiemannSphere}{\mathbb{C}_\infty}
\newcommand*{\projection}{\mathscr{P}}

\DeclareMathOperator{\CS}{CS}
\DeclareMathOperator{\TM}{TM}
\DeclareMathOperator{\TMLP}{TM_{LP}}
\DeclareMathOperator{\TMLC}{TM_{LC}}
\DeclareMathOperator{\Hol}{Hol}
\newcommand*{\testfunc}{\varphi}
\DeclarePairedDelimiterXPP{\wronskian}[2]{\mathcal{W}}{(}{)}{}{#1, #2}

\newcommand*{\shift}{E}
\newcommand*{\setshift}[1]{%
    \renewcommand*{\shift}{#1}%
}
\newcommand*{\timechange}{\eta_{\shift}}

\newcommand*{\sine}{\mathrm{sine}_\beta}
\makeatletter
\newcommand*{\sinemat}{\@ifstar\@sinemat\@@sinemat}
\newcommand*{\@sinemat}{\tilde{R}_\beta}
\newcommand*{\@@sinemat}[1][\beta]{R_{#1}}
\newcommand*{\sineTM}{\@ifstar\@sineTM\@@sineTM}
\newcommand*{\@sineTM}{T_{\sinemat*}}
\newcommand*{\@@sineTM}{T_{\sinemat}}
\makeatother
\newcommand*{\sineWT}{m_{\sinemat}}
\newcommand*{\logtime}{\upsilon}
\newcommand*{\slogtime}{\upsilon_{\shift}}
\newcommand*{\taumat}{R_{\tau_\beta}}

\newcommand*{\Airy}{\mathrm{Airy}_\beta}
\newcommand*{\Airyop}{\mathcal{H}_\beta}
\newcommand*{\sAiryop}{\mathcal{H}_{\beta,\shift}}
\makeatletter
\newcommand*{\sAirymat}{\@ifstar\@sAirymat\@@sAirymat}
\newcommand*{\@sAirymat}{\tilde{H}_{\beta,\shift}}
\newcommand*{\@@sAirymat}[1][\beta]{H_{#1,\shift}}
\newcommand*{\sAiryTM}{\@ifstar\@sAiryTM\@@sAiryTM}
\newcommand*{\@sAiryTM}{T_{\sAirymat*}}
\newcommand*{\@@sAiryTM}{T_{\sAirymat}}
\makeatother
\newcommand*{\sAiryWT}{m_{\sAirymat}}
\newcommand*{\sAirySLtoCS}[1][\beta]{A_{#1, \shift}}
\newcommand*{\sAiryDirichlet}[1][\beta]{\mathrm{f}_{#1, \shift}}
\newcommand*{\sAiryNeumann}[1][\beta]{\mathrm{g}_{#1, \shift}}
\newcommand*{\unsAiryop}{\tilde{\mathcal{H}}_\beta}
\newcommand*{\unsAiryNeumann}{\tilde{\mathrm{g}}_{\beta, \shift}}

\newcommand*{\trigarg}{\theta_{\shift}}
\newcommand*{\diln}{c_{\shift}}
\newcommand*{\lasttime}{\tau_{\shift}}
\newcommand*{\secondtolasttime}{\tau_{\shift,\alpha}}
\newcommand*{\timedom}{\mathcal{I}_{\shift, \alpha}}

\newcommand*{\sBM}{B_{\shift}}

\newcommand*{\dirichlet}{\mathrm{f}}
\newcommand*{\neumann}{\mathrm{g}}
\makeatletter
\newcommand*{\newdirichletneumanncommands}{\@ifstar\@@newdirichletneumanncommands\@newdirichletneumanncommands}
\newcommand*{\@newdirichletneumanncommands}[2]{%
    \expandafter\newcommand\expandafter*\csname #1\endcsname{#2_{\beta, \shift}}
    \expandafter\newcommand\expandafter*\csname #1D\endcsname{#2_{\beta, \shift}^{\dirichlet}}
    \expandafter\newcommand\expandafter*\csname #1N\endcsname{#2_{\beta, \shift}^{\neumann}}
}
\newcommand*{\@@newdirichletneumanncommands}[3]{%
    \expandafter\newcommand\expandafter*\csname #1\endcsname{#2_{\beta, \shift}^{#3}}
    \expandafter\newcommand\expandafter*\csname #1D\endcsname{#2_{\beta, \shift}^{#3, \dirichlet}}
    \expandafter\newcommand\expandafter*\csname #1N\endcsname{#2_{\beta, \shift}^{#3, \neumann}}
}
\makeatother

\newdirichletneumanncommands{AMP}{\rho}
\newdirichletneumanncommands{phase}{\xi}
\newcommand*{\diffamps}{\Delta^\rho_{\beta, \shift}}
\newcommand*{\diffphases}{\Delta^\xi_{\beta, \shift}}
\newcommand*{\sumamps}{\Sigma^\rho_{\beta, \shift}}
\newcommand*{\sumphases}{\Sigma^\xi_{\beta, \shift}}

\newdirichletneumanncommands*{phasedrift}{R}{\xi}
\newdirichletneumanncommands*{phasediff}{S}{\xi}
\newdirichletneumanncommands*{ampdrift}{R}{\rho}
\newdirichletneumanncommands*{ampdiff}{S}{\rho}
\newcommand*{\sumdrift}{R_{\beta, \shift}^{\Sigma}}
\newcommand*{\sumdiff}{S_{\beta, \shift}^{\Sigma}}

\newcommand*{\GBM}{Z_{\beta, \shift}}
\newcommand*{\AiryABM}{Y_{\beta, \shift}}

\newcommand*{\rBM}{\reflectbox{\(\vec{\reflectbox{\(B\)}}\)}}
\newcommand*{\ramp}{r_{\beta}}
\newcommand*{\rphase}{\xi_{\beta}}

\counterwithin{equation}{section}

\title{Operator level soft edge to bulk transition in \(\beta\)-ensembles via canonical systems}
\author[1]{Vincent Painchaud\thanks{vincent.painchaud@mail.mcgill.ca}}
\author[2]{Elliot Paquette\thanks{elliot.paquette@mcgill.ca}}
\affil[1]{Department of Mathematics, KTH Royal Institute of Technology}
\affil[2]{Department of Mathematics and Statistics, McGill University}
\date{\today}

\begin{document}

\maketitle

\begin{abstract}
The stochastic Airy and sine operators, which are respectively a random Sturm--Liouville operator and a random Dirac operator, characterize the soft edge and bulk scaling limits of \(\beta\)-ensembles.  Dirac and Sturm--Liouville operators are distinct operator classes which can both be represented as canonical systems, which gives a unified framework for defining important properties, such as their spectral data. Seeing both as canonical systems, we prove that in a suitable high-energy scaling limit, the Airy operator converges to the sine operator.  We prove this convergence in the vague topology of canonical systems' coefficient matrices, and deduce the convergence of the associated Weyl--Titchmarsh functions and spectral measures. Our proof relies on a coupling between the Brownian paths that drive the two operators, under which the convergence holds in probability. This extends the corresponding result at the level of the eigenvalue point processes, proven in \cite{valko_continuum_2009} by comparison to the Gaussian $\beta$-ensemble.
\end{abstract}

\tableofcontents

\section{Introduction}

One of the central projects of random matrix theory, from its inception, is the description of the local statistics of eigenvalues in the large matrix size limit.  In the case of $\beta$-ensembles, which are point processes admitting the joint density 
\begin{equation}\label{eq:betaensemble}
    (\lambda_1, \lambda_2, \dots, \lambda_N) \mapsto \frac{1}{Z_{N,V,\beta}}\exp\Bigl(-\sum_{i=1}^N \beta N V(\lambda_i)\Bigr) \prod_{i > j} \abs{\lambda_i-\lambda_j}^{\beta} 
\end{equation}
for constraining potential $V \colon \mathbb{R} \to \mathbb{R}$ and inverse temperature $\beta > 0$ (see \cite{anderson_introduction_2009} for background), the important cases of $\beta \in \{1,2,4\}$ famously admit additional structure.  These allow for explicit computations of the correlation functions, albeit with slightly different formalisms in the cases $\beta=2$ versus the cases $\beta \in \{1,4\}$, and they make $N\to\infty$ limits relatively straightforward.  Moreover, when specializing to the classical $\beta\in\{1,2,4\}$ determinantal or Pfaffian cases, \emph{all} the local scaling limits retain this determinantal or Pfaffian structure; it is possible to work entirely within the category of determinantal (respectively, Pfaffian) point processes to study the various local scaling limits.

For general $\beta > 0$, a powerful overarching idea was introduced in \cite{edelman_random_2007}: to define random operators whose spectra would give the local point process of $\beta$-ensembles. These ideas were developed and implemented by \cite{ramirez_beta_2011} and \cite{ramirez_diffusion_2009} for the soft edge (Airy) and hard edge (Bessel) operators; the corresponding bulk (sine) operator \cite{valko_sine_beta_2017} followed two independent constructions of the \(\sine\) point process \cite{valko_continuum_2009} and \cite{killip_eigenvalue_2009}.  We also note that these point processes have been shown to be universal across a wide range of potentials \cite{bekerman_transport_2015, bourgade_bulk_2012, bourgade_edge_2014, bourgade_universality_2014, krishnapur_universality_2016, rider_universality_2019}.

There is a fundamental difference between the operator types that arise in the bulk versus the edge.  The sine operator is a \emph{random Dirac operator}, while the Airy and Bessel operators are \emph{random Schrödinger operators}.  This at first glance suggests that fundamentally different techniques are required to work with these operators, and this also impedes comparisons between them---how should one show the convergence (if possible) from the edge operators to the bulk operator?  Is there a single framework that can simultaneously describe all the general-$\beta$ point processes?

As noted by Valkó and Virág in~\cite{valko_sine_beta_2017}, there is actually a common representation that can be used for both the stochastic Airy and sine operators: the \emph{canonical system}. We will show in this paper that in this embedding, we are able to obtain the stochastic sine operator as a scaling limit of the stochastic Airy operator, thus giving an \emph{operator-level} convergence.  We further show that the \emph{spectral measures} of these canonical systems converge; this measure is a weighted form of the empirical measure of eigenvalues of the associated operator.  Moreover, we note that all known limit operators of $\beta$-ensembles (as well as the tridiagonal matrix models and the CMV models) also embed as canonical systems, which could therefore serve as a common mathematical framework for \emph{all} point process limits of $\beta$-ensembles.

\paragraph{Airy and sine operators.}

We now formally introduce the \emph{stochastic sine operator} of~\cite{valko_sine_beta_2017}, a random two-dimensional first-order differential operator defined as follows. Given \(\beta > 0\), let \(\HBM\) be a hyperbolic Brownian motion with variance \(\nicefrac{4}{\beta}\) started at \(i\) in the upper half-plane, meaning that \(\HBM\) solves the It\^o stochastic differential equation
\beq{eq.HBMSDE}
\diff{\HBM}(t) = \frac{2}{\sqrt{\beta}} \Im\HBM(t) \diff{W}(t)
\qquadtext{with}
\HBM(0) = i
\eeq
where \(W\) is a standard complex Brownian motion. An important property of a hyperbolic Brownian motion is that it always converges to a boundary point of the hyperbolic plane as \(t\to\infty\)~\cite[Proposition~7.6.7.3]{franchi_hyperbolic_2012}, which here (as we use the upper half-plane model of the hyperbolic plane) means that \(\HBM(\infty)\) is a well-defined real random variable. Now, setting
\beq{eq.sinemat}
\sinemat \defeq \frac{1}{2\Im\HBM} \ABM^\transpose \ABM
\qquadtext{with}
\ABM \defeq \begin{pmatrix} 1 & -\Re\HBM \\ 0 & \Im\HBM \end{pmatrix},
\eeq
the stochastic sine operator is the random differential operator sending \(u\colon (0,1) \to \mathbb{C}^2\) to
\beq{eq.sineop}
(\sinemat^{-1}\circ\logtime) Ju'
\quadtext{where}
J \defeq \begin{pmatrix} 0 & -1 \\ 1 & 0 \end{pmatrix}
\quadtext{and with boundary data}
\begin{cases}
    u(0) \parallel (1,0), \\ 
    u(1) \parallel (\Re\HBM(\infty), 1) &
    \text{if } \beta > 2.
\end{cases}
\eeq
Here \(\logtime(t) \defeq -\log(1-t)\) and $\parallel$ denotes parallel. Under these boundary conditions, the stochastic sine operator is self-adjoint on a suitable domain, and has discrete eigenvalues which are the points of the \(\sine\) point process~\cite{valko_sine_beta_2017}.

We also introduce the \emph{stochastic Airy operator}, which appears at the soft edge.  It is a random Schrödinger operator densely defined on \(L^2(0,\infty)\) as mapping \(f\) to
\beq{eq.AiryopWN}
\Airyop f(t) = - f''(t) + \Bigl( t + \frac{2}{\sqrt{\beta}} B'(t) \Bigr) f(t)
\qquadtext{with boundary data}
f(0) = 0
\eeq
where \(B\) is a standard Brownian motion. Due to the presence of the white noise in the potential, some care is required to make this definition rigorous, but it can be done using the theory of distributions~\cite{ramirez_beta_2011}, or through the theory of generalized Sturm--Liouville operators~\cite{minami_definition_2015} as we will see in Section~\ref{sec.AirysineCS.defAiry}. In both cases, the resulting operator is well-defined and self-adjoint on the appropriate domain, and it has a discrete spectrum which by definition is \(-\Airy\).


\paragraph{Canonical system embedding.}

Our goal is to prove the convergence of the (scaled) Airy operator to the sine operator, and therefore our first task is to find a suitable framework that can encompass both of them.  This brings us to canonical systems.

\begin{definition}
\label{def.CS}
A \emph{canonical system} on an interval \((a,b) \subset \mathbb{R}\) is a differential equation of the form
\[
Ju' = -zHu \qquadtext{with} J \defeq \begin{pmatrix} 0 & -1 \\ 1 & 0 \end{pmatrix},
\]
where \(H\colon (a,b) \to \mathbb{C}^{2\times 2}\) is called the \emph{coefficient matrix}, \(z \in \mathbb{C}\) and \(u\colon (a,b) \to \mathbb{C}^2\).
\end{definition}
\noindent It may be necessary to impose boundary conditions at $a$ and at $b$, which will play an important role going forward.

The eigenvalue equation for the stochastic sine operator can immediately be seen to be equivalent to a canonical system on \((0,1)\) with coefficient matrix \(\sinemat* \defeq \sinemat\circ\logtime\) and the same boundary data given in \eqref{eq.sineop}. The stochastic Airy operator cannot directly be represented as a canonical system, but we will see in Section~\ref{sec.CS.SL} that there is a general recipe to turn the eigenvalue equation of a (generalized) Sturm--Liouville operator into a canonical system. Applying it to the eigenvalue equation of the shifted and scaled stochastic Airy operator \(\sAiryop \defeq 2\sqrt{\shift} (\Airyop - \shift)\) yields the canonical system on \((0,\infty)\) with coefficient matrix
\beq{eq.sAirymat}
    \sAirymat \defeq \frac{1}{2\sqrt{\shift}} \begin{pmatrix} \sqrt{\shift} \sAiryNeumann^2 & \sAiryDirichlet\sAiryNeumann \\ \sAiryDirichlet\sAiryNeumann & \frac{1}{\sqrt{\shift}} \sAiryDirichlet^2 \end{pmatrix}
\eeq
where \(\sAiryDirichlet\) and \(\sAiryNeumann\) are fundamental solutions to \(\sAiryop f = 0\) which satisfy \(\sAiryDirichlet(0) = \sAiryNeumann'(0) = 1\) and \(\sAiryDirichlet'(0) = \sAiryNeumann(0) = 0\). 

So while this construction allows us to use a common first-order differential operator framework, there is an important distinction between the embedded sine and Airy systems.  The Airy canonical system is \enquote{degenerate}, in the sense that its coefficient matrix \(\sAirymat\) is not invertible, and therefore it has no associated operator of the same form as the stochastic sine operator. Following \cite{valko_sine_beta_2017}, we call canonical systems with invertible coefficient matrices \emph{Dirac} operators (note that this name is also used for other families of operators in canonical systems theory, c.f.~\cite[Section~6.4]{remling_spectral_2018}).

Nevertheless, the theory of canonical systems provides a way to think of such systems as operators in a generalized sense (as we will see Section \ref{sec.CS}), and therefore it provides a way to unify the analysis of the \(\Airy\) and \(\sine\) point processes under the same mathematical framework. In particular, the convergence of the shifted \(\Airy\) point process \(2\sqrt{\shift} (\Airy + \shift)\) to \(\sine\) can be upgraded to an operator-level convergence in the sense of the associated canonical systems.

Our first result in that sense is at the level of the coefficient matrices. We will see in more detail in Section~\ref{sec.CS.topologies} that spaces of coefficient matrices can be topologized by thinking of coefficient matrices as matrix-valued measures and testing them against compactly supported continuous functions, which results in what we call the \emph{vague topology} of coefficient matrices. What we prove is the following.

\begin{theorem}
\label{thm.CSconvinlaw}
\setshift{E_n}
Let \(\mathcal{I} \defeq [0,1)\) if \(\beta \leq 2\) and \(\mathcal{I} \defeq [0,1]\) if \(\beta > 2\). For any diverging sequence \(\{\shift\}_{n\in\mathbb{N}} \subset (0,\infty)\), there are \(\mathscr{C}^1\) bijections \(\timechange\colon [0,1+\varepsilon_{\shift}) \to [0,\infty)\) where \(\varepsilon_{\shift} = 0\) when \(\beta \leq 2\) but \(\varepsilon_{\shift} \downarrow 0\) when \(\beta > 2\) such that
\[
\sAirymat* \defeq \timechange'(\sAirymat\circ\timechange) \lawto[n\to\infty] \sinemat* = \sinemat\circ\logtime
\]
in the vague topology of coefficient matrices on \(\mathcal{I}\).
\end{theorem}

\begin{remark}
The number \(\varepsilon_{\shift}\) appears for technical reasons that will be explained in Sections~\ref{sec.CS} and~\ref{sec.AirysineCS}. At this point, the mapping \(\timechange\) should essentially be understood as a time change from \((0,1)\) to \((0,\infty)\).   See \eqref{eq.deftimechange} for a precise definition.
\end{remark}

As a corollary of this result, we can show the convergence of the systems' \emph{transfer matrices} (see Definition \ref{def.T}), which are their (matrix) solutions for a given value \(z\in\mathbb{C}\) when started at the identity matrix \(I_2\).

\begin{corollary}
\label{cor.TMconvinlaw}
\setshift{E_n}
Let \(\{\shift\}_{n\in\mathbb{N}} \subset (0,\infty)\) be a diverging sequence, and let \(\sAiryTM*, \sineTM* \colon \mathcal{I} \times \mathbb{C} \to \mathbb{C}^{2\times 2}\) be the transfer matrices of the shifted and scaled Airy system and of the sine system respectively. Then
\[
\sAiryTM* \lawto[n\to\infty] \sineTM*
\]
compactly on \(\mathcal{I} \times \mathbb{C}\).
\end{corollary}

Unfortunately, the vague convergence of a sequence of coefficient matrices is not strong enough to imply the convergence of the spectra of the canonical systems, in general. However, the Weyl theory of second-order differential operators generalizes to canonical systems: given a canonical system, one can always define a Weyl--Titchmarsh function (see Section \ref{sec.Weyl}), which is always a generalized Herglotz function (i.e., a holomorphic function from the upper half-plane \(\UHP\) into its closure \(\clUHP\) in the Riemann sphere) and essentially the Stieltjes transform of the system's spectral measure. By proving the convergence of the boundary conditions of the canonical systems, we can upgrade the convergence of the coefficient matrices to the convergence of the Weyl--Titchmarsh functions.

\begin{theorem}
\label{thm.WTconvinlaw}
\setshift{E_n}
Let \(\{\shift\}_{n\in\mathbb{N}} \subset (0,\infty)\) be a diverging sequence, and let \(\sAiryWT, \sineWT \colon \UHP \to \clUHP\) be the Weyl--Titchmarsh functions of the shifted and scaled Airy system and of the sine system respectively. Then
\[
    \sAiryWT \lawto[n\to\infty] \sineWT
\]
compactly on \(\UHP\). This holds jointly with the convergence of transfer matrices in Corollary \ref{cor.TMconvinlaw}, and furthermore the Weyl--Tichmarsh function is invariant under the time change, so 
\(m_{\sAirymat*}
=
\sAiryWT\).  In particular, the spectral measures of the corresponding systems converge vaguely in distribution.
\end{theorem}

\begin{remark}
\label{rmk.WTconvinlaw}
The spectral measures associated to \(\sAiryWT, \sineWT\) are pure point and have positive masses precisely at the (simple) eigenvalues of the shifted Airy and sine operators.  The masses are independent of the eigenvalues and are iid Gamma random variables with shape and rate parameters \(\nicefrac{\beta}{2}\) and \(\nicefrac{4}{\beta}\). This was proven in \cite[Proposition~3]{valko_palm_2023} in the case of the sine operator, and we prove it here in Corollary~\ref{cor.spectralweights} in the case of the shifted Airy operator. Hence, it follows that the associated empirical spectral measures of $\sAiryop$ converge vaguely to the empirical spectral measure of the sine operator, which gives a new proof of the convergence of the point process $2\smash{\sqrt{\shift}}(\Airy + \shift)$ to the $\sine$ point process.
\end{remark}

As part of the proof of this result, we also obtain the following asymptotics of solutions to \(\Airyop f = 0\) towards \(-\infty\), which complement the $t \to \infty$ asymptotics for these solutions derived in \cite{lambert_strong_2021}.

\begin{theorem}
\label{thm.asymptoticsSAi}
Let \(\Airyop\) be defined on the full real line from a two-sided standard Brownian motion \(B\). If \(f\) solves \(\Airyop f = 0\) and if \(f(-1)\) and \(f'(-1)\) are independent of the restriction of \(B\) to \((-\infty, -1)\), then for \(t \geq 1\),
\[
f(-t) = C_{f0}\, t^{\nicefrac{-1}{4} + \nicefrac{1}{2\beta}} e^{X(t)} \cos\rphase(t)
\qquadtext{and}
f'(-t) = C_{f0}\, t^{\nicefrac{1}{4} + \nicefrac{1}{2\beta}} e^{X(t)} \sin\rphase(t)
\]
where \(C_{f0}^2 \defeq f^2(-1) + {f'}^2(-1)\), \(\rphase\) is a process that satisfies \(\rphase(t) - 2\pi\floor[\big]{\frac{\rphase(t)}{2\pi}} \to U\) in law as \(t \to \infty\) for \(U \sim \uniform[0,2\pi)\), and \(X\) is a process such that for any \(\varepsilon, \delta > 0\), there are \(C, C' > 0\) for which
\[
\bprob[\Big]{\forall t \geq 1, \abs{X(t)} \leq C + C'(\log t)^{\nicefrac{1}{2} + \delta}} \geq 1 - \varepsilon.
\]
\end{theorem}

\subsection*{Discussion and related work}

\paragraph{Related work on operator convergence.}

Operator level convergence is an important topic in the limit theory of \(\beta\)-ensembles, and there are many significant related works.  \cite{dumaz_operator_2021} shows convergence of the Bessel operator in the limit of large $a$ (the charge at $0$) to the Airy operator.  The Bessel operator is again a random Schrödinger operator, and the convergence is in the \emph{norm-resolvent} sense (see \cite[Theorem 1.VIII.20]{reed_methods_1980} for this and other notions of resolvent convergence).  We note that it is possible to define the resolvent $(\mathcal{S}_H - z)^{-1}$ of a canonical system $H$ with appropriate boundary conditions (see Theorem \ref{thm.CS.resolvent} in appendix~\ref{apx.resolvents}).  Moreover, as a corollary of Theorems~\ref{thm.CSconvinlaw} and \ref{thm.WTconvinlaw}, for $z \in \mathbb{C} \setminus \mathbb{R}$ and any compactly supported continuous function $\varphi : \mathcal{I} \to \mathbb{C}^2$,
\[\setshift{E_n}
    (\mathcal{S}_{\sAirymat*} - z)^{-1}\varphi  
    \lawto[n\to\infty]
    (\mathcal{S}_{\sinemat*} - z)^{-1}\varphi
\]
compactly in $z$ (this essentially follows from the integral representation for a canonical system's resolvent---see Proposition~\ref{prop.CS.resolventconvergence}).  For a Sturm--Liouville operator, the canonical system resolvent $(\mathcal{S}_{\sAirymat*} - z)^{-1}$ is conjugate to the more usual Sturm--Liouville resolvent (see \eqref{eq.resolventconjugation}).  Hence, this convergence is more closely related to the \emph{strong-resolvent} convergence, albeit the operators are not defined on the same Hilbert space.  We note that in related problems in random Schrödinger theory, norm-resolvent convergence can be too strong (see \cite[Theorem 1.6]{dumaz_delocalized_2023}). Although we do not study resolvent convergence in more detail in this paper, for comparison with other works we give in Appendix~\ref{apx.resolvents} some basic results about resolvents of canonical systems and their convergence.

The original work of \cite{ramirez_beta_2011} also constitutes a type of operator convergence, in that it establishes convergence of the Dumitriu--Edelman tridiagonal matrix \cite{dumitriu_matrix_2002}, embedded as a discrete difference operator, to the Airy operator in a sufficiently strong sense to get convergence of eigenvalues and eigenvectors.  In contrast, convergence in the bulk, for the Gaussian $\beta$-ensemble, requires making a Schrödinger to Dirac type convergence.  For the circular $\beta$-ensembles, this instead can be done using Dirac to Dirac type convergence~\cite{valko_sine_beta_2017}.

\paragraph{Related work on \(\Airy\).}

In \cite{lambert_strong_2021}, the authors give an analytic construction of the \emph{stochastic Airy function} \(\SAi_\lambda(t)\), which is a random entire function of $\lambda \in \mathbb{C}$ satisfying \(\Airyop \SAi_\lambda(t) = - \lambda \SAi_\lambda(t)\) in $L^2(0,\infty)$ as a function of $t$ with explicit almost sure $t \to \infty$ asymptotics. The Weyl--Titchmarsh function of Theorem \ref{thm.WTconvinlaw} admits a representation in terms of this function, namely 
\beq{eq.WTSAi}
    \sAiryWT(z) = \frac{\SAi'_{\lambda_{\shift}(z)}(0)}{\sqrt{\shift} \SAi_{\lambda_{\shift}(z)}(0)}
    \qquadtext{with}
    \lambda_{\shift}(z) = - E - \frac{z}{2\sqrt{\shift}},
\eeq
where the prime denotes a derivative with respect to time. Similarly, the Weyl--Titchmarsh function of the stochastic sine system can be written in terms of a random entire function: the stochastic zeta function introduced in \cite{valko_many_2022}. Hence, Theorem~\ref{thm.WTconvinlaw} gives the following.

\begin{corollary}
\label{cor.SAizeta}
\setshift{E_n}
Let \(\zeta_\beta\) denote the stochastic zeta function of~\cite{valko_many_2022}. There is another random entire function \(\xi\) such that for any diverging sequence \(\{\shift\}_{n\in\mathbb{N}} \subset (0,\infty)\) and any \(z \in \mathbb{C} \setminus \mathbb{R}\),
\[
\frac{\SAi'_{\lambda_{\shift}(z)}(0)}{\sqrt{\shift} \SAi_{\lambda_{\shift}(z)}(0)}
\lawto[n\to\infty]
\frac{\xi(-z)}{\zeta_\beta(-z)}.
\]
Moreover, the vector \((\xi, \zeta_\beta)\) can be written in terms of a standard Cauchy random variable and an independent solution to the eigenvalue equation for a differential operator \(\tau_\beta\) that is orthogonally equivalent to the stochastic sine operator.
\end{corollary}

\noindent
From Theorem~\ref{thm.WTconvinlaw}, the proof of this result only amounts to unraveling definitions, so we skip it in the main text. For the sake of completeness, we still provide it in Appendix~\ref{apx.proofSAizeta}.

In \cite[Theorem 11]{ashbury-bridgwood_random_2022}, an independent construction of the stochastic Airy function is given via canonical systems. They define the regularized Fredholm determinant for $z\in \UHP$,
\[
\mathfrak{p}(z) \defeq {\det}_2\bigl(\operatorname{Id} - z\Airyop^{-1} \bigr) 
= \lim_{t\to\infty} e^{z \mathcal{T}(t)} \bigl(T_{H_\beta}(t,z)\bigr)_{11}
\]
where $T_{H_\beta}$ is the transfer matrix and $\mathcal{T}(t)$ is the trace of $\Airyop^{-1}$ restricted to $[0,t]$ (c.f.\ \eqref{eq.SLresolvent}, which gives an integral representation for $\Airyop^{-1}$).  The function $\mathfrak{p}(z)$ can also be related explicitly to $\SAi_z(0)$.

Finally, we comment that there is an independent approach of \cite{gorin_stochastic_2018} and \cite{gaudreau_lamarre_edge_2019} which constructs the semigroup of the stochastic Airy operator $\exp(T \Airyop)$.  We do not know if there is an explicit relationship between this semigroup and the canonical systems theory.

\paragraph{Remarks on the point process convergence.}

Theorem~\ref{thm.WTconvinlaw} proves that the spectral measure of the shifted and scaled Airy system converges to that of the sine system in law, with respect to the vague topology of measures. These spectral measures are pure point with positive masses precisely at the points of \(-2\sqrt{\shift}(\Airy + \shift)\) and of \(-\sine\). Their vague convergence is not strong enough to imply the vague convergence of the point processes, as a priori some spectral masses could vanish and masses could merge in the limit of \(\shift \to \infty\). 

Nevertheless, as pointed out in Remark~\ref{rmk.WTconvinlaw}, in this case the convergence of spectral measures does imply the point process convergence, because the spectral weights have a very specific structure: they are independent of the eigenvalues and are iid Gamma random variables with shape parameter \(\nicefrac{\beta}{2}\) and mean \(2\), as proven in \cite[Proposition~3]{valko_palm_2023} for the sine system, and here in Corollary~\ref{cor.spectralweights} for the Airy system. However, in both cases, the distribution of the spectral weights is deduced by comparison with random matrix models whose spectra converge to those of the operators, so our proof of the point process convergence is almost intrinsic, but not quite. It would be valuable to have a direct proof that the spectral weights are independent of the eigenvalues from the definitions of the canonical systems.

We also note that the point process convergence of \(-2\sqrt{\shift}(\Airy + \shift)\) to \(-\sine\) is already known: it was proven by Valkó and Virág in \cite[Corollary~3]{valko_continuum_2009} from a comparison with the Gaussian \(\beta\)-ensemble. We expect that it could also be proven without appealing to a random matrix model, but still without really using the operator-theoretic framework, by relating the counting functions of the point processes to the operators' phase functions, in a similar way as what was done in ~\cite[Theorem~1.1]{holcomb_random_2018} to describe the transition from the hard edge to the bulk.

\subsection*{Outline of proof and organization of the paper}

\paragraph{Proof strategy.}

We now give a brief overview of the method that we use to prove our main results.

The canonical system for the shifted and scaled stochastic Airy operator \(\sAiryop = 2\sqrt{\shift} (\Airyop - \shift)\) is naturally defined on \((0,\infty)\), whereas the one for the stochastic sine operator is defined on \((0,1)\). Hence, the first step is to introduce a time change to have the Airy system acting on \((0,1)\) as well; this will have the form of a \(\mathscr{C}^1\) bijection \(\timechange\colon [0,1+\varepsilon_{\shift}) \to [0,\infty)\) where \(\varepsilon_{\shift} \downarrow 0\). The role of this bijection is actually twofold: in addition to adapting the time domain of the Airy system, it modifies the coefficient matrix~\eqref{eq.sAirymat} of the Airy system in an essential way to set it up for convergence to the sine system. Intuitively, it counteracts the decay of the solutions that appear in the coefficient matrix to convert their \enquote{Airy-like} oscillations into \enquote{sine-like} oscillations. 

To prove the vague convergence in Theorem~\ref{thm.CSconvinlaw}, we must then analyse the behavior of the solutions that appear in the coefficient matrix~\eqref{eq.sAirymat} of the Airy system. To do so, in Proposition~\ref{prop.polarcoords} we change variables and write the (time-changed) solutions using pairs of polar coordinates. This allows us to write the coefficient matrix in two terms---which is done in~\eqref{eq.sAirymatpolar}---and to effectively decouple a highly oscillatory part that will vanish in the vague limit from an \enquote{average} part, which is the one that ends up converging to the sine system's coefficient matrix. This also allows us to identify the process (written in terms of our polar coordinates) that converges to the hyperbolic Brownian motion from the sine system's coefficient matrix, as well as the process that converges to its driving complex Brownian motion.

Now, to prove the vague convergence of the canonical systems on \([0,1)\), it would suffice to analyse the solutions on compact time intervals in \([0,1)\). Although this approach could be simple and efficient, in order to prove Theorem~\ref{thm.CSconvinlaw} and then extend the vague convergence to a spectral convergence, we actually need a stronger result: the convergence must hold on the full interval \([0,1]\) when \(\beta > 2\). For this reason, we rather analyse the solutions on \([0, 1 - \shift^{\nicefrac{-1}{2}+\alpha}]\) for an \(\alpha \in (0, \nicefrac{1}{2})\), which corresponds when reversing the logarithmic time scaling of the sine system to \([0, (\frac{1}{2}-\alpha) \log\shift]\). Because these intervals grow unboundedly with \(\shift\), our strategy is to build (in Lemma~\ref{lem.probspace}) a coupling between the Brownian motions that drive the Airy and sine systems. This allows us to compare, pathwise, the hyperbolic Brownian motion of the sine system with the process that converges to it. This convergence is proven in Propositions~\ref{prop.GBM} and~\ref{prop.ReHBM}, with explicit rates that hold uniformly on intervals \([0, 1 - \shift^{\nicefrac{-1}{2}+\alpha}]\). From this, we deduce the desired vague convergence in Theorem~\ref{thm.Airytosine.vaguely}, which implies Theorem~\ref{thm.CSconvinlaw}.
The convergence of the canonical systems' transfer matrices is then essentially a consequence of results from general canonical systems theory. We prove this convergence in probability, under the coupling mentioned above, in Theorem~\ref{thm.Airytosine.TMconvergence}, which implies Corollary~\ref{cor.TMconvinlaw}. 

The last part is to prove the convergence of the Weyl--Titchmarsh functions announced in Theorem~\ref{thm.WTconvinlaw}. The proof is split in two cases: \(\beta \leq 2\) and \(\beta > 2\). When \(\beta \leq 2\), the sine system is limit point at \(1\), and for that reason the convergence of the Weyl--Titchmarsh functions is basically a consequence of the convergence of the transfer matrices by general theory of canonical systems (we will see this in Theorem~\ref{thm.TMtoWT.LP}). There is more work left to do when \(\beta > 2\), as the sine system is limit circle at \(1\). In that case, what is missing is essentially the convergence of the boundary conditions at the right endpoint, with the integrability condition playing the role of the boundary condition for the Airy system. As we will see in detail in Section~\ref{sec.WTconv} (which begins with a detailed sketch), the proof in that case mostly boils down to understanding the asymptotic behavior of solutions to \(\Airyop f = 0\) towards \(-\infty\). Hence, we first derive these asymptotics and prove Theorem~\ref{thm.asymptoticsSAi}, and from this we prove the convergence in probability (still under our coupling) of the Weyl--Titchmarsh functions in Theorem~\ref{thm.Airytosine.WTconvergence}, which implies Theorem~\ref{thm.WTconvinlaw}.

\paragraph{Organization of the paper.}

The rest of the paper is organized as follows. In Section~\ref{sec.CS}, we provide a short introduction to the classical theory of canonical systems, and we give more details on topologies on canonical systems and associated objects. In Section~\ref{sec.AirysineCS}, we review some properties of the stochastic Airy operator, and we build its canonical system version. We also construct the time change \(\timechange\) and derive the polar coordinates mentioned above. At the end of the section, having these tools in hand, we also describe more precisely the idea behind the vague convergence of the canonical systems. We start the proof in Section~\ref{sec.coupling}, in which we build the announced coupling between the Airy and sine systems. We continue in Section~\ref{sec.asymptotics}, where we derive the asymptotic behavior of processes that appear in the entries of the shifted and scaled Airy system's coefficient matrix. From these results, we conclude the proof of Theorem~\ref{thm.CSconvinlaw} in Section~\ref{sec.CSconv}. Finally, we extend this result to the convergence of Weyl--Titchmarsh functions in Section~\ref{sec.WTconv}, in which we prove Theorems~\ref{thm.WTconvinlaw} and~\ref{thm.asymptoticsSAi}. 

Some of our proofs use martingale concentration inequalities that we could not find in the literature, and their statements and proofs are included in Appendix~\ref{apx.concentration}. There are three more appendices: Appendix~\ref{apx.resolvents}, which contains some basic theory about canonical system's resolvents, Appendix~\ref{apx.proofSAizeta}, which contains the proof of Corollary~\ref{cor.SAizeta}, and Appendix~\ref{apx.spectralweights}, in which we study the distribution of the spectral weights of the shifted Airy operator's spectral measure.

\paragraph{Notation and conventions.}

We denote by \(\UHP \defeq \{z \in \mathbb{C} : \Im z > 0\}\) the upper half-plane, and by \(\clUHP\) its closure in the Riemann sphere, which we denote by \(\RiemannSphere\). We write \(\Hol(\UHP, \clUHP)\) for the set of generalized Herglotz functions, that is, holomorphic functions \(\UHP \to \clUHP\) (which are just holomorphic functions \(\UHP \to \UHP\) along with extended real constants). We also denote by \(\projection\colon \mathbb{C}^2 \setminus \{0\} \to \RiemannSphere\) the map that sends \(\begin{smallpmatrix} z \\ w \end{smallpmatrix}\) to \(\nicefrac{z}{w}\). 

We usually write points of the unit circle \(\mathbb{S}^1\) as \(e_\theta \defeq (\cos\theta, \sin\theta)\) for an angle \(\theta \in [0,2\pi)\). In particular, \(e_0 \defeq (1,0)\).

If \(f, g \colon \mathbb{R} \to \mathbb{R}\), we use the notation \(f \lesssim g\) to mean that \(f(t) = O\bigl( g(t) \bigr)\) as \(t\to\infty\). 

Given metric spaces \(X\) and \(Y\), we denote by \(\mathscr{C}(X, Y)\) the space of continuous functions \(X \to Y\), and by \(\mathscr{C}_c(X, Y)\) the space of those functions that are also compactly supported.

We denote by \(\AC\loc(a,b)\) the set of locally absolutely continuous functions on the interval \((a,b)\), that is, functions \(f\colon (a,b) \to \mathbb{C}\) such that for some \(t_0 \in (a,b)\), \(f(t) = f(t_0) + \int_{t_0}^t g(t) \diff{t}\) for a \(g \in L^1\loc(a,b)\). If we write \(u \in \AC\loc(a,b)\) and \(u\) is a vector function, we mean that \(u \in \AC\loc(a,b)\) entrywise.

Finally, we point out that we use the convention that a standard complex Brownian motion \(W\) has independent standard real Brownian motions as its real and imaginary parts, so that in particular \(\expect\abs{W(t)}^2 = 2t\) for all \(t > 0\).

\paragraph{Acknowledgements.}

Most of this work was done when V.P.\@ was a PhD student at McGill University, supported by an NSERC CGS-D scholarship and an FRQ doctoral scholarship (doi:~\href{https://doi.org/10.69777/319962}{\texttt{10.69777/319962}}). E.P.\@ is supported by an NSERC Discovery grant. Part of this work was conducted while the authors were in residence at Institut Mittag-Leffler in Djursholm, Sweden during Fall 2024, and so we acknowledge support from the Swedish Research Council under grant no.\ 2021-06594. The authors would also like to thank Laure Dumaz, Cyril Labbé, and Gaultier Lambert for helpful conversations, as well as two anonymous referees for helpful comments on an earlier version of this paper.

\section{Survey of canonical systems theory}
\label{sec.CS}

In this section, we introduce the tools from canonical systems theory that we will need. We start in subsection~\ref{sec.CS.classical} with the basic definitions, and we continue in subsection~\ref{sec.CS.topologies} with the analysis of sequences of canonical systems. We briefly describe how these results can be applied to random canonical systems in subsection~\ref{sec.CS.random}. Finally, we show in subsection~\ref{sec.CS.SL} how to build a canonical system that is equivalent to a (generalized) Sturm--Liouville operator. 

Most of the classical theory of canonical systems is due to de Branges~\cite{de_branges_hilbert_1968, de_branges_hilbert_1960, de_branges_hilbert_1961, de_branges_hilbert_1961-1, de_branges_hilbert_1962}. The presentation we give here is essentially based on Remling's book~\cite{remling_spectral_2018}, except for subsection~\ref{sec.CS.random}, the content of which we could not find in the literature. For a more extensive introduction to canonical systems theory in the same spirit as that given here, we also refer the interested reader to the first author's PhD thesis~\cite[Chapter~2]{painchaud_canonical_2026}.

\subsection{Canonical systems and their spectral theory}
\label{sec.CS.classical}

\subsubsection{Canonical systems}

Recall from the introduction the definition of a canonical system. 

\let\oldthedefinition\thedefinition
\def\thedefinition{\ref{def.CS}}
\begin{definition}
A \emph{canonical system} on an interval \((a,b) \subset \mathbb{R}\) is a differential equation of the form
\[
Ju' = -zHu \qquadtext{with} J \defeq \begin{pmatrix} 0 & -1 \\ 1 & 0 \end{pmatrix},
\]
where \(H\colon (a,b) \to \mathbb{C}^{2\times 2}\) is called the \emph{coefficient matrix}, \(z \in \mathbb{C}\) and \(u\colon (a,b) \to \mathbb{C}^2\).
\end{definition}
\def\thedefinition{\oldthedefinition}
\addtocounter{theorem}{-1}

Throughout, we always suppose that coefficient matrices satisfy the following.

\begin{hypothesis}
\label{hyp.CS}
The coefficient matrix \(H\) of a canonical system has the following properties.
\begin{enumerate}
    \item \(H \colon (a,b) \to \mathbb{R}^{2\times 2}\),
    \item \(H \in L^1\loc(a,b)\) entrywise,
    \item \(H(t) \succeq 0\) for a.e.\@ \(t \in (a,b)\),
    \item \(H(t) \neq 0\) for a.e.\@ \(t \in (a,b)\).
\end{enumerate}
\end{hypothesis}

As coefficient matrices do not have to be very regular, solutions to canonical systems are understood in the following weak sense: \(u\colon (a,b) \to \mathbb{C}^2\) is said to solve the system on \((a,b)\) provided \(u \in \AC\loc(a,b)\) and the equality \(Ju' = -zHu\) holds a.e.\@ on \((a,b)\). Moreover, by the theory of ordinary differential equations, the initial value problem for a canonical system, that is, the problem
\[
Ju' = -zHu
\qquadtext{with}
u(t_0) = v
\]
for some \(t_0 \in (a,b)\) and \(v \in \mathbb{C}^2\), always has a unique solution (see e.g.~\cite[Theorem~1.1]{remling_spectral_2018}). 

An important property of canonical systems is that they can easily be modified to change their domain. Indeed, the following result can be verified by a direct calculation (see also~\cite[Theorem~1.5]{remling_spectral_2018}).

\begin{proposition}
\label{prop.timechange}
Let \((a,b)\) and \((A,B)\) be two real intervals, and let \(\eta\colon (A,B) \to (a,b)\) be an absolutely continuous bijection. Then a canonical system \(Ju' = -zHu\) on \((a,b)\) is equivalent to the canonical system \(Jv' = -z\tilde{H}v\) on \((A,B)\) for \(\tilde{H} \defeq \eta' (H\circ\eta)\), in the sense that \(u\colon (a,b) \to \mathbb{C}^2\) solves \(Ju' = -zHu\) if and only if \(v\defeq u\circ\eta\) solves \(Jv' = -z\tilde{H}v\).
\end{proposition}

\subsubsection{Self-adjoint realizations of the maximal relation}

We now turn to the spectral theory of canonical systems. In a similar fashion as for a differential operator, the analysis of a canonical system takes place on an approriate Hilbert space. Given a measurable \(u\colon (a,b) \to \mathbb{C}^2\), define
\[
\norm{u}_H^2 \defeq \int_a^b u^*(t) H(t) u(t) \diff{t}.
\]
This seminorm induces the separable infinite-dimensional Hilbert space
\[
L_H^2(a,b) \defeq \mathcal{L} \bigm/ \{u\in\mathcal{L} : \norm{u}_H = 0\}
\qquadtext{where}
\mathcal{L} \defeq \bigl\{ u\colon (a,b) \to \mathbb{C}^2 : u \text{ is measurable}, \norm{u}_H < \infty \bigr\}.
\]
This is entirely analogous to the usual construction of an \(L^2\) space, except for the fact that when \(H\) is not invertible, it can be possible for an element of \(L_H^2(a,b)\) to have representatives which are distinct continuous functions. 

Now, a canonical system should be understood as the eigenvalue equation of a differential operator in a generalized sense. In general, what plays the role of such an operator is the \emph{maximal relation}
\[
\mathcal{T}_H \defeq \bigl\{ (u,v) \in L_H^2(a,b) \oplus L_H^2(a,b) : u \text{ has a representative } \hat{u} \in \AC\loc(a,b) \text{ with } J\hat{u}' = -Hv \text{ a.e.} \bigr\},
\]
where by a \emph{relation} on a Hilbert space \(\mathscr{H}\) we simply mean a linear subspace of \(\mathscr{H} \oplus \mathscr{H}\). If the coefficient matrix \(H\) is invertible everywhere on \((a,b)\), then the maximal relation is simply the graph of \(u \mapsto -H^{-1}Ju'\), which is a differential operator on \(L_H^2(a,b)\) whose eigenvalue equation is precisely the canonical system. When \(H\) is not invertible, the maximal relation can still be the graph of an operator---as we will see in Section~\ref{sec.CS.SL}, this is the case in particular for canonical systems that are equivalent to Sturm--Liouville operators. In full generality, the maximal relation can be thought of as the graph of a multi-valued operator (although we won't be concerned with this case in this paper). One of the appeals of the theory of canonical systems is that it gives a common framework for studying their spectral theory through the maximal relations, regardless of the invertibility of the coefficient matrices.

The notion of self-adjointness can easily be extended to relations by reproducing what happens to the graph of an operator when taking its adjoint. It is then possible to define the spectrum of a self-adjoint relation in such a way that when the relation is the graph of an operator, then the spectrum of the relation coincides with that of the underlying operator (see~\cite[Section~3.1]{remling_spectral_2018} for more details). 

Much like what can be done with second-order differential operators, it is possible to restrict a maximal relation and make it self-adjoint through boundary conditions, depending on the behavior of the system near the endpoints. The endpoints \(a\) and \(b\) of a canonical system's domain can be of either of two types (the terminology here stems from Weyl theory).

\begin{definition}
An endpoint is said to be \emph{limit circle} if \(H\) is integrable near that endpoint. Otherwise, it is said to be \emph{limit point}.
\end{definition}

The self-adjoint restrictions of the maximal relation can then be described as follows.

\begin{theorem}[Theorem~2.25 of~\cite{remling_spectral_2018}]
\label{thm.CS.selfadjointrealizations}
Let \(\mathcal{T}_H\) be the maximal relation of a canonical system on \((a,b)\).
\begin{enumerate}
    \item If both endpoints are limit point, then \(\mathcal{T}_H\) is self-adjoint.
    \item If \(a\) is limit circle and \(b\) is limit point, then for each \(\theta \in [0,\pi)\), with \(e_\theta \defeq (\cos\theta, \sin\theta)\), the restriction of \(\mathcal{T}_H\) by the boundary condition \(e_\theta^* Ju(a) = 0\) is self-adjoint. 
    \item If \(a\) and \(b\) are both limit circle, then for \(\theta, \phi \in [0,\pi)\), the restriction of \(\mathcal{T}_H\) by the boundary conditions \(e_\theta^* Ju(a) = e_\phi^*Ju(b) = 0\) is self-adjoint.
\end{enumerate}
\end{theorem}

\subsubsection{Transfer matrices and Weyl--Titchmarsh functions}
\label{sec.Weyl}

Transfer matrices and Weyl--Titchmarsh functions are useful tools to describe solutions of canonical systems and their spectra. We will only define these objects under an additional hypothesis, as their definitions become simpler in that case and this will suffice for our purposes.

\begin{hypothesis}
\label{hyp.LCata}
The system is limit circle at \(a\), and subject to the boundary condition \(e_0^*Ju(a) = 0\).
\end{hypothesis}

One advantage of this is that the solutions of a system that is limit circle at an endpoint can be continuously extended to that endpoint (see e.g.~\cite[Lemma~2.6]{remling_spectral_2018}). This allows the following definition.

\begin{definition}\label{def.T}
The \emph{transfer matrix} of a canonical system satisfying Hypothesis~\ref{hyp.LCata} is the function \(T_H \colon (a,b) \times \mathbb{C} \to \mathbb{C}^{2\times 2}\) such that for each \(z \in \mathbb{C}\), the function \(T = T_H(\cdot, z)\) is the (matrix) solution to \(JT' = -zHT\) with \(T(a) = I_2\), the identity matrix.
\end{definition}

The Weyl--Titchmarsh function is defined in two different ways, depending on the behavior of the system at the right endpoint. 

\begin{definition}
\label{def.WT}
Suppose that Hypothesis~\ref{hyp.LCata} holds and that \(b\) is limit circle. Fix \(\theta \in [0,\pi)\), and for each \(z \in \UHP\), let \(u_z \colon [a,b] \to \mathbb{C}^2\) denote a non-trivial solution of the canonical system satisfying the boundary conditions \(e_0^*Ju_z(a) = e_\theta^*Ju_z(b) = 0\). The \emph{Weyl--Titchmarsh function} for this problem is the map \(m^\theta\colon \UHP \to \clUHP\) defined by \(m^\theta(z) \defeq \projection u_z(a)\) where \(\projection \begin{smallpmatrix} z \\ w \end{smallpmatrix} \defeq \nicefrac{z}{w}\). 
\end{definition}

\begin{remark}
By definition of the transfer matrix, \(m^\theta(z) = \projection T_H(b,z)^{-1}e_\theta\).
\end{remark}

\begin{definition}
Suppose that Hypothesis~\ref{hyp.LCata} holds and that \(b\) is limit point. For each \(z \in \UHP\), let \(u_z\colon [a,b) \to \mathbb{C}^2\) be a non-trivial \(L_H^2(a,b)\) solution of the canonical system. The \emph{Weyl--Titchmarsh function} for this problem is the map \(m\colon \UHP \to \clUHP\) defined by \(m(z) \defeq \projection u_z(a)\).
\end{definition}

\begin{remark}
For each \(z\), the solution \(u_z\) is only defined up to a multiplicative constant, but the Weyl--Titchmarsh function is still well defined since this constant cancels out when taking the ratio of the coordinates.
\end{remark}

In both cases, it can be shown that the Weyl--Titchmarsh function is a \emph{generalized Herglotz function}, that is, a holomorphic function \(\UHP \to \clUHP\) (see~\cite[Theorems~3.10 and 3.15]{remling_spectral_2018}). Note that by the open mapping theorem, a generalized Herglotz function is either a holomorphic function \(\UHP \to \UHP\) (which makes it a genuine Herglotz function), or an extended real constant. Moreover, by definition, a Weyl--Titchmarsh function blows up when the solution it is defined from satisfies the boundary condition at the left endpoint, and therefore it has poles exactly at the system's eigenvalues. This link between the Weyl--Titchmarsh function and the spectrum is completed by the following classical theorem, which allows one to extract a spectral measure from the Weyl--Titchmarsh function.

\begin{theorem}[Herglotz representation theorem]
\label{thm.Herglotzrepresentation}
Let \(m\) be a Herglotz function. Then there are \(a \in \mathbb{R}\), \(b \geq 0\) and a positive regular Borel measure \(\mu\) on \(\mathbb{R}\) with \(\int_\mathbb{R} \frac{1}{1+t^2} \diff{\mu}(t) < \infty\) such that
\[
m(z) = a + bz + \int_\mathbb{R} \Bigl( \frac{1}{t-z} - \frac{t}{1+t^2} \Bigr) \diff{\mu}(t),
\]
and \(a\), \(b\) and \(\mu\) are uniquely determined by \(m\). Conversely, any \(m\) defined as such is a Herglotz function.
\end{theorem}

In particular, the spectral measure \(\mu\) can be recovered from the Weyl--Titchmarsh function through Stieltjes inversion: its absolutely continuous part \(\mu_{\mathrm{ac}}\) satisfies \(\diff{\mu_{\mathrm{ac}}}(t) = \frac{1}{\pi} \Im m(t) \diff{t}\), and its singular part can be recovered from the fact that for all \(t\in\mathbb{R}\), \(\mu\bigl( \{t\} \bigr) = -i \lim_{\varepsilon\downarrow 0} \varepsilon m(t + i\varepsilon)\). We refer to~\cite[Appendix~F]{schmudgen_unbounded_2012} for more details on Herglotz functions.

\subsection{Convergence of sequences of canonical systems}
\label{sec.CS.topologies}

We call the \emph{vague topology} on the set of locally finite signed measures on an interval \(\mathcal{I}\) the topology generated by all maps \(\mu \mapsto \int_\mathcal{I} \varphi\diff{\mu}\) for \(\varphi \in \mathscr{C}_c(\mathcal{I},\mathbb{C})\). From this, we can define the following.

\begin{definition}
Let \(\CS\mathcal{I}\) denote the set of coefficient matrices of canonical systems defined on \(\mathcal{I}\), and for which the endpoints included in \(\mathcal{I}\) (if any) are limit circle. We call the \emph{vague topology} on \(\CS\mathcal{I}\) the topology obtained by identifying the entries of coefficient matrices with signed measures on \(\mathcal{I}\), and using the vague topology.
\end{definition}

This vague convergence on \(\CS\mathcal{I}\) can be described using the pseudometrics
\[
d_\varphi(H, H') \defeq \abs[\Big]{\int_\mathcal{I} \varphi^*(t) \bigl( H(t) - H'(t) \bigr) \varphi(t) \diff{t}}
\]
for \(\varphi \in \mathscr{C}_c(\mathcal{I}, \mathbb{C}^2)\). Indeed, since coefficient matrices are symmetric, it is easy to check that \(H_n \to H\) vaguely on \(\mathcal{I}\) if and only if \(d_\varphi(H_n,H) \to 0\) for all \(\varphi \in \mathscr{C}_c(\mathcal{I}, \mathbb{C}^2)\). Just like in more familiar cases of weak and vague topologies, the above pseudometrics can be combined to build a metric on \(\CS\mathcal{I}\), and it can be proven from the Stone--Weierstrass theorem that the resulting space is separable. This results in the following (see~\cite[Theorem~2.11]{painchaud_canonical_2026} for a full proof). 

\begin{theorem}
\label{thm.CSmetric}
Let \(\{\varphi_k\}_{k\in\mathbb{N}} \subset \mathscr{C}_c(\mathcal{I},\mathbb{C}^2)\) be a countable collection of functions that is dense with respect to the topology of compact convergence. Let \(d\colon \CS\mathcal{I}\times\CS\mathcal{I} \to [0,\infty)\) be defined by
\[
d(H,H') \defeq \sum_{k=1}^\infty \frac{1}{2^k} \frac{d_{\varphi_k}(H,H')}{1 + d_{\varphi_k}(H,H')}.
\]
Then \(d\) is a metric on \(\CS\mathcal{I}\) that induces the vague topology, and \((\CS\mathcal{I}, d)\) is separable.
\end{theorem}

The vague topology is strong enough to keep track of the solutions of canonical systems provided mild conditions on their coefficient matrices are met. In fact, it can be shown that the map from coefficient to transfer matrices is continuous on suitable domains.

To simplify the notation in what follows, we denote by \(\TM\mathcal{I}\) the space of transfer matrices of canonical systems that satisfy Hypothesis~\ref{hyp.LCata}. As a space of continuous functions between hemicompact, locally compact metric spaces, \(\TM\mathcal{I}\) is a separable metric spaces (see e.g.~\cite[Example~2.2]{conway_course_2007} and~\cite[Theorem~XII.5.2]{dugundji_topology_1966}). 

\begin{theorem}
\label{thm.CStoTM}
Let \(\mathcal{I} = [a,b)\) or \(\mathcal{I} = [a,b]\). Given \(f \in L^1\loc(\mathcal{I})\), let
\(
\CS_f\mathcal{I} \defeq \bigl\{ H \in \CS\mathcal{I} : \tr H < f \text{ on } (a,b) \bigr\}.
\)
Then the map \(\CS_f\mathcal{I} \to \TM\mathcal{I}\) sending a coefficient matrix to the corresponding transfer matrix is continuous.
\end{theorem}

\noindent
This result can be proven by adapting the proof of~\cite[Theorem~5.7(a)]{remling_spectral_2018}, which shows that the same map is continuous on the domain \(\{H \in \CS[0,b] : \tr H \equiv 1\}\) for any \(b < \infty\). See~\cite[Theorem~2.12]{painchaud_canonical_2026} for a full proof.

Finally, the convergence of transfer matrices can be related to the convergence of Weyl--Titchmarsh functions. These belong to the space \(\Hol(\UHP,\clUHP)\) of generalized Herglotz functions, which is also a separable metric space under the topology of compact convergence. Moreover, as \(\clUHP\) is compact, the space \(\Hol(\UHP,\clUHP)\) is itself compact~\cite[Theorem~5.6(b)]{remling_spectral_2018}. Note that the Herglotz representation is continuous (see e.g.~\cite[Theorem~7.3(a)]{remling_spectral_2018}), so the convergence of Weyl--Titchmarsh functions always implies the vague convergence of the underlying spectral measures.

As when we introduced Weyl--Titchmarsh functions, we only consider systems which satisfy Hypothesis~\ref{hyp.LCata}. The convergence of Weyl--Titchmarsh functions can then essentially be split into two cases, depending on the behavior of the limit system at the right endpoint, which leads to fundamentally different behavior. We thus partition \(\TM\mathcal{I}\) into its subsets \(\TMLP\mathcal{I}\) and \(\TMLC\mathcal{I}\) of transfer matrices of systems which are respectively limit point and limit circle at \(b\). 

The simplest case is the limit point case. The following result can be proven using the same strategy as in~\cite[Theorem~5.7(b)]{remling_spectral_2018}, which restricts to systems on \([0,\infty)\) whose coefficient matrices have trace \(1\).

\begin{theorem}
\label{thm.TMtoWT.LP}
Let \(\mathcal{I} = [a,b)\) or \([a,b]\). Let \(T_n \in \TM\mathcal{I}\) and \(T \in \TMLP\mathcal{I}\) be transfer matrices with corresponding Weyl--Titchmarsh functions \(m_n,m \in \Hol(\UHP,\clUHP)\). If \(T_n \to T\) compactly, then \(m_n \to m\) compactly. In particular, the map \(\TMLP[a,b) \to \Hol(\UHP,\clUHP)\) sending a transfer matrix to the corresponding Weyl--Titchmarsh function is continuous.
\end{theorem}

In the limit circle case, however, the convergence of the transfer matrices is not enough to imply the convergence of the Weyl--Titchmarsh functions, since they not only depend on the transfer matrices, but also on the boundary conditions. Therefore, one could imagine that the analogous result to Theorem~\ref{thm.TMtoWT.LP} in the limit circle case is the convergence of the obvious map \(\TMLC[a,b] \times \mathbb{S}^1 \to \Hol(\UHP, \clUHP)\), where the unit circle \(\mathbb{S}^1\) is taken as the set of boundary conditions. This map does turn out to be continuous, but in fact more is true and we now motivate the extension we will consider.

So far, in our considerations to pass from the convergence of transfer matrices to that of Weyl--Titchmarsh functions, we have left out an important case. Theorem~\ref{thm.TMtoWT.LP} shows how the convergence goes through in any case in which the limit system is limit point at \(b\), while the continuity of the map \(\TMLC[a,b]\times \mathbb{S}^1 \to \Hol(\UHP, \clUHP)\) handles systems that are all limit circle at \(b\). Yet, it is possible for a sequence of systems that are limit point at \(b\) to converge to one that is limit circle at \(b\). In our sense, it seems more appropriate to consider these along with convergent sequences of systems that are limit circle at \(b\) for two reasons. The first one is that the type of result we are looking for is the continuity of a map, as our ultimate goal is to exploit this continuity to pass some probabilistic type of convergence of random transfer matrices through, and thus obtain the convergence of random Weyl--Titchmarsh functions. But it is not obvious (at least to us) to find an appropriate formulation which would include at the same time both types of systems. 

The other reason has to do with the way in which systems that are limit point at \(b\) can converge to one that is limit circle at \(b\). For this to happen, the idea is that the integrability condition enforced by the limit point systems should converge to the boundary condition of the limit circle system. To make things clearer, consider \(T_n \in \TMLP[a,b)\) with \(u_n\colon [a,b) \times \mathbb{C} \to \mathbb{C}^2\) such that \(u_n(\cdot,z)\) is an integrable solution to the corresponding canonical system, and take \(T \in \TMLC[a,b]\) with boundary condition \(e_\theta^*Ju(b) = 0\). Then, given times \(t_n \in [a,b)\), the Weyl--Titchmarsh functions of the systems in the sequence can be written as \(m_n(z) = \projection T_n(t_n, z)^{-1} u_n(t_n,z)\), and that of the limit system is \(m(z) = \projection T(b,z)^{-1} e_\theta\). Hence, to deduce that \(m_n(z) \to m(z)\), it suffices to find \(t_n\)'s such that \(T_n(t_n,z)^{-1} \projection u_n(t_n, z) \to T(b,z)^{-1} \cot\phi\) (of course, it is not obvious that such \(t_n\)'s even exist). If additionally \(t_n \uparrow b\), we can then think of this convergence as that of the systems \emph{restricted} to \([a, t_n]\) (which effectively become limit circle at \(t_n\)) to the limit system, along with the convergence of the \enquote{boundary conditions} at \(t_n\) (this being \(\projection u_n(t_n,z)\), which are not proper boundary conditions since they depend on \(z\)) to \(\cot\theta\). If this works, it is straightforward to stretch the time domains \([a,t_n]\) to \([a,b]\) without changing properties of the restricted systems, and then the problem is reduced to the convergence of systems which are all limit circle at \(b\), only with \enquote{boundary conditions} which depend on \(z\). For a more extensive discussion on this idea, see~\cite[Section~2.2.3]{painchaud_canonical_2026}.

This motivates the following result, which gives the continuity of the map \(\TMLC[a,b] \times \mathbb{S}^1 \to \Hol(\UHP, \clUHP)\) mentioned earlier, but allows for such \(z\)-dependent \enquote{boundary conditions} as well.

\begin{theorem}
\label{thm.TMtoWT.LC}
Define \(M\colon \TMLC[a,b] \times \mathscr{C}(\UHP, \mathbb{C}^2) \to \RiemannSphere^{\UHP}\) by setting the function \(M(T,w) \colon \UHP \to \RiemannSphere\) as \(M(T,w)(z) \defeq \projection T(b,z)^{-1}w(z)\). Then \(M\) is continuous on \(M^{-1}\bigl( \Hol(\UHP,\clUHP) \bigr)\) under the topology of compact convergence.
\end{theorem}

\begin{proof}
To prove the continuity of \(M\) on \(M^{-1}\bigl( \Hol(\UHP, \clUHP) \bigr)\), it suffices to prove that \(M(T_n,w_n) \to M(T,w)\) compactly as \(n\to\infty\) whenever \((T_n,w_n) \to (T,w)\) compactly in \(M^{-1}\bigl( \Hol(\UHP, \clUHP) \bigr)\), and to do so, it suffices to prove that \(M(T_n,w_n)(z) \to M(T,w)(z)\) for every \(z \in \UHP\), as the space of generalized Herglotz functions is sequentially compact. 

Fix \(z \in \UHP\). As \(T_n \to T\) compactly, \(\begin{smallpmatrix} A_n & B_n \\ C_n & D_n \end{smallpmatrix} \defeq T_n(b,z)^{-1} \to T(b,z)^{-1} \eqdef \begin{smallpmatrix} A & B \\ C & D \end{smallpmatrix}\). Now, if both \(\abs{M(T,w)(z)} < \infty\) and \(\abs{\projection w(z)} < \infty\), then \(C\projection w(z) + D\) is bounded away from zero, so \(C_n \projection w_n(z) + D_n\) is bounded away from zero for \(n\) large enough, \rlap{and}
\[
M(T_n,w_n)(z)
    = \frac{A_n\projection w_n(z) + B_n}{C_n \projection w_n(z) + D_n}
    \to \frac{A\projection w(z) + B}{C\projection w(z) + D}
    = M(T,w)(z).
\]
If \(\abs{M(T,w)(z)} < \infty\) but \(\projection w(z) = \infty\), then \(M(T_n,w_n)(z) \to \nicefrac{A}{C} = M(T,w)(z)\). Finally, if \(M(T,w)(z) = \infty\), then \(C\projection w(z) + D = 0\), and it must be that \(\abs{\projection w(z)} < \infty\) and that \(A\projection w(z) + B \neq 0\) because \(\det T(b,z)^{-1} = 1\), so \(M(T_n,w_n)(z) \to \infty\).
\end{proof}

\subsection{Convergence of random canonical systems}
\label{sec.CS.random}

In this section, we use the results from the last section to find useful criteria to describe the convergence in probability of random canonical systems. 

We first point out the following straightforward result.

\begin{proposition}
\label{prop.CS.random.convergence}
Let \(\mathcal{I}\) be a real interval, and let \(H_n, H\) be random coefficient matrices with values in \(\CS\mathcal{I}\). If
\[
\int_\mathcal{I} \varphi^*(t) \bigl( H_n(t) - H(t) \bigr) \varphi(t) \diff{t} \probto[n\to\infty] 0
\]
for all \(\varphi \in \mathscr{C}_c(\mathcal{I},\mathbb{C}^2)\), then \(H_n \to H\) vaguely on \(\mathcal{I}\) in probability.
\end{proposition}

\begin{proof}
Let \(\varepsilon > 0\), and let \(d\) and \(\{\varphi_k\}_{k\in\mathbb{N}}\) be the metric and functions from Theorem~\ref{thm.CSmetric}. If \(d(H_n,H) > \varepsilon\) and \(K\) is large enough so that \(\sum_{k=K+1}^\infty \nicefrac{1}{2^k} < \nicefrac{\varepsilon}{2}\), then there must be a \(k \leq K\) with \(d_{\varphi_k}(H_n,H) > \nicefrac{\varepsilon}{2}\). Therefore, \(\bprob{d(H_n,H) > \varepsilon} \leq \sum_{k=1}^K \bprob{d_{\varphi_k}(H_n,H) > \nicefrac{\varepsilon}{2}}\), and the result follows.
\end{proof}

Now, we have seen in Theorem~\ref{thm.CStoTM} that the map from coefficient to transfer matrices is continuous on domains with dominated trace. When dealing with \emph{random} canonical systems, it can be useful to relax the domination condition to a high probability event. Thus, we extend Theorem~\ref{thm.CStoTM} to the following.

\begin{proposition}
\label{prop.CS.random.TMconvergence}
Let \(\mathcal{I} = [a,b)\) or \([a,b]\). Let \(H_n,H\) be random coefficient matrices with values in \(\CS\mathcal{I}\), and let \(T_{H_n},T_H\) be their transfer matrices. Suppose that for any \(\varepsilon > 0\), there are \(f_\varepsilon, g_\varepsilon \in L^1\loc(\mathcal{I})\) such that \(\bprob{\tr H \leq f_\varepsilon} \geq 1 - \varepsilon\) and \(\bprob{\tr H_n \leq g_\varepsilon} \geq 1 - \varepsilon\) for any \(n\) large enough. If \(H_n \to H\) vaguely on \(\mathcal{I}\) in probability, then \(T_{H_n} \to T_H\) compactly on \(\mathcal{I} \times \mathbb{C}\) in probability.
\end{proposition}

\begin{proof}
Let \(d_{\TM\mathcal{I}}\) denote the metric on \(\TM\mathcal{I}\). If \(\varepsilon, \delta > 0\), by hypothesis there is an \(N\in\mathbb{N}\) such that for \(n \geq N\),
\[
\bprob[\big]{d_{\TM\mathcal{I}}(T_{H_n}, T_H) > \delta}
    \leq \bprob[\big]{d_{\TM\mathcal{I}}(T_{H_n},T_H) > \delta, \tr H \leq f_{\varepsilon}, \tr H_n \leq g_{\varepsilon}} + 2\varepsilon.
\]
Now, pick some \(\tilde{H} \in \CS\mathcal{I}\) such that \(\tr\tilde{H} \leq f_{\varepsilon} + g_{\varepsilon}\), and define
\[
H' \defeq H \charf{\{\tr H \leq f_{\varepsilon}\}} + \tilde{H} \charf{\{\tr H > f_{\varepsilon}\}}
\qquadtext{and}
H_n' \defeq H_n \charf{\{\tr H_n \leq g_{\varepsilon}\}} + \tilde{H} \charf{\{\tr H_n > g_{\varepsilon}\}}.
\]
Then \(H_n' \to H'\) vaguely on \(\mathcal{I}\) in probability, and they are \(\CS_{f_\varepsilon+g_\varepsilon}\mathcal{I}\)-valued. Hence, Theorem~\ref{thm.CStoTM} guarantees that their transfer matrices \(T_{H_n'} \to T_{H'}\) compactly in probability, so \(\bprob{d_{\TM\mathcal{I}}(T_{H_n'}, T_{H'}) > \delta} < \varepsilon\) for all \(n\) large enough. Since by definition \(T_{H'} = T_H\) on \(\{\tr H \leq f_\varepsilon\}\) and \(T_{H_n'} = T_{H_n}\) on \(\{\tr H_n \leq g_\varepsilon\}\), it follows that \(\bprob{d_{\TM\mathcal{I}}(T_{H_n}, T_H) > \delta} < 3\varepsilon\) for all \(n\) large enough, and \(T_{H_n} \to T_H\) compactly in probability.  
\end{proof}

\subsection{Generalized Sturm--Liouville operators as canonical systems}
\label{sec.CS.SL}

By a generalized Sturm--Liouville operator on some interval \((a,b) \subseteq \mathbb{R}\), we mean a second-order differential operator
\beq{eq.genSL}
L \colon f \mapsto \frac{1}{w} \Bigl( -\bigl( \quasi{f} \bigr)' + s\quasi{f} + qf \Bigr)
\qquadtext{with}
\quasi{f} \defeq p(f' + sf),
\eeq
where \(p,q,s,w \colon (a,b) \to \mathbb{R}\) satisfy \(p\neq 0\) and \(w > 0\) a.e.\@ on \((a,b)\) and \(p^{-1},q,s,w \in L^1\loc(a,b)\), and \(\quasi{f}\) is called the \emph{(first) quasi-derivative} of \(f\). Here, \(L\) is acting on functions from
\(
\mathfrak{D}_L \defeq \bigl\{ f \in \AC\loc(a,b) : \quasi{f} \in \AC\loc(a,b) \bigr\}.
\)
This class of generalized Sturm--Liouville operators is studied in detail by Eckhardt et al.\@ in~\cite{eckhardt_weyl-titchmarsh_2013}. The point of extending classical Sturm--Liouville operators to operators of this form is that here \(L\) is allowed to have a distributional potential, as \(s\) does not have to be differentiable. Nevertheless, from the classical theory of ordinary differential equations, it can be shown that for any \(g \in L^1\loc\bigl( (a,b), w(t) \diff{t} \bigr)\), any \(x, y, z \in \mathbb{C}\) and any \(t_0 \in (a,b)\), the initial value problem
\[
Lf = zf + g,
\qquad
f(t_0) = x, \quad
\quasi{f}(t_0) = y
\]
has a unique solution \(f \in \mathfrak{D}_L\) (see~\cite[Theorem~2.2]{eckhardt_weyl-titchmarsh_2013}).

It is possible to construct a canonical system that is equivalent to the eigenvalue equation of such a generalized Sturm--Liouville operator basically by choosing the right change of variables. This is done through a matrix \(A\colon (a,b) \to \mathbb{R}^{2\times 2}\) that solves, for some \(t_0 \in (a,b)\) and an \(A_0 \in \mathbb{R}^{2\times 2}\) with \(\det A_0 = 1\),
\beq{eq.CS.SL.SLtoCS}
A' = \begin{pmatrix} s & q \\ p^{-1} & -s \end{pmatrix} A
\quadtext{with}
A(t_0) = A_0,
\qquadtext{so that}
A = \begin{pmatrix} \quasi{g} & \quasi{h} \\ g & h \end{pmatrix}
\eeq
for two linearly independent solutions \(g\) and \(h\) to \(Lf = 0\) which are determined by \(A_0\). Here, \(t_0\) could also be chosen to be \(a\) if it is limit circle, and likewise for \(b\). The coefficient matrix of the canonical system for \(L\) is
\beq{eq.CS.SL.coefficientmatrix}
H \defeq wJA^{-1}\begin{pmatrix} 0 & 1 \\ 0 & 0 \end{pmatrix} A
    = w \begin{pmatrix} g^2 & gh \\ gh & h^2 \end{pmatrix},
\eeq
where the equality uses the fact that the Wronskian \(h\quasi{g} - \quasi{h}g\) is constant, meaning that the determinant of \(A\) is equal to \(1\) at all times. The constancy of this Wronskian is straightforward to verify by computing its derivative; it is a simple consequence of the fact that \(g\) and \(h\) solve the eigenvalue equation \(Lf = zf\) for the same \(z\) (here, \(z = 0\)).

To understand the resulting canonical system, it helps to understand first the structure of the space \(L_H^2(a,b)\). To do so, notice that \(J\) has the property that \(M^\transpose J M = J \det M\) for any invertible \(M \in \mathbb{C}^{2\times 2}\). Hence, if \(u\colon (a,b) \to \mathbb{C}^2\), then
\[
\norm{u}_H^2
    = \int_a^b u^* H u \diff{t}
    = \int_a^b u^* A^{\transpose} A^{-\transpose} J A^{-1} \begin{pmatrix} 0 & 1 \\ 0 & 0 \end{pmatrix} A u w \diff{t}
    = \int_a^b \abs{(Au)_2}^2 w \diff{t}.
\]
This means that an element of \(L_H^2(a,b)\) is determined (up to a.e.\@ equality) by an \(f \in L^2\bigl( (a,b), w(t)\diff{t} \bigr)\), and is the class of functions of the form \(A^{-1} \begin{smallpmatrix} * \\ f \end{smallpmatrix}\) where \(*\) stands for any measurable function. This mapping sets up an isometry between the two Hilbert spaces, and thus the operator \(L\) on \(L^2\bigl( (a,b), w(t)\diff{t} \bigr)\) transfers to an operator \(\tilde{L}\) on \(L_H^2(a,b)\) which maps \(A^{-1} \begin{smallpmatrix} * \\ f \end{smallpmatrix}\) to \(A^{-1} \begin{smallpmatrix} * \\ Lf \end{smallpmatrix}\). In particular, provided boundary conditions imposed on the two systems are consistent, this correspondence sets up a bijection betwen the eigenfunctions of \(L\) and \(\tilde{L}\), since the integrability conditions are met exactly at the same time. 

The canonical system defined by~\eqref{eq.CS.SL.coefficientmatrix} is equivalent to \(L\) in the following sense.

\begin{lemma}
\label{lem.CS.SL}
Let \(L\) and \(\tilde{L}\) be as above, and let \(A\) and \(H\) be as in~\eqref{eq.CS.SL.SLtoCS} and~\eqref{eq.CS.SL.coefficientmatrix} respectively. Then the maximal relation of the canonical system with coefficient matrix \(H\) is the graph of \(\tilde{L}\), and an \(f \in \mathfrak{D}_L\) solves the eigenvalue equation \(Lf = zf\) if and only if \(u \defeq A^{-1} \begin{smallpmatrix} \quasi{f} \\ f \end{smallpmatrix}\) solves \(Ju' = -zHu\).
\end{lemma}

\begin{proof}
Let \(u \in L_H^2(a,b)\) have a representative \(\hat{u} = A^{-1} \begin{smallpmatrix} \varphi \\ f \end{smallpmatrix} \in \AC\loc(a,b)\). Then
\[
\begin{pmatrix} \varphi' \\ f' \end{pmatrix}
    = \begin{pmatrix} \varphi' \\ p^{-1}\quasi{f} - sf \end{pmatrix}
    = A'A^{-1} \begin{pmatrix} \quasi{f} \\ f \end{pmatrix} + \begin{pmatrix} \varphi' - s\quasi{f} - qf \\ 0 \end{pmatrix},
\]
so because \((A^{-1})' = -A^{-1}A'A^{-1}\), 
\[
J\hat{u}'
    = -JA^{-1}A'A^{-1} \begin{pmatrix} \varphi \\ f \end{pmatrix} + JA^{-1}A'A^{-1} \begin{pmatrix} \quasi{f} \\ f \end{pmatrix} + JA^{-1} \begin{pmatrix} \varphi' - s\quasi{f} - qf \\ 0 \end{pmatrix} \\
    = - JA^{-1} \begin{pmatrix} s\varphi - \varphi' + qf \\ p^{-1} (\varphi - \quasi{f}) \end{pmatrix}.
\]
Notice that by definition of \(H\), this has the form \(-Hv\) if and only if the second entry of the last factor is zero. In that case, \(\varphi = \quasi{f}\) and the first entry of the vector becomes \(-(\quasi{f})' + s\quasi{f} + qf = wLf\), so \(J\hat{u}' = -Hv\) for \(v = A^{-1} \begin{smallpmatrix} * \\ Lf \end{smallpmatrix}\). Hence, a pair \((u,v) \in L_H^2(a,b)\) is in the maximal relation of \(H\) if and only if there is an \(f \in \mathfrak{D}_L\) such that \(u\) is represented by \(\hat{u} = A^{-1} \begin{smallpmatrix} \quasi{f} \\ f \end{smallpmatrix}\) and \(v = A^{-1} \begin{smallpmatrix} * \\ Lf \end{smallpmatrix}\); in other words, if and only if \(v = \tilde{L}u\) so that \((u,v)\) is in the graph of \(\tilde{L}\). 

Taking \(v = zu\) in the above computation also verifies that \(\hat{u} = A^{-1} \begin{smallpmatrix} \quasi{f} \\ f \end{smallpmatrix}\) solves \(J\hat{u}' = -zHu\) if and only if \(\tilde{L}u = zu\), which by definition of \(L\) is equivalent to \(Lf = zf\). 
\end{proof}

\section{The Airy and sine canonical systems and the setup for the convergence}
\label{sec.AirysineCS}

In this section, we first provide a rigorous definition of the stochastic Airy operator \(\Airyop\), and then we build a canonical system that is equivalent to the shifted and scaled operator \(\sAiryop \defeq 2\sqrt{\shift} (\Airyop - \shift)\). We then introduce the time change \(\timechange\) used in Theorem~\ref{thm.CSconvinlaw}, which allows us to define the canonical system on \((0,1)\), like the sine system is. Finally, we derive a change of variables into polar coordinates for solutions of \(\sAiryop f = 0\), and by switching the canonical system's coefficient matrix into these polar coordinates, we obtain a heuristic argument for the vague convergence of Theorem~\ref{thm.CSconvinlaw}, which paves the way for the full proof that is carried out in the subsequent sections.

\subsection{A precise definition of the stochastic Airy operator}
\label{sec.AirysineCS.defAiry}

There are several ways to rigorously define the stochastic Airy operator in order to make precise the heuristic definition~\eqref{eq.AiryopWN}. Originally, Edelman and Sutton avoided the white noise by defining an equivalent operator on a weighted \(L^2\) space, and then defining the Airy operator by conjugating it with an isometry between the weighted and unweighted \(L^2\) spaces~\cite{edelman_random_2007}. Ramírez, Rider and Virág rather used the theory of Schwartz distributions to make sense of the white noise~\cite{ramirez_beta_2011}.

Here, we use the approach used by Minami~\cite{minami_definition_2015}: we understand the stochastic Airy operator as the random generalized Sturm--Liouville operator defined pathwise by the expression
\beq{eq.Airyop}
\Airyop f(t) = - \Bigl( f' - \frac{2}{\sqrt{\beta}} Bf \Bigr)'(t) - \frac{2}{\sqrt{\beta}} B(t) \Bigl( f' - \frac{2}{\sqrt{\beta}} Bf \Bigr)(t) + \Bigl( t - \frac{4}{\beta} B^2(t) \Bigr) f(t),
\eeq
acting on functions \(f \in \mathfrak{D}_{\Airyop} \defeq \bigl\{ f \in \AC\loc(0,\infty) : f' - \frac{2}{\sqrt{\beta}} Bf \in \AC\loc(0,\infty) \bigr\}\), where \(B\) is a standard Brownian motion. This operator fits the framework studied in~\cite{eckhardt_weyl-titchmarsh_2013}, which we related to canonical systems in Section~\ref{sec.CS.SL}, with 
\[
    p\equiv 1,
    \quad 
    q(t) = t - \frac{4}{\beta}B^2(t),
    \quad 
    s(t) = -\frac{2}{\sqrt{\beta}}B(t),
    \quadtext{and} 
    w\equiv 1,
    \quadtext{for all $t \in (0,\infty).$}
\]
Differentiating formally the above expression, one retrieves the heuristic expression~\eqref{eq.AiryopWN} of the operator that was given in the introduction.

In the sequel, we will also need a characterisation of the solutions to \(\Airyop f = zf\) in the form of a stochastic differential equation. To obtain it, note that if a stochastic process \(f\) does solve \(\Airyop f = zf\) pathwise, then by the theory of generalized Sturm--Liouville operators, its sample paths must be absolutely continuous with a derivative \(f'\) a.e.\@ such that \(f' - \frac{2}{\sqrt{\beta}} Bf\) is absolutely continuous. As long as \(f(0)\) and \(f'(0)\) are independent of the Brownian motion, it follows that \(f\) is a semimartingale with finite variation. Two expressions for the Itô differential of \(f' - \frac{2}{\sqrt{\beta}} Bf\) can then be computed: Itô's formula shows that
\[
\diff{\Bigl( f' - \frac{2}{\sqrt{\beta}} Bf \Bigr)}(t)
    = \diff{f'}(t) - \frac{2}{\sqrt{\beta}} f(t) \diff{B}(t) - \frac{2}{\sqrt{\beta}} B(t) f'(t) \diff{t},
\]
but as this process is absolutely continuous, we can also get from the eigenvalue equation that
\[
\diff{\Bigl( f' - \frac{2}{\sqrt{\beta}} Bf \Bigr)}(t)
    = \Bigl( f' - \frac{2}{\sqrt{\beta}} Bf \Bigr)'(t) \diff{t}
    = - \frac{2}{\sqrt{\beta}} B(t) \Bigl( f' - \frac{2}{\sqrt{\beta}} B f \Bigr)(t) \diff{t} + \Bigl( t - \frac{4}{\beta} B^2(t) - z \Bigr) f(t) \diff{t}.
\]
Comparing these expressions shows that \(f\) solves
\beq{eq.SAE}
\begin{aligned}
\diff{f}(t) & = f'(t) \diff{t}, \\
\diff{f'}(t) & = (t - z) f(t) \diff{t} + \frac{2}{\sqrt{\beta}} f(t) \diff{B}(t).
\end{aligned}
\eeq
This stochastic differential equation, called the \emph{stochastic Airy equation}, was studied in detail by Lambert and Paquette in~\cite{lambert_strong_2021}. In particular, they showed that this equation has a unique (up to a multiplicative constant) integrable solution, called the \emph{stochastic Airy function}.

\subsection{A canonical system equivalent to the shifted stochastic Airy operator}
\label{sec.AirysineCS.sAiry}

We now focus on the shifted and scaled stochastic Airy operator \(\sAiryop \defeq 2\sqrt{\shift} (\Airyop - \shift)\) for \(\shift > 0\). Following the ideas from Section~\ref{sec.CS.SL}, we know that as a generalized Sturm--Liouville operator, \(\sAiryop\) is equivalent to a canonical system on \((0,\infty)\). To set up the equivalence, let \(\sAirySLtoCS\colon [0,\infty) \to \mathbb{R}^{2\times 2}\) solve
\[
\sAirySLtoCS'(t) = \begin{pmatrix}
    - \frac{2}{\sqrt{\beta}} B(t) & t - \shift - \frac{4}{\beta} B^2(t) \\ 1 & \frac{2}{\sqrt{\beta}} B(t)
\end{pmatrix} \sAirySLtoCS(t)
\qquadtext{with}
\sAirySLtoCS(0) = \begin{pmatrix}
    \shift^{\nicefrac{1}{4}} & 0 \\ 0 & \shift^{\nicefrac{-1}{4}}
\end{pmatrix}.
\]
Then \(\sAirySLtoCS\) can be written as
\beq{eq.sAirySLtoCS}
\sAirySLtoCS = \begin{pmatrix}
    \shift^{\nicefrac{1}{4}} \sAiryNeumann' - \frac{2}{\sqrt{\beta}} \shift^{\nicefrac{1}{4}} B \sAiryNeumann & \shift^{\nicefrac{-1}{4}} \sAiryDirichlet' - \frac{2}{\sqrt{\beta}} \shift^{\nicefrac{-1}{4}} B \sAiryDirichlet \\ \shift^{\nicefrac{1}{4}} \sAiryNeumann & \shift^{\nicefrac{-1}{4}} \sAiryDirichlet
\end{pmatrix}
\eeq
where \(\sAiryNeumann\) and \(\sAiryDirichlet\) solve \(\sAiryop f = 0\) with initial conditions \(\sAiryDirichlet(0) = \sAiryNeumann'(0) = 1\) and \(\sAiryDirichlet'(0) = \sAiryNeumann(0) = 0\). Then Lemma~\ref{lem.CS.SL} shows that \(f\) solves the eigenvalue equation \(\sAiryop f = zf\) if and only if \(u \defeq \sAirySLtoCS^{-1} \begin{smallpmatrix} \quasi{f} \\ f \end{smallpmatrix}\) solves the canonical system
\beq{eq.sAiryCS}
Ju' = - z\sAirymat u
\quadtext{on}
(0, \infty)
\qquadtext{with}
\sAirymat \defeq \frac{1}{2\sqrt{\shift}} \begin{pmatrix} \sqrt{\shift} \sAiryNeumann^2 & \sAiryDirichlet\sAiryNeumann \\ \sAiryDirichlet\sAiryNeumann & \frac{1}{\sqrt{\shift}} \sAiryDirichlet^2 \end{pmatrix}.
\eeq
Note that the boundary condition \(f(0) = 0\) of the (shifted) stochastic Airy operator corresponds to the boundary condition \(e_0^* Ju(0) = 0\) in the above canonical system. 

In order to work out the vague convergence of the canonical system~\eqref{eq.sAiryCS} to the sine canonical system, we must first modify it so that its time domain is \((0,1)\), like the sine system. Following the time-change property from Proposition~\ref{prop.timechange}, we are looking for an absolutely continuous bijection \(\timechange\colon (0, 1+\varepsilon_{\shift}) \to (0,\infty)\) for some \(\varepsilon_{\shift} > 0\) with \(\varepsilon_{\shift} \downarrow 0\). 

The point of adding this parameter \(\varepsilon_{\shift}\) is related to the change of behavior of the systems at the right endpoint and to the discussion preceeding Theorem~\ref{thm.TMtoWT.LC}. Recall that both the Airy and the sine systems are limit circle at their left endpoint \(0\), but the Airy system is always limit point at its right endpoint \(\infty\) while the sine system is either limit point (when \(\beta \leq 2\)) or limit circle (when \(\beta > 2\)) at its right endpoint \(1\). When \(\beta \leq 2\), the right endpoints of all systems are limit point and we take \(\varepsilon_{\shift} = 0\). However, when \(\beta > 2\), the behavior at the right endpoint changes from limit point to limit circle, so we are in the situation discussed just before Theorem~\ref{thm.TMtoWT.LC}. Thus, we rather want to take \(\varepsilon_{\shift} \downarrow 0\) in order to work out the vague convergence of the canonical systems on \([0,1]\).

A good candidate for \(\timechange\) stems from the zero-temperature limit case \(\beta = \infty\), in which the problem becomes deterministic. Indeed, when \(\beta = \infty\) the Brownian motion disappears from \(\sAiryop\), so the functions \(\sAiryDirichlet[\infty]\) and \(\sAiryNeumann[\infty]\) are solutions of the equation \(2\sqrt{\shift} \bigl( - f''(t) + (t - \shift) f(t) \bigr) = 0\). These are simply solutions of the Airy differential equation but shifted by \(\shift\), so \(\sAiryDirichlet[\infty]\) and \(\sAiryNeumann[\infty]\) can be written as linear combinations of \(\Ai(\cdot - \shift)\) and \(\Bi(\cdot - \shift)\), where \(\Ai\) and \(\Bi\) are the usual Airy functions. The Wronskian property \(\Ai\Bi' - \Ai'\Bi \equiv \nicefrac{1}{\pi}\) then directly leads to
\begin{subequations}
\label{eq.sAiryDirichletNeumanninfty}
\begin{align}
\sAiryDirichlet[\infty](t) & = \pi\bigl( \Bi'(-\shift)\Ai(t-\shift) - \Ai'(-\shift)\Bi(t-\shift) \bigr), \\
\sAiryNeumann[\infty](t) & = \pi\bigl( \Ai(-\shift)\Bi(t-\shift) - \Bi(-\shift)\Ai(t-\shift) \bigr).
\end{align}
\end{subequations}
The Airy functions have well-known asymptotic expansions (see e.g.~\cite{NIST:DLMF, olver_asymptotics_1997}). For \(t < \shift\), these lead to
\begin{subequations}
\label{eq.sAiryDirichletNeumannasymptotics}
\begin{align}
\label{eq.sAiryDirichletasymptotics}
\sAiryDirichlet[\infty](t) = \frac{\shift^{\nicefrac{1}{4}}}{(\shift - t)^{\nicefrac{1}{4}}} \biggl( \cos\Bigl( \frac{2}{3} \shift^{\nicefrac{3}{2}} - \frac{2}{3} (\shift - t)^{\nicefrac{3}{2}} \Bigr) + O\bigl( (\shift - t)^{\nicefrac{-3}{2}} \bigr) \biggr), \\
\label{eq.sAiryNeumannasymptotics}
\sAiryNeumann[\infty](t) = \frac{1}{\shift^{\nicefrac{1}{4}} (\shift - t)^{\nicefrac{1}{4}}} \biggl( \sin\Bigl( \frac{2}{3} \shift^{\nicefrac{3}{2}} - \frac{2}{3} (\shift - t)^{\nicefrac{3}{2}} \Bigr) + O\bigl( (\shift - t)^{\nicefrac{-3}{2}} \bigr) \biggr).
\end{align}
\end{subequations}
Dropping the errors, it follows that if \(\timechange < \shift\), then
\beq{eq.sAirymatinfty}
\timechange' (\sAirymat[\infty]\circ\timechange)
    \approx \frac{\timechange'}{2\sqrt{\shift}} \frac{1}{\sqrt{\shift - \timechange}} \begin{pmatrix}
        \sin^2 (\trigarg\circ\timechange) & \sin(\trigarg\circ\timechange) \cos(\trigarg\circ\timechange) \\ \sin(\trigarg\circ\timechange)\cos(\trigarg\circ\timechange) & \cos^2(\trigarg\circ\timechange)
    \end{pmatrix}
\eeq
where \(\trigarg(t) \defeq \frac{2}{3} \shift^{\nicefrac{3}{2}} - \frac{2}{3} \bigl( \shift - t \bigr)^{\nicefrac{3}{2}}\). Now, \(\timechange\) should be chosen so that this coefficient matrix converges to \(\sinemat[\infty] \circ \logtime\), which is identically equal to \(\frac{1}{2} I_2\) since when \(\beta = \infty\), the \enquote{hyperbolic Brownian motion} \(\HBM[\infty]\) has zero variance, so it is stuck at \(i\) at all times. If the function \(\trigarg\circ\timechange\) grows increasingly fast with \(\shift\), we can expect the oscillations of the trigonometric functions given here to make them converge weakly to their average values, making the above matrix, without the prefactor, converge vaguely to \(\frac{1}{2} I_2\) as \(\shift\to\infty\). This leads us to choose \(\timechange\) as a function that makes the prefactor go to 1, that is, \(\timechange\) should solve the equation
\beq{eq.timechangeODE}
\timechange' = 2\diln \sqrt{\shift} \sqrt{\shift - \timechange}
\qquadtext{with}
\timechange(0) = 0,
\eeq
where \(\diln\) is something that converges to 1 in an appropriate sense. The simplest choice is, of course, to force the \(\diln\)'s to be constants. In that case, the differential equation is separable and the initial value problem is solved by the function
\beq{eq.timechangesol}
\timechange(t) \defeq \shift - \shift (1 - \diln t)^2.
\eeq
However, this is only a bijection \([0, \nicefrac{1}{\diln}] \to [0,\shift]\). When \(\beta \leq 2\), we are looking for a bijection \([0,1) \to [0,\infty)\), so we can take \(\diln = 1\), use \eqref{eq.timechangesol} as a bijection \([0, 1-\nicefrac{1}{\sqrt{\shift}}] \to [0, \shift-1]\) and complete it with a bijection \([1-\nicefrac{1}{\sqrt{\shift}}, 1) \to [\shift-1, \infty)\). When \(\beta > 2\), we are looking for a bijection \([0,1+\varepsilon_{\shift}) \to [0,\infty)\), so if \(\diln < 1\) we can use \eqref{eq.timechangesol} as a bijection \([0,1] \to [0, \shift - \shift(1 - \diln)^2]\). In this case, we take \(\diln = 1 - \nicefrac{1}{\sqrt{\shift}}\) so that \(\timechange(1) = \shift - 1\). 

In both regimes of \(\beta\), the above gives a bijection \([0, \lasttime] \to [0, \shift-1]\) for \(\lasttime \defeq 1 - \nicefrac{1}{\sqrt{\shift}}\) when \(\beta \leq 2\) and \(\lasttime \defeq 1\) when \(\beta > 2\). What we use to complete the bijection will turn out not to make a difference, so we leave it unspecified, assuming only that \(\timechange\colon [\lasttime, \lasttime + \nicefrac{1}{\sqrt{\shift}}) \to [\shift-1, \infty)\) is a bijection so that \(\timechange\) is \(\mathscr{C}^1\) on the full interval \([0, \lasttime+\nicefrac{1}{\sqrt{\shift}})\). 

To sum up, our time change is a \(\mathscr{C}^1\) bijection \(\timechange\colon [0, \lasttime+\nicefrac{1}{\sqrt{\shift}}) \to [0,\infty)\) such that
\beq{eq.deftimechange}
\timechange(t) = \shift - \shift (1 - \diln t)^2 \quad\text{if } t \leq \lasttime
\qquadtext{with}
\Biggl\{
\begin{aligned}
    \diln & = 1, & 
    \lasttime & = 1 - \nicefrac{1}{\sqrt{\shift}} &
    \text{when } & \beta \leq 2, \\[1pt]
    \diln & = 1 - \nicefrac{1}{\sqrt{\shift}}, &
    \lasttime & = 1 &
    \text{when } & \beta > 2.
\end{aligned}
\eeq
It always satisfies \(\timechange(\lasttime) = \shift - 1\) and \(\timechange'(\lasttime) = 2\diln\sqrt{\shift}\). Let us emphasize that the reason to take this specific time change is to counteract the natural decay of the \enquote{Airy-like} oscillations of the solutions that appear in the system's coefficient matrix as in~\eqref{eq.sAirymatinfty}, in order to convert them to \enquote{sine-like} oscillations, with no decay in amplitude. Moreover, the point of taking \(\diln < 1\) when \(\beta > 2\) (i.e., when the sine system is limit circle at \(1\)) is to allow us to analyse the vague convergence of the canonical systems on the closed interval \([0, 1]\), which is more natural in the canonical system setting than working on restricted intervals that would grow to \([0,1]\). Our specific choice of \(\diln = 1 - \nicefrac{1}{\sqrt{\shift}}\) for that case also implies that \(\timechange(1) - \shift = 1\) for all \(\shift\), which will play an important role in our proof of the convergence of the Weyl--Titchmarsh functions in Section~\ref{sec.WTconv}.

\subsection{Polar coordinates and idea behind the convergence}
\label{sec.polarcoords}

To analyse the behavior of the system and its solutions, we now introduce a change of variables into polar coordinates with the goal of obtaining approximations of solutions analogous to those obtained from asymptotics of Airy functions in the deterministic case. These polar coordinates will be crucial in the proof of Theorem~\ref{thm.CSconvinlaw}, as they allow to split the Airy system's coefficient matrix into a highly oscillating part, which vanishes in the vague limit, and an \enquote{average} part, which converges to the stochastic sine system's coefficient matrix.

After deriving the polar coordinates in Section~\ref{sec.polarcoords.polar}, we use them in Section~\ref{sec.polarcoords.matrix} to rewrite the Airy system's coefficient matrix in~\eqref{eq.sAirymatpolar}. Doing so also allows us to identify in Section~\ref{sec.polarcoords.HBM} the process that will converge to the hyperbolic Brownian motion driving the limiting stochastic sine system. 

\subsubsection{Polar coordinates}
\label{sec.polarcoords.polar}

\begin{proposition}
\label{prop.polarcoords}
Let \(\sAiryop\) be defined from a standard Brownian motion \(B\) on a filtered probability space \((\Omega, \mathscr{F}, \{\mathscr{F}_t\}_{t\geq 0}, \mathbb{P})\), and let \(\timechange\), \(\diln\) and \(\lasttime\) be as in \eqref{eq.deftimechange}. If \(f\) solves \(\sAiryop f = 0\) and if \(f(0)\) and \(f'(0)\) are independent of \(B\), then for \(t \in [0,\lasttime]\),
\[
f \circ \timechange(t)
    = \frac{C_{f0}}{\shift^{\nicefrac{1}{4}} \sqrt{1 - \diln t}} \, e^{\AMP(t)} \cos\phase(t)
\qquadtext{and}
f' \circ \timechange(t)
    = C_{f0} \shift^{\nicefrac{1}{4}} \sqrt{1 - \diln t} \, e^{\AMP(t)} \sin\phase(t)
\]
where \(C_{f0}^2 \defeq \sqrt{\shift} f^2(0) + \frac{1}{\sqrt{\shift}} {f'}^2(0)\) and \(\AMP\) and \(\phase\) solve the coupled stochastic differential equations
\begin{align*}
\diff{\AMP}(t)
    & = \biggl( \frac{1}{\beta} + \Bigl( \frac{2}{\beta} - \frac{1}{2} \Bigr) \cos 2\phase(t) + \frac{1}{\beta} \cos 4\phase(t) \biggr) \frac{\diln}{1 - \diln t} \diff{t} + \sqrt{\frac{2}{\beta}} \sin 2\phase(t) \sqrt{\frac{\diln}{1 - \diln t}} \diff{\sBM}(t), \\
\begin{split}
\diff{\phase}(t)
    & = -2\diln \shift^{\nicefrac{3}{2}} (1 - \diln t)^2 \diff{t} - \biggl( \Bigl( \frac{2}{\beta} - \frac{1}{2} \Bigr) \sin 2\phase(t) + \frac{1}{\beta} \sin 4\phase(t) \biggr) \frac{\diln}{1 - \diln t} \diff{t} \\
    & \hspace*{77mm} + \frac{2\sqrt{2}}{\sqrt{\beta}} \cos^2\phase(t) \sqrt{\frac{\diln}{1 - \diln t}} \diff{\sBM}(t)
\end{split}
\end{align*}
with \(\AMP(0) = 0\) and \(\phase(0) = \arctan\bigl( \nicefrac{f'(0)}{\sqrt{\shift} f(0)} \bigr)\), and where \(\sBM(t) \defeq \frac{1}{2\diln\sqrt{\shift}} \int_0^{\timechange(t)} \frac{1}{\sqrt{\shift - s}} \diff{B}(s)\) is a Brownian motion and a martingale with respect to the filtration \(\{\mathscr{F}_{\timechange(t)}\}_{t\in [0,\lasttime)}\).
\end{proposition}

\begin{proof}
Let \(y \defeq f \circ \timechange\) and define two real-valued processes \(r\) and \(\xi\) from
\[
e^{r + i\xi} = Sy + \frac{iy'}{S}
\]
where \(S \colon [0,\lasttime] \to (0,\infty)\) will be specified later. Note that since \(\sAiryop f = 0\), \(f\) and \(f'\) satisfy the coupled SDEs given in~\eqref{eq.SAE} with \(z = \shift\). It then follows from the Dambis--Dubins--Schwarz theorem that
\[
\diff{(f'\circ\timechange)}(t)
    = \bigl( \timechange(t) - \shift \bigr) f\circ\timechange(t) \timechange'(t) \diff{t} + \frac{2}{\sqrt{\beta}} f\circ\timechange(t) \sqrt{\timechange'(t)} \diff{\sBM}(t)
\]
where \(\sBM\) is the Brownian motion defined in the Proposition statement. Itô's formula then shows that
\[
\diff{y'}(t)
    = \timechange''(t) f' \circ \timechange(t) \diff{t} + \timechange'(t) \biggl( \bigl( \timechange(t) - \shift \bigr) f \circ\timechange(t) \timechange'(t) \diff{t} + \frac{2}{\sqrt{\beta}} f\circ\timechange(t) \sqrt{\timechange'(t)} \diff{\sBM}(t) \biggr).
\]
Omitting the explicit time dependence to simplify notation, this simplifies to
\[
\diff{y'}
    = \frac{\timechange''}{\timechange'} y' \diff{t} + (\timechange')^2 (\timechange - \shift) y \diff{t} + \frac{2}{\sqrt{\beta}} (\timechange')^{\nicefrac{3}{2}} y \diff{\sBM}.
\]

We now use the above expression to find the Itô differentials of \(r\) and \(\xi\). Applying Itô's formula,
\begin{align*}
\diff{r} + i\diff{\xi}
    & = \diff{\Bigl( \log\bigl( Sy + \nicefrac{iy'}{S} \bigr) \Bigr)} \\
    & = \frac{1}{Sy + \nicefrac{iy'}{S}} \biggl( \Bigl( S'y + Sy' \Bigr) \diff{t} + i \Bigl( \frac{\timechange''}{\timechange'} \frac{y'}{S} + (\timechange')^2 (\timechange - \shift) \frac{y}{S} - \frac{S'y'}{S^2} \Bigr) \diff{t} + \frac{2i}{\sqrt{\beta}} (\timechange')^{\nicefrac{3}{2}} \frac{y}{S} \diff{\sBM} \biggr) \\
    &\hspace*{88mm} + \frac{2}{\beta} \frac{1}{(Sy + \nicefrac{iy'}{S})^2} (\timechange')^3 \frac{y^2}{S^2} \diff{t}
    \\
    & = e^{-2r} \biggl( \Bigl( SS'y^2 + S^2yy' + \frac{\timechange''}{\timechange'} \frac{(y')^2}{S^2} + (\timechange')^2 (\timechange - \shift) \frac{yy'}{S^2} - \frac{S'(y')^2}{S^3} \Bigr) \diff{t} + \frac{2}{\sqrt{\beta}} (\timechange')^{\nicefrac{3}{2}} \frac{yy'}{S^2} \diff{\sBM} \\
    &\hspace*{22mm} + i\Bigl( \frac{\timechange''}{\timechange'} yy' + (\timechange')^2(\timechange-\shift) y^2 - \frac{2S'yy'}{S} - (y')^2 \Bigr) \diff{t} + \frac{2i}{\sqrt{\beta}} (\timechange')^{\nicefrac{3}{2}} y^2 \diff{\sBM} \biggr) \\
    &\hspace*{22mm} + \frac{2}{\beta} e^{-4r} (\timechange')^3 \Bigl( y^4 - \frac{y^2(y')^2}{S^4} - 2i \frac{y^3y'}{S^2} \Bigr) \diff{t}.
\end{align*}
We thus get the Itô differentials of \(r\) and \(\xi\) by reading off the real and imaginary parts of the equation. Using the identities \(y = \frac{1}{S} e^r \cos\xi\) and \(y' = Se^r \sin\xi\), these differentials can be written only in terms of \(r\) and \(\xi\) as
\begin{subequations}
\label{eq.polarcoords.S}
\begin{align}
\label{eq.polarcoords.S.dr}
\begin{split}
\diff{r}
    & = \frac{2}{\sqrt{\beta}} \frac{(\timechange')^{\nicefrac{3}{2}}}{S^2} \cos\xi\sin\xi \diff{\sBM} + \biggl( \frac{S'}{S} \cos^2\xi + \Bigl( S^2 + \frac{(\timechange')^2 (\timechange-\shift)}{S^2} \Bigr) \cos\xi\sin\xi \\
    &\hspace*{44mm} + \Bigl( \frac{\timechange''}{\timechange'} - \frac{S'}{S} \Bigr) \sin^2\xi + \frac{2}{\beta} \frac{(\timechange')^3}{S^4} \cos^2\xi \bigl( \cos^2\xi - \sin^2\xi \bigr) \biggr) \diff{t}
\end{split}
\shortintertext{and}
\label{eq.polarcoords.S.dxi}
\begin{split}
\diff{\xi}
    & = \frac{2}{\sqrt{\beta}} \frac{(\timechange')^{\nicefrac{3}{2}}}{S^2} \cos^2\xi \diff{\sBM} + \biggl( \frac{(\timechange')^2 (\timechange-\shift)}{S^2} \cos^2\xi + \Bigl( \frac{\timechange''}{\timechange'} - \frac{2S'}{S} \Bigr) \cos\xi\sin\xi \\
    &\hspace*{72mm} - S^2 \sin^2\xi - \frac{4}{\beta} \frac{(\timechange')^3}{S^4} \cos^3\xi\sin\xi \biggr) \diff{t}
\end{split}
\end{align}
\end{subequations}

We now choose \(S\) so as to simplify the equations \eqref{eq.polarcoords.S} on \([0, \lasttime]\). To do this, note that by definition of \(\timechange\) on that interval, \(\timechange(t) - \shift = -\shift(1 - \diln t)^2 = - \frac{1}{4\diln^2\shift} \bigl( \timechange'(t) \bigr)^2\). Hence, if
\[
S \defeq \frac{\timechange'}{\sqrt{2\diln}\shift^{\nicefrac{1}{4}}},
\qquadtext{then}
(\timechange')^2 (\timechange - \shift)
    = - \frac{(\timechange')^4}{4\diln^2\shift}
    = - S^4,
\]
which means that the coefficients of the \(\cos\xi\sin\xi \diff{t}\) and \(\sin^2\xi \diff{t}\) terms in \(\diff{r}\) cancel out, and that the coefficients of the \(\cos^2\xi \diff{t}\) and \(\sin^2\xi \diff{t}\) terms in \(\diff{\xi}\) are equal. Moreover, with that choice of \(S\),
\[
\frac{S'(t)}{S(t)} = \frac{\timechange''(t)}{\timechange'(t)} = - \frac{\diln}{1 - \diln t}
\qquadtext{and}
\frac{\bigl( \timechange'(t) \bigr)^3}{S^4(t)} = \frac{4\diln^2\shift}{\timechange'(t)} = \frac{2\diln}{1 - \diln t},
\]
so the stochastic differential equations for \(r\) and \(\xi\) become
\begin{subequations}
\begin{align}
\label{eq.polarcoords.dr}
\diff{r}(t)
    & = \Bigl( - \cos^2\xi(t) + \frac{4}{\beta} \cos^2\xi(t) \bigl( \cos^2\xi(t) - \sin^2\xi(t) \bigr) \Bigr) \frac{\diln}{1 - \diln t} \diff{t} + \frac{2\sqrt{2}}{\sqrt{\beta}} \cos\xi(t) \sin\xi(t) \sqrt{\frac{\diln}{1 - \diln t}} \diff{\sBM}(t)
\shortintertext{and}
\begin{split}
\label{eq.polarcoords.dxi}
\diff{\xi}(t)
    & = - \Bigl( 2\diln \shift^{\nicefrac{3}{2}} (1 - \diln t)^2 - \frac{\diln}{1 - \diln t} \cos\xi(t) \sin\xi(t) + \frac{8}{\beta} \frac{\diln}{1 - \diln t} \cos^3\xi(t) \sin\xi(t) \Bigr) \diff{t} \\
    & \hspace*{88mm} + \frac{2\sqrt{2}}{\sqrt{\beta}} \cos^2\xi(t) \sqrt{\frac{\diln}{1 - \diln t}} \diff{\sBM}(t).
\end{split}
\end{align}
\end{subequations}
Simplifying with trigonometric identities, we see that \(\phase \defeq \xi\) satisfies the desired stochastic differential equation. For the radial part, we first rewrite \eqref{eq.polarcoords.dr} as 
\[
\diff{r}(t) = \biggl( \frac{1}{\beta} - \frac{1}{2} + \Bigl( \frac{2}{\beta} - \frac{1}{2} \Bigr) \cos 2\xi(t) + \frac{1}{\beta} \cos 4\xi(t) \biggr) \frac{\diln}{1 - \diln t} \diff{t} + \sqrt{\frac{2}{\beta}} \sin 2\xi(t) \sqrt{\frac{\diln}{1 - \diln t}} \diff{\sBM}(t).
\]
Then we can define
\[
\tilde{r}(t)
    \defeq r(0) - \frac{1}{2} \int_0^t \frac{\diln}{1 - \diln s} \diff{s}
    = r(0) + \frac{1}{2} \log(1 - \diln t)
\]
in order to integrate out explicitly the term \(-\frac{1}{2} \frac{\diln}{1-\diln t} \diff{t}\) from the equation for \(r\), and \(\AMP \defeq r - \tilde{r}\) satisfies the desired stochastic differential equation. Moreover, \(\AMP(0) = 0\) by definition, and 
\[
e^{r(0) + i\xi(0)} = \sqrt{2\diln} \shift^{\nicefrac{3}{4}} f(0) + i\sqrt{2\diln} \shift^{\nicefrac{1}{4}} f'(0)
\]
so
\[
\xi(0) = \arctan\Bigl( \frac{f'(0)}{\sqrt{\shift} f(0)} \Bigr)
\qquadtext{and}
e^{r(0)} = \sqrt{2\diln\shift \Bigl( \sqrt{\shift} f^2(0) + \frac{1}{\sqrt{\shift}} {f'}^2(0) \Bigr)} \eqdef \sqrt{2\diln\shift} C_{f0}.
\]
This finally yields the announced representations
\begin{align*}
f \circ \timechange(t)
    & = \frac{1}{S(t)} e^{r(t)} \cos\xi(t)
    = \frac{\sqrt{2\diln} \shift^{\nicefrac{1}{4}}}{\timechange'(t)} e^{\tilde{r}(t) + \AMP(t)} \cos\phase(t)
    = \frac{C_{f0}}{\shift^{\nicefrac{1}{4}} \sqrt{1 - \diln t}} \, e^{\AMP(t)} \cos\phase(t)
\shortintertext{and}
f' \circ \timechange(t)
    & = \frac{S(t)}{\timechange'(t)} e^{r(t)} \sin\xi(t)
    = \frac{1}{\sqrt{2\diln} \shift^{\nicefrac{1}{4}}} e^{\tilde{r}(t) + \AMP(t)} \sin\phase(t)
    = C_{f0} \shift^{\nicefrac{1}{4}} \sqrt{1 - \diln t}\, e^{\AMP(t)} \sin\phase(t).
\qedhere
\end{align*}
\end{proof}

\subsubsection{The stochastic Airy system in polar coordinates}
\label{sec.polarcoords.matrix}

Using the polar coordinates introduced in Proposition~\ref{prop.polarcoords}, we can write the time-changed fundamental solutions \(\sAiryDirichlet\circ\timechange\) and \(\sAiryNeumann\circ\timechange\) from pairs \((\AMPD, \phaseD)\) and \((\AMPN, \phaseN)\) as 
\begin{subequations}
\label{eq.polarcoordsDN}
\begin{align}
\label{eq.polarcoordsD}
\shift^{\nicefrac{-1}{4}} \sAiryDirichlet\circ\timechange(t) 
    & = \frac{1}{\shift^{\nicefrac{1}{4}} \sqrt{1 - \diln t}}\, e^{\AMPD(t)} \cos\phaseD(t), \\
\label{eq.polarcoordsN}
\shift^{\nicefrac{1}{4}} \sAiryNeumann\circ\timechange(t)
    & = \frac{1}{\shift^{\nicefrac{1}{4}} \sqrt{1 - \diln t}}\, e^{\AMPN(t)} \cos\phaseN(t),
\intertext{with \(\phaseD(0) = 0\) and \(\phaseN(0) = \nicefrac{\pi}{2}\), and we can write their derivatives as}
\label{eq.polarcoordsD'}
\shift^{\nicefrac{-1}{4}} \sAiryDirichlet'\circ\timechange(t) 
    & = \shift^{\nicefrac{1}{4}} \sqrt{1 - \diln t}\, e^{\AMPD(t)} \sin\phaseD(t), \\
\label{eq.polarcoordsN'}
\shift^{\nicefrac{1}{4}} \sAiryNeumann'\circ\timechange(t)
    & = \shift^{\nicefrac{1}{4}} \sqrt{1 - \diln t}\, e^{\AMPN(t)} \sin\phaseN(t).
\end{align}
\end{subequations}
These four polar coordinates are not completely independent from one another. Indeed, as \(\sAiryDirichlet\) and \(\sAiryNeumann\) both solve \(\sAiryop f = 0\), their Wronskian \(\sAiryDirichlet\quasi{\sAiryNeumann} - \quasi{\sAiryDirichlet}\sAiryNeumann\) is constant. It is thus equal to its value at \(0\), which is \(1\), at all times. Hence,
\[
1 \equiv (\sAiryDirichlet\sAiryNeumann' - \sAiryDirichlet'\sAiryNeumann)\circ\timechange
    = e^{\AMPD+\AMPN} \cos\phaseD \sin\phaseN - e^{\AMPD+\AMPN} \cos\phaseN \sin\phaseD,
\]
and it follows that
\beq{eq.ampsphasesWronskianidentity}
1 \equiv e^{\AMPD+\AMPN} \sin\bigl( \phaseN - \phaseD \bigr).
\eeq

Now, recall that the coefficient matrix of the time-changed Airy system is
\beq{eq.sAirymattimechanged}
\timechange'(t) \sAirymat \circ \timechange(t)
    = \frac{\timechange'(t)}{2\sqrt{\shift}} \begin{pmatrix} \sqrt{\shift} \sAiryNeumann^2 & \sAiryDirichlet\sAiryNeumann \\ \sAiryDirichlet\sAiryNeumann & \frac{1}{\sqrt{\shift}} \sAiryDirichlet^2 \end{pmatrix} \bigl(\timechange(t)\bigr).
\eeq
Changing variables to polar coordinates, this matrix becomes, for \(t \in [0,\lasttime]\),
\begin{align*}
\timechange' (\sAirymat\circ\timechange)
    & = \diln \begin{pmatrix} e^{2\AMPN} \cos^2\phaseN & e^{\AMPD + \AMPN} \cos\phaseD\cos\phaseN \\[2mm] e^{\AMPD + \AMPN} \cos\phaseD\cos\phaseN & e^{2\AMPD} \cos^2\phaseD \end{pmatrix} \\
    & = \frac{\diln e^{2\AMPN}}{2} \begin{pmatrix} 1 & e^{-\diffamps}\cos\diffphases \\ e^{-\diffamps}\cos\diffphases & e^{-2\diffamps} \end{pmatrix} + \frac{\diln}{2} \begin{pmatrix} e^{2\AMPN} \cos 2\phaseN & e^{\sumamps} \cos\sumphases \\ e^{\sumamps} \cos\sumphases & e^{2\AMPD} \cos 2\phaseD \end{pmatrix}
\end{align*}
where we introduced the notations \(\diffamps \defeq \AMPN - \AMPD\) and \(\sumamps \defeq \AMPN + \AMPD\), and likewise for phases. Note that \eqref{eq.ampsphasesWronskianidentity} implies that \(e^{-\diffamps} \sin\diffphases = e^{-2\AMPN}\). Using this to replace the prefactor, we get
\beq{eq.sAirymatpolar}
\timechange' (\sAirymat\circ\timechange) = \frac{\diln}{2e^{-\diffamps}\sin\diffphases} \begin{pmatrix} 1 & e^{-\diffamps}\cos\diffphases \\ e^{-\diffamps}\cos\diffphases & e^{-2\diffamps} \end{pmatrix} + \frac{\diln}{2} \begin{pmatrix} e^{2\AMPN} \cos 2\phaseN & e^{\sumamps} \cos\sumphases \\ e^{\sumamps} \cos\sumphases & e^{2\AMPD} \cos 2\phaseD \end{pmatrix}. 
\eeq

\subsubsection{Identification of the limiting hyperbolic Brownian motion}
\label{sec.polarcoords.HBM}

Recall that the coefficient matrix of the stochastic sine canonical system is given by
\beq{eq.sinematHBM}
\sinemat\circ\logtime = \frac{1}{2\Im\HBM\circ\logtime} \begin{pmatrix} 1 & -\Re\HBM\circ\logtime \\ -\Re\HBM\circ\logtime & \abs{\HBM\circ\logtime}^2 \end{pmatrix}
\eeq
where \(\HBM\) is a hyperbolic Brownian motion with variance \(\nicefrac{4}{\beta}\) started at \(i\) in the upper half-plane and \(\logtime\) is a logarithmic time-change. Comparing the two coefficient matrices~\eqref{eq.sAirymatpolar} and~\eqref{eq.sinematHBM}, we see that \eqref{eq.sAirymatpolar} will converge to \eqref{eq.sinematHBM} if the second term of \eqref{eq.sAirymatpolar} vanishes in the vague limit (which is what we expect because of the increasingly fast oscillations of the phases \(\phaseD\) and \(\phaseN\)) and if \(-\exp\bigl(-\diffamps-i\diffphases\bigr)\) converges to a hyperbolic Brownian motion with variance \(\nicefrac{4}{\beta}\) started at \(i\) in the upper half-plane, run in logarithmic time. 

To see how this can be true, we can compute the Itô differential of the process \(-\exp\bigl( -\diffamps - i\diffphases \bigr)\). First, taking the differences of the SDEs from Proposition~\ref{prop.polarcoords}, we get
\begin{align*}
\begin{split}
\diff{\diffamps}(t)
    & = \frac{2\sqrt{2}}{\sqrt{\beta}} \sin\diffphases(t) \cos\sumphases(t) \sqrt{\frac{\diln}{1 - \diln t}} \diff{\sBM}(t) \\
    & \hspace*{18mm} - \biggl( \Bigl( \frac{4}{\beta} - 1 \Bigr) \sin\diffphases(t) \sin\sumphases(t) + \frac{2}{\beta} \sin 2\diffphases(t) \sin 2\sumphases(t) \biggr) \frac{\diln}{1-\diln t} \diff{t}
\end{split} \\
\shortintertext{and}
\begin{split}
\diff{\diffphases}(t)
    & = - \frac{2\sqrt{2}}{\sqrt{\beta}} \sin\diffphases(t) \sin\sumphases(t) \sqrt{\frac{\diln}{1 - \diln t}} \diff{\sBM}(t) \\
    & \hspace*{18mm} - \biggl( \Bigl( \frac{4}{\beta} - 1 \Bigr) \sin\diffphases(t) \cos\sumphases(t) + \frac{2}{\beta} \sin 2\diffphases(t) \cos 2\sumphases(t) \biggr) \frac{\diln}{1-\diln t} \diff{t}.
\end{split}
\end{align*}
Then, applying Itô's formula to \(-\exp\bigl( -\diffamps - i\diffphases \bigr)\) and using the above expressions for \(\diff{\diffamps}\) and \(\diff{\diffphases}\), we get
\beq{eq.dalmostHBM}
\begin{aligned}
\diff{\bigl( -e^{-\diffamps-i\diffphases} \bigr)}(t)
    & = \frac{2}{\sqrt{\beta}} e^{-\diffamps(t)} \sin\diffphases(t) \, e^{-2i\phaseN(t)} \sqrt{\frac{2\diln}{1-\diln t}} \diff{\sBM}(t) \\
    & \hspace*{22mm} - ie^{-\diffamps(t)} \sin\diffphases(t) \biggl( \Bigl( \frac{4}{\beta} - 1 \Bigr) e^{-2i\phaseN(t)} + \frac{4}{\beta} e^{-4i\phaseN(t)} \biggr) \frac{\diln}{1 - \diln t} \diff{t}.
\end{aligned}
\eeq
Due to the strong oscillations of \(e^{-2i\phaseN}\) and \(e^{-4i\phaseN}\), it is reasonable to expect the second line here to vanish as \(\shift\to\infty\). Neglecting the second line, this SDE has the same form as the SDE for a hyperbolic Brownian motion in the upper half-plane, but where the driving complex Brownian motion is replaced with the process
\beq{eq.CBM}
\int_0^t e^{-2i\phaseN(s)} \sqrt{\frac{2\diln}{1-\diln s}} \diff{\sBM}(s).
\eeq
We will see that indeed, this process converges in distribution to a complex Brownian motion run in logarithmic time.

\section{A coupling between the Airy and sine systems}
\label{sec.coupling}

In this section, we start the proof of Theorem~\ref{thm.CSconvinlaw}, that is, the proof of the vague convergence of the canonical system for \(\sAiryop\) to the sine canonical system as \(\shift\to\infty\). Following the ideas developed at the end of the last section, the first step of the proof is to build a coupling between real Brownian motions \(\sBM\) and a single complex Brownian motion \(W\) such that the process~\eqref{eq.CBM} becomes pathwise close to \(W\) as \(\shift\to\infty\). This first step is the purpose of this section and the content of the following lemma.

\begin{lemma}
\label{lem.probspace}
\setshift{E_n}
Let \(\{\shift\}_{n\in\mathbb{N}} \subset (0,\infty)\) satisfy \(\shift \to \infty\). There exists a probability space on which are defined a collection \(\{\sBM\}_{n\in\mathbb{N}}\) of standard real Brownian motions and a standard complex Brownian motion \(W\) such that if \(\phaseN\) is the solution with \(\phaseN(0) = \nicefrac{\pi}{2}\) to the SDE from Proposition~\ref{prop.polarcoords} driven by \(\sBM\), then for any \(\alpha \in (0,\nicefrac{1}{2})\) and \(\delta \in (0,\alpha)\), there is a \(C > 0\) such that for all \setshift{E}\(\shift \in \{\shift_n\}_{n\in\mathbb{N}}\) large enough,
\beq{eq.probspace}
\bprob[\bigg]{\sup_{t\in\timedom} \abs[\bigg]{\int_0^t e^{-2i\phaseN(s)} \sqrt{\frac{2\diln}{1 - \diln s}} \diff{\sBM}(s) - W\circ\slogtime(t)} \geq \shift^{-\alpha+\delta}} 
    \leq 3\shift^{4\alpha} \log^2\shift \exp\bigl( -C\shift^{\nicefrac{2\delta}{3}} \bigr)
\eeq
where \(\timedom \defeq [0, (1 - \shift^{\nicefrac{-1}{2}+\alpha}) / \diln]\) and \(\slogtime(t) \defeq -\log(1 - \diln t)\).
\end{lemma}

\begin{remark}
\label{rmk.probspace}
This probability space therefore supports a sequence of shifted and scaled stochastic Airy systems. Indeed, each Brownian motion \(\sBM\) can be taken to define two pairs of polar coordinates \((\AMPD, \phaseD)\) and \((\AMPN, \phaseN)\) as solutions to the SDEs from Proposition~\ref{prop.polarcoords} with \(\AMPD(0) = \AMPN(0) = \phaseD(0) = 0\) and \(\phaseN(0) = \nicefrac{\pi}{2}\), and these can be used to define fundamental solutions \(\sAiryDirichlet\) and \(\sAiryNeumann\) as in \eqref{eq.polarcoordsDN}, which then allow to define a coefficient matrix \eqref{eq.sAirymattimechanged} for the canonical system equivalent to \(\sAiryop\). In the same way, the probability space also supports a stochastic sine canonical system, since its coefficient matrix can be defined from a hyperbolic Brownian motion driven by \(W\). 
\end{remark}

\begin{proof}
To build an appropriate probability space, we start by fixing an \(\shift > 0\) and building a coupling between a real and a complex Brownian motions in such a way that the above property is satisfied. To do this, we first couple a discretization of the integral in \eqref{eq.probspace} with a discrete random walk, which we then extend to a complex Brownian motion. Then, we combine all spaces together.

\setcounter{step}{-1}
\step{Discrete time setup}

Fix \(\shift > 0\). Let \(\sBM\) be a standard real Brownian motion, and let \(\phaseN\) be the solution with \(\phaseN(0) = 0\) to the SDE from Proposition~\ref{prop.polarcoords} driven by \(\sBM\). We discretize the time interval \([0, \lasttime]\) with a partition \(\{t_j\}_{j=0}^N\) with the bounds \(t_0 \defeq 0\) and \(\diln t_N \defeq 1 - \nicefrac{1}{\sqrt{\shift}}\), and with
\[
\diln t_j \defeq 1 - \frac{1}{(1 + \shift^{-p})^j}
\quadtext{for} 1 \leq j \leq N-1
\]
and a parameter \(p > 0\) that will be specified later. This means that we take
\beq{eq.probspace.Ndefandbound}
N \defeq \ceil[\Big]{\frac{\log\shift}{2\log(1 + \shift^{-p})}}
    \leq 1 + \frac{\shift^p \log\shift}{2 - \shift^{-p}}
\eeq
where the inequality follows from \(\log(1 + x) \geq x(1 - \nicefrac{x}{2})\), which holds for any \(x > 0\). 

We begin by considering the integral in \eqref{eq.probspace} at discrete times \(\{t_j\}_{j=1}^N\), that is, we focus on the discrete martingale
\beq{eq.probspace.defBxi}
\sum_{j=1}^n \mathscr{B}^\xi_j
\qquadtext{where}
\mathscr{B}^\xi_j \defeq \int_{t_{j-1}}^{t_j} e^{-2i\phaseN(s)} \sqrt{\frac{2\diln}{1 - \diln s}} \diff{\sBM}(s)
\eeq
for \(n \in \{1, \hdots, N\}\). To compare this with a discretization of a complex Brownian motion, we will first compare it with the discrete martingale
\beq{eq.probspace.defBtheta}
\sum_{j=1}^n e^{-2i{(\phaseN(t_{j-1}) - \theta(t_{j-1}))}} \mathscr{B}^\theta_j
\qquadtext{where}
\mathscr{B}^\theta_j \defeq \int_{t_{j-1}}^{t_j} e^{-2i\theta(s)} \sqrt{\frac{2\diln}{1 - \diln s}} \diff{\sBM}(s)
\eeq
and where \(\theta\) is the deterministic part of \(\phaseN\), that is, \(\theta\) solves
\beq{eq.probspace.deftheta}
\theta'(t) = -2\diln \shift^{\nicefrac{3}{2}} (1 - \diln t)^2
\quadtext{with}
\theta(0) = \nicefrac{\pi}{2},
\qquadtext{so}
\theta(t) = \frac{\pi}{2} - \frac{2}{3} \shift^{\nicefrac{3}{2}} \bigl( 1 - (1 - \diln t)^3 \bigr).
\eeq
The \(\mathscr{B}^\theta_j\)'s are Gaussian, and due to the fast oscillations of \(\theta\), we expect their real and imaginary parts to be almost independent and to have variances close to
\beq{eq.probspace.defsigma}
\sigma^2
    \defeq \int_{t_{j-1}}^{t_j} \frac{\diln}{1 - \diln s} \diff{s}
    = \log\Bigl( \frac{1 - \diln t_{j-1}}{1 - \diln t_j} \Bigr)
    = \log(1 + \shift^{-p})
\eeq
for \(j < N\). The actual correlation structure of \(\mathscr{B}^\theta_j\) is given by the correlation matrix
\[
\Sigma_j \defeq \begin{pmatrix}
    \expect(\Re\mathscr{B}^\theta_j)^2 & \expect\Re\mathscr{B}^\theta_j\Im\mathscr{B}^\theta_j \\
    \expect\Re\mathscr{B}^\theta_j\Im\mathscr{B}^\theta_j & \expect(\Im\mathscr{B}^\theta_j)^2
\end{pmatrix},
\]
so if we define
\beq{eq.probspace.defWj}
W_j \defeq e^{-2i(\phaseN(t_{j-1}) - \theta(t_{j-1}))} \sigma \Sigma_j^{\nicefrac{-1}{2}} \mathscr{B}^\theta_j
\eeq
where we identify \(x + iy \in \mathbb{C}\) with \((x,y) \in \mathbb{R}^2\), then \(W_j \sim \mathbb{C}\normal{0}{\sigma^2}\) in the sense that its real and imaginary parts are independent and distributed as \(\normal{0}{\sigma^2}\). Indeed, multiplying \(\mathscr{B}^\theta_j\) by \(\Sigma_j^{\nicefrac{-1}{2}}\) removes its correlation structure so \(\sigma \Sigma_j^{\nicefrac{-1}{2}} \mathscr{B}^\theta_j \sim \mathbb{C}\normal{0}{\sigma^2}\), and since the multivariate normal distribution is invariant under multiplication by orthogonal martices, the multiplication by \(e^{-2i(\phaseN(t_{j-1}) - \theta(t_{j-1}))}\) has no effect on the distribution. This also shows that the \(W_j\)'s are independent.

Our first goal is thus to study, for \(n \in \{1, \hdots, N\}\),
\[
\sum_{j=1}^n (\mathscr{B}^\xi_j - W_j)
    = \Delta^{\xi\theta}_n + \Delta^{\theta W}_n
\]
where
\[
\Delta^{\xi\theta}_n
    \defeq \sum_{j=1}^n \bigl( \mathscr{B}^\xi_j - e^{-2i(\phaseN(t_{j-1}) - \theta(t_{j-1}))} \mathscr{B}^\theta_j \bigr)
\qquadtext{and}
\Delta^{\theta W}_n
    \defeq \sum_{j=1}^n \bigl( e^{-2i(\phaseN(t_{j-1}) - \theta(t_{j-1}))} \mathscr{B}^\theta_j - W_j \bigr).
\]
Setting \(\Delta^{\xi\theta}_0 \defeq 0 \eqdef \Delta^{\theta W}_0\), both \(\{\Delta^{\xi\theta}_n\}_{n=0}^N\) and \(\{\Delta^{\theta W}_n\}_{n=0}^N\) are martingales with respect to the filtration \(\{\mathscr{F}_j\}_{j=0}^N\) generated by \(\{\sBM(t_j)\}_{j=0}^N\). 

\step{Replacing the phase by its deterministic approximation}

We first control the martingale \(\{\Delta^{\xi\theta}_n\}_{n=0}^N\). Fixing \(j \in \{1, \hdots, N-1\}\) and using the identity \(e^{-2ix} - e^{-2iy} = 2ie^{-i(x+y)}\sin(y-x)\), we can write the martingale's \(j\)th increment as
\begin{multline*}
\mathscr{B}^\xi_j - e^{-2i(\phaseN(t_{j-1}) - \theta(t_{j-1}))} \mathscr{B}^\theta_j \\
    = 2i \int_{t_{j-1}}^{t_j} e^{-i(\phaseN(s) + \theta(s) - \theta(t_{j-1}) + \phaseN(t_{j-1}))} \sin\Bigl( \theta(s) - \theta(t_{j-1}) - \phaseN(s) + \phaseN(t_{j-1}) \Bigr) \sqrt{\frac{2\diln}{1 - \diln s}} \diff{\sBM}(s).
\end{multline*}
This expression already makes it obvious that the quadratic variation of the stochastic integral is bounded by \(8\sigma^2\), and therefore that the increment \(\mathscr{B}^\xi_j - e^{-2i(\phaseN(t_{j-1}) - \theta(t_{j-1}))} \mathscr{B}^\theta_j\) is \(8\sigma^2\)-subgaussian, both conditionally on \(\mathscr{F}_{j-1}\) and unconditionally. 

This control on the subgaussian constant of the increments is not enough to obtain a good control on the whole martingale, but we can improve it by getting a better control on the conditional variances of the increments. To do so, we apply Itô's isometry to the above expression for the \(j\)th increment, and as sine is \(1\)-Lipschitz, we get
\[
\bexpect[\Big]{\abs[\big]{\mathscr{B}^\xi_j - e^{-2i(\phaseN(t_{j-1}) - \theta(t_{j-1}))} \mathscr{B}^\theta_j}^2 \given[\Big] \mathscr{F}_{j-1}}
    \leq 8 \bexpect[\bigg]{\int_{t_{j-1}}^{t_j} \abs[\big]{\theta(s) - \theta(t_{j-1}) - \phaseN(s) + \phaseN(t_{j-1})}^2 \frac{\diln}{1 - \diln s} \diff{s} \given[\Big] \mathscr{F}_{j-1}}.
\]
Now, by definition of \(\theta\) and \(\phaseN\), for any \(s \in [t_{j-1}, t_j]\),
\begin{multline*}
\theta(s) - \theta(t_{j-1}) - \phaseN(s) + \phaseN(t_{j-1})
    = \int_{t_{j-1}}^s \biggl( \Bigl( \frac{2}{\beta} - \frac{1}{2} \Bigr) \sin 2\phaseN(u) + \frac{1}{\beta} \sin 4\phaseN(u) \biggr) \frac{\diln}{1 - \diln u} \diff{u} \\
    - \frac{2\sqrt{2}}{\sqrt{\beta}} \int_{t_{j-1}}^s \cos^2 \phaseN(u) \sqrt{\frac{\diln}{1 - \diln u}} \diff{\sBM}(u).
\end{multline*}
The first term is immediately bounded by \(\bigl( \frac{3}{\beta} + \frac{1}{2} \bigr) \sigma^2\), while the quadratic variation of the second one is bounded by \(\frac{8}{\beta} \sigma^2\) so Bernstein's inequality for martingales (see e.g.\@ \cite[Exercise~IV.3.16]{revuz_continuous_1999} for a general statement of that inequality) shows that for any \(x > 0\),
\[
\bprob[\bigg]{\sup_{s\in [t_{j-1}, t_j]} \abs[\Big]{\frac{2\sqrt{2}}{\sqrt{\beta}} \int_{t_{j-1}}^s \cos^2 \phaseN(u) \sqrt{\frac{\diln}{1 - \diln u}} \diff{\sBM}(u)} \geq x \given[\Big] \mathscr{F}_{j-1}} \leq 2 \exp\Bigl( - \frac{\beta x^2}{16 \sigma^2} \Bigr).
\]
On the complementary event,
\beq{eq.probspace.boundforVarBxi-Btheta}
\int_{t_{j-1}}^{t_j} \abs[\big]{\theta(s) - \theta(t_{j-1}) - \phaseN(s) + \phaseN(t_{j-1})}^2 \frac{\diln}{1 - \diln s} \diff{s}
    < \biggl( \Bigl( \frac{3}{\beta} + \frac{1}{2} \Bigr) \sigma^2 + x \biggr)^2 \sigma^2.
\eeq
Now, using a layer cake representation, we can write
\begin{multline*}
\bexpect[\Big]{\abs[\big]{\mathscr{B}^\xi_j - e^{-2i(\phaseN(t_{j-1}) - \theta(t_{j-1}))} \mathscr{B}^\theta_j}^2 \given[\Big] \mathscr{F}_{j-1}}
    \leq 8 \Bigl( \frac{3}{\beta} + \frac{1}{2} \Bigr)^2 \sigma^6 \\
    + 8 \int_{(\frac{3}{\beta} + \frac{1}{2})^2\sigma^6}^\infty \bprob[\bigg]{\int_{t_{j-1}}^{t_j} \abs[\big]{\theta(s) - \theta(t_{j-1}) - \phaseN(s) + \phaseN(t_{j-1})}^2 \frac{\diln}{1 - \diln s} \diff{s} \geq y \given[\Big] \mathscr{F}_{j-1}} \diff{y}.
\end{multline*}
Changing variables by setting \(y = \bigl( \bigl( \frac{3}{\beta} + \frac{1}{2} \bigr) \sigma^2 + x \bigr)^2 \sigma^2\), we can use the estimate~\eqref{eq.probspace.boundforVarBxi-Btheta} and integrate explicitly. Since \(\sigma^2 = \log(1 + \shift^{-p}) \leq \shift^{-p} \leq 1\) for any \(\shift \geq 1\), we find in that case
\beq{eq.probspace.VarBxi-Btheta}
\bexpect[\Big]{\abs[\big]{\mathscr{B}^\xi_j - e^{-2i(\phaseN(t_{j-1}) - \theta(t_{j-1}))} \mathscr{B}^\theta_j}^2 \given[\Big] \mathscr{F}_{j-1}} \leq C_\beta \sigma^4
\quadtext{with}
C_\beta \defeq 8 \Bigl( \frac{3}{\beta} + \frac{1}{2} \Bigr)^2 + 64 \sqrt{\frac{\pi}{\beta}} \Bigl( \frac{3}{\beta} + \frac{1}{2} \Bigr) + \frac{256}{\beta}.
\eeq

Finally, from the estimate \eqref{eq.probspace.VarBxi-Btheta} on the conditional variances of increments and the fact that the increments are \(8\sigma^2\)-subgaussian, both conditionally on their past and unconditionally, we can apply Corollary~\ref{cor.FreedmanSubG} of Freedman's inequality (see Theorem~\ref{thm.Freedman} in the appendix) to both the real and imaginary parts of the martingale. Because \(\sigma^2 \leq \shift^{-p}\) and \(N \leq \shift^p\log\shift\) for all \(\shift\) large enough by \eqref{eq.probspace.Ndefandbound}, this shows that
\beq{eq.probspace.Deltaxitheta}
\begin{aligned}
\bprob[\bigg]{\sup_{1\leq n\leq N-1} \abs{\Delta^{\xi\theta}_n} > x}
    \leq 2\shift^{2p} \log^2\shift \exp\biggl( - \frac{1}{2} \min\Bigl\{ & \Bigl( \frac{\shift^{\nicefrac{p}{2}} x\log 2}{12\sqrt{2}} \Bigr)^{\nicefrac{2}{3}}, \frac{\shift^p x^2 \log 2}{9C_\beta \log\shift} \Bigr\} \biggr) \\
    & + 4\exp\biggl( - \frac{x\log 2}{3\sqrt{2C_\beta} \log\shift} \exp\biggl( \frac{1}{4} \Bigl( \frac{\shift^{\nicefrac{p}{2}} x\log 2}{12\sqrt{2}} \Bigr)^{\nicefrac{2}{3}} \biggr) \biggr)
\end{aligned}
\eeq
for any \(\shift\) large enough.

\step{Comparing with a Gaussian random walk}

We now estimate the supremum of the martingale \(\{\Delta^{\theta W}_n\}_{n=1}^N\). To do so, our plan is to estimate the conditional variances of the increments, and then apply Azuma's inequality for subgaussian martingales. 

Recall that \(e^{-2i(\phaseN(t_{j-1}) - \theta(t_{j-1}))} \mathscr{B}^\theta_j - W_j = e^{-2i(\phaseN(t_{j-1}) - \theta(t_{j-1}))} (I_2 - \sigma\Sigma_j^{\nicefrac{-1}{2}}) \mathscr{B}^\theta_j\) by definition of \(W_j\), where again we identify complex numbers with vectors in \(\mathbb{R}^2\) for matrix products. This means that
\beq{eq.probspace.VarBtheta-W.1}
\bexpect[\Big]{\abs[\big]{e^{-2i(\phaseN(t_{j-1}) - \theta(t_{j-1}))} \mathscr{B}^\theta_j - W_j}^2 \given[\Big] \mathscr{F}_{j-1}} \leq \norm{I_2 - \sigma\Sigma_j^{\nicefrac{-1}{2}}}_2^2 \bexpect[\big]{\abs{\mathscr{B}^\theta_j}^2 \given[\big] \mathscr{F}_{j-1}}.
\eeq
The definition of \(\mathscr{B}^\theta_j\) directly yields
\beq{eq.probspace.VarBtheta}
\bexpect[\big]{\abs{\mathscr{B}^\theta_j}^2 \given[\big] \mathscr{F}_{j-1}}
    = \int_{t_{j-1}}^{t_j} \frac{2\diln}{1 - \diln s} \diff{s}
    = 2\sigma^2,
\eeq
so we only have to estimate the norm of the matrix \(I_2 - \sigma\Sigma_j^{\nicefrac{-1}{2}}\), which is its largest eigenvalue, in magnitude. To do so, note that by Itô isometry,
\begin{align*}
\expect\Re\mathscr{B}^\theta_j \Im\mathscr{B}^\theta_j
    & = - \int_{t_{j-1}}^{t_j} \cos 2\theta(s) \sin 2\theta(s) \frac{2\diln}{1 - \diln s} \diff{s}
    = - \int_{t_{j-1}}^{t_j} \sin 4\theta(s) \frac{\diln}{1 - \diln s} \diff{s}, \\
\expect(\Re\mathscr{B}^\theta_j)^2 - \sigma^2
    & = \int_{t_{j-1}}^{t_j} \bigl( 2 \cos^2 2\theta(s) - 1 \bigr) \frac{\diln}{1 - \diln s} \diff{s}
    = \int_{t_{j-1}}^{t_j} \cos 4\theta(s) \frac{\diln}{1 - \diln s} \diff{s}, \\
\expect(\Im\mathscr{B}^\theta_j)^2 - \sigma^2
    & = \int_{t_{j-1}}^{t_j} \bigl( 2 \sin^2 2\theta(s) - 1 \bigr) \frac{\diln}{1 - \diln s} \diff{s}
    = - \int_{t_{j-1}}^{t_j} \cos 4\theta(s) \frac{\diln}{1 - \diln s} \diff{s},
\end{align*}
so that
\[
\sigma^{-2} \Sigma_j - I_2 = \sigma^{-2} \begin{pmatrix} x & -y \\ -y & -x \end{pmatrix}
\qquadtext{for}
z = x + iy \defeq \int_{t_{j-1}}^{t_j} e^{4i\theta(s)} \frac{\diln}{1 - \diln s} \diff{s}.
\]
Hence, \(\sigma^{-2} \Sigma_j - I_2\) has eigenvalues \(\pm \sigma^{-2} \abs{z}\). This also means that \(\sigma^{-2} \Sigma_j\) has eigenvalues \(1 \pm \sigma^{-2} \abs{z}\), so that \(\sigma \Sigma_j^{\nicefrac{-1}{2}}\) has eigenvalues \((1 \pm \sigma^{-2} \abs{z})^{\nicefrac{-1}{2}}\). Writing \(\sigma^{-2} \Sigma_j - I_2 = (\sigma^{-1}\Sigma_j^{\nicefrac{1}{2}} - I_2)(\sigma^{-1}\Sigma_j^{\nicefrac{1}{2}} + I_2)\), we can then decompose \(I_2 - \sigma \Sigma_j^{\nicefrac{-1}{2}} = \sigma \Sigma_j^{\nicefrac{-1}{2}} (\sigma^{-2} \Sigma_j - I_2)(\sigma^{-1} \Sigma_j^{\nicefrac{1}{2}} + I_2)^{-1}\). Because \(\Sigma_j \succeq 0\), the eigenvalues of \(\sigma^{-1} \Sigma_j^{\nicefrac{1}{2}} + I_2\) must be at least \(1\), so combining our eigenvalue computations yields
\beq{eq.probspace.normsigmaSigma-I}
\norm{I_2 - \sigma\Sigma_j^{\nicefrac{-1}{2}}}_2 \leq \frac{\abs{z}}{\sigma^2 \sqrt{1 - \sigma^{-2} \abs{z}}}.
\eeq
To estimate \(\abs{z}\), we then use the fact that \(\theta'(s) = -2\diln \shift^{\nicefrac{3}{2}} (1 - \diln s)^2\) to integrate by parts in
\[
z = \frac{i}{8\shift^{\nicefrac{3}{2}}} \int_{t_{j-1}}^{t_j} e^{4i\theta(s)} \frac{4i\theta'(s)}{(1 - \diln s)^3} \diff{s}
    = \frac{i}{8\shift^{\nicefrac{3}{2}}} \biggl( \frac{e^{4i\theta(s)}}{(1 - \diln s)^3} \biggr\rvert_{t_{j-1}}^{t_j} - 3 \int_{t_{j-1}}^{t_j} e^{4i\theta(s)} \frac{\diln}{(1 - \diln s)^4} \diff{s} \biggr).
\]
This can be bounded as
\[
\abs{z}
    \leq \frac{1}{8\shift^{\nicefrac{3}{2}}} \biggl( \frac{1}{(1 - \diln t_j)^3} + \frac{1}{(1 - \diln t_{j-1})^3} + 3 \int_{t_{j-1}}^{t_j} \frac{\diln}{(1 - \diln s)^4} \diff{s} \biggr)
    = \frac{1}{4 \shift^{\nicefrac{3}{2}} (1 - \diln t_j)^3}
    \eqdef \frac{\delta_j}{4}.
\]
This can be used to control the norm in~\eqref{eq.probspace.normsigmaSigma-I}, but only if \(\delta_j\) is small compared to \(\sigma^2\). To check this, note that if \(\diln t_j \in [0, 1 - \shift^{\nicefrac{-1}{2} + \alpha}]\) for some \(\alpha \in (0, \nicefrac{1}{2})\), then \((1 - \diln t_j)^{-3} \leq \shift^{\nicefrac{3}{2}-3\alpha}\) so \(\delta_j \leq \shift^{-3\alpha}\). Combining this with the bound \(\log(1 + x) \geq x(1 - \nicefrac{x}{2})\) for \(x > 0\) by which \(2\sigma^2 \geq \shift^{-p} (2 - \shift^{-p})\),
\[
\frac{\delta_j}{2\sigma^2}
    \leq \frac{\shift^{-(3\alpha - p)}}{2 - \shift^{-p}}.
\]
If we take \(p < 3\alpha\), then the right-hand side vanishes when \(\shift\to\infty\), so \(\delta_j \leq 2\sigma^2\) for any \(\shift\) large enough. Then \(\sigma^{-2}\abs{z} \leq \nicefrac{1}{2}\), and the estimate~\eqref{eq.probspace.normsigmaSigma-I} becomes \(\norm{I_2 - \sigma\Sigma_j^{\nicefrac{-1}{2}}}_2 \leq \nicefrac{\delta_j}{2\sqrt{2}\sigma^2}\). Using this with \eqref{eq.probspace.VarBtheta} in \eqref{eq.probspace.VarBtheta-W.1}, we finally conclude that
\beq{eq.probspace.VarBtheta-W}
\bexpect[\Big]{\abs[\big]{e^{-2i(\phaseN(t_{j-1}) - \theta(t_{j-1}))} \mathscr{B}^\theta_j - W_j}^2 \given[\Big] \mathscr{F}_{j-1}}
    \leq \frac{\delta_j^2}{4\sigma^2}
    = \frac{1}{4\sigma^2 \shift^3 (1 - \diln t_j)^6}
\eeq
for any \(\shift\) large enough.

Now, by definition, the increments \(e^{-2i(\phaseN(t_{j-1}) - \theta(t_{j-1}))} \mathscr{B}^\theta_j - W_j\) are centered Gaussian conditionally on their past. Therefore, the estimate~\eqref{eq.probspace.VarBtheta-W} on the conditional variances is equivalent to an estimate on the conditional subgaussian constants, and we can use this to estimate the subgaussian bracket \(\subGbracket{\Delta^{\theta W}}\) of the full martingale (which is essentially the sum of the conditional subgaussian constants---see Appendix~\ref{apx.concentration} for the full definition). Summing up~\eqref{eq.probspace.VarBtheta-W} while assuming that \(\diln t_n \in [0, 1-\shift^{\nicefrac{-1}{2}+\alpha}]\) so that the estimate applies to all increments, we get
\begin{align*}
\subGbracket{\Delta^{\theta W}}_n
    & = \sum_{j=1}^n \bexpect[\Big]{\abs[\big]{e^{-2i(\phaseN(t_{j-1}) - \theta(t_{j-1}))} \mathscr{B}^\theta_j - W_j}^2 \given[\Big] \mathscr{F}_{j-1}} \\
    & \leq \frac{1}{4\sigma^2\shift^3} \sum_{j=1}^n (1 + \shift^{-p})^{6j}
    \leq \frac{(1 + \shift^{-p})^6}{4\sigma^2\shift^3} \frac{(1 + \shift^{-p})^{6n} - 1}{\shift^{-p}}.
\end{align*}
Because \((1 + \shift^{-p})^n = (1 - \diln t_n)^{-1} \leq \shift^{\nicefrac{1}{2}-\alpha}\), this yields
\[
\subGbracket{\Delta^{\theta W}}_n
    \leq \frac{(1 + \shift^{-p})^6}{2(2 - \shift^{-p})} \shift^{2p - 6\alpha}.
\]
The prefactor converges to \(\nicefrac{1}{4}\) as \(\shift\to\infty\), so this shows that \(\subGbracket{\Delta^{\theta W}}_n \leq \shift^{-2(3\alpha-p)}\) for any \(\shift\) large enough. Applying Azuma's inequality (Theorem~\ref{thm.Azuma}) to both the real and imaginary parts of the martingale, it follows that for any \(\shift\) large enough and any \(x > 0\),
\beq{eq.probspace.DeltathetaW}
\bprob[\bigg]{\sup_{n\in S_{\shift,\alpha}} \abs{\Delta^{\theta W}_{n-1}} \geq x} \leq 4 \exp\Bigl( - \frac{\shift^{2(3\alpha-p)} x^2}{2} \Bigr)
\eeq
where \(S_{\shift,\alpha} \defeq \{n\in\mathbb{N} : 0 \leq \diln t_{n-1} \leq 1-\shift^{\nicefrac{-1}{2}+\alpha}\}\).

\step{Extending to the continuum}

Now, we extend the random walk \(\{W_j\}_{j=1}^{N-1}\) to a complex Brownian motion. To do this, extending the probability space, we define a collection \(\{X_j\}_{j=1}^{N-1}\) of independent complex Brownian bridges with \(X_j(0) = X_j(1) = 0\), taken to be independent of \(\sBM\). From these, we start by defining Brownian motions on intervals \(\bigl[ \slogtime(t_{j-1}), \slogtime(t_j) \bigr]\) where \(\slogtime(t) \defeq -\log(1-\diln t)\). Since \(W_j \sim \mathbb{C}\normal{0}{\sigma^2}\), both real and imaginary parts of \(\sigma^{-1} W_j\) are real standard normals, so given an independent random variable \(Z\), by properties of Brownian bridges, each process
\[
s \mapsto Z + \sigma X_j\Bigl( \frac{s - \slogtime(t_{j-1})}{\sigma^2} \Bigr) + \frac{s - \slogtime(t_{j-1})}{\sigma^2} W_j
\]
is a complex Brownian motion on \(\bigl[ \slogtime(t_{j-1}), \slogtime(t_j) \bigr]\) with the law of \(Z\) as its distribution at time \(\slogtime(t_{j-1})\). To build our complex Brownian motion \(W\), we can stitch the above Brownian motions together as follows. We set \(W(0) \defeq 0\), and then for \(0 < s \leq \slogtime(t_{N-1})\) we set \(k(s) \defeq \min\{j\in\mathbb{N} : s < \slogtime(t_{j+1})\}\) to be the index \(j\) such that \(\slogtime(t_j) \leq s < \slogtime(t_{j+1})\), and
\[
W(s) \defeq \sum_{j=1}^{k(s)} W_j + \sigma X_{k(s)+1}\Bigl( \frac{s - \slogtime(t_{k(s)})}{\sigma^2} \Bigr) + \frac{s - \slogtime(t_{k(s)})}{\sigma^2} W_{k(s)+1}.
\]
Finally, for \(s > \slogtime(t_{N-1})\), we simply extend again the probability space to add an independent standard complex Brownian motion \(\tilde{W}\), and we define \(W(s) \defeq W\bigl( \slogtime(t_{N-1}) \bigr) + \tilde{W}\bigl( s - \slogtime(t_{N-1}) \bigr)\). 

Since \(W\) is a complex Brownian motion on each interval \(\bigl[ \slogtime(t_{n-1}), \slogtime(t_n) \bigr]\) for \(n < N\) as well as on \([\slogtime(t_{N-1}), \infty)\), to check that it is a complex Brownian motion on \([0, \infty)\) we only have to check that it is a.s.\@ continuous at the endpoints \(\slogtime(t_n)\) of these intervals. To do this, fix \(n < N\), and note that \(k\bigl( \slogtime(t_n) \bigr) = n\) by definition, so that \(W\bigl( \slogtime(t_n) \bigr) = \sum_{j=1}^n W_j\) for each \(n < N\). It is clear that this is a.s.\@ the value approached from the right, since \(k\) does not change from that direction (or just because \(\tilde{W}\) is a.s.\@ continuous if \(n = N-1\)). With \(s \uparrow \slogtime(t_n)\), \(k(s) = k\bigl( \slogtime(t_n) \bigr) - 1 = n - 1\), and as \(\slogtime(t_n) - \slogtime(t_{n-1}) = \sigma^2\), we get that
\[
W(s) \to \sum_{j-1}^{n-1} W_j + \sigma X_n(1) + W_n = \sum_{j=1}^n W_j.
\]
Hence, \(W\) has a.s.\@ continuous paths on all of \([0, \infty)\), and we conclude that it is a standard complex Brownian motion.

\step{Full comparison of the processes}

We now have all of the necessary ingredients to obtain the estimate~\eqref{eq.probspace} on
\[
\Delta \defeq \sup_{t\in\timedom} \abs[\bigg]{\int_0^t e^{-2i\phaseN(s)} \sqrt{\frac{2\diln}{1 - \diln s}} \diff{\sBM}(s) - W\circ\slogtime(t)}.
\]
With \(S_{\shift,\alpha} \defeq \{n\in\mathbb{N} : 0\leq \diln t_{n-1} \leq 1-\shift^{\nicefrac{-1}{2}+\alpha}\}\), we can extend a little bit the interval over which the supremum is taken to write
\begin{align*}
\Delta
    & \leq \sup_{n\in S_{\shift,\alpha}} \sup_{t \in [t_{n-1}, t_n]} \abs[\bigg]{\int_0^t e^{-2i\phaseN(s)} \sqrt{\frac{2\diln}{1-\diln s}} \diff{\sBM}(s) - W\circ\slogtime(t)} \\
    & \leq \sup_{n\in S_{\shift,\alpha}} \abs[\bigg]{\int_0^{t_{n-1}} e^{-2i\phaseN(s)} \sqrt{\frac{2\diln}{1-\diln s}} \diff{\sBM}(s) - W\circ\slogtime(t_{n-1})} \\
    &\hspace*{11mm} + \sup_{n\in S_{\shift,\alpha}} \sup_{t \in [t_{n-1}, t_n]} \biggl( \abs[\Big]{\int_{t_{n-1}}^t e^{-2i\phaseN(s)} \sqrt{\frac{2\diln}{1-\diln s}} \diff{\sBM}(s)} + \abs[\big]{W\circ\slogtime(t) - W\circ\slogtime(t_{n-1})} \biggr).
\end{align*}
The integral that appears in the first line is exactly \(\sum_{j=1}^{n-1} \mathscr{B}^\xi_j\), and \(W\circ\slogtime(t_{n-1}) = \sum_{j=1}^{n-1} W_j\) by construction of \(W\). In other words, in the notation of Steps 1 and 2,
\begin{subequations}
\label{eq.probspace.supbound}
\begin{align}
\Delta
    & \leq \sup_{n\in S_{\shift,\alpha}} \abs[\big]{\Delta^{\xi\theta}_{n-1}} + \sup_{n\in S_{\shift,\alpha}} \abs[\big]{\Delta^{\theta W}_{n-1}}
    \label{eq.probspace.supbound.step2} \\
    &\hspace*{11mm} + \sup_{n\in S_{\shift,\alpha}} \sup_{t \in [t_{n-1}, t_n]} \abs[\bigg]{\int_{t_{n-1}}^t e^{-2i\phaseN(s)} \sqrt{\frac{2\diln}{1-\diln s}} \diff{\sBM}(s)} 
    \label{eq.probspace.supbound.exp} \\
    &\hspace*{11mm} + \sup_{n\in S_{\shift,\alpha}} \sup_{t \in [t_{n-1}, t_n]} \abs[\big]{W\circ\slogtime(t) - W\circ\slogtime(t_{n-1})}.
    \label{eq.probspace.supbound.CBM}
\end{align}
\end{subequations}
The first line is controlled by the results \eqref{eq.probspace.Deltaxitheta} and \eqref{eq.probspace.DeltathetaW} from Steps 1 and 2. In the other two lines, for fixed \(n\) the supremum is that of a complex martingale on \([t_{n-1}, t_n]\), which we can control by bounding the brackets of their real and imaginary parts and using Bernstein's inequality on both of them. The brackets of the real and imaginary parts of the martingale on line \eqref{eq.probspace.supbound.exp} are both bounded by
\[
\int_{t_{n-1}}^{t_n} \frac{2\diln}{1 - \diln s} \diff{s}
    = 2\log\frac{1 - \diln t_{n-1}}{1 - \diln t_n}
    = 2\log(1 + \shift^{-p})
    \leq 2\shift^{-p},
\]
so by Bernstein's inequality, for \(x > 0\), \eqref{eq.probspace.supbound.exp} is bounded by \(x\) with probability at least \(1 - 4N\exp(\nicefrac{-\shift^{p}x^2}{4})\). Likewise, the brackets of the real and imaginary parts of the martingale on line~\eqref{eq.probspace.supbound.CBM} are bounded by \(\slogtime(t_n) - \slogtime(t_{n-1}) \leq \shift^{-p}\), so Bernstein's inequality yields a similar bound. Combining this with the results~\eqref{eq.probspace.Deltaxitheta} and~\eqref{eq.probspace.DeltathetaW} from Steps 1 and 2, we get from \eqref{eq.probspace.supbound} that for \(\shift\) large enough and \(x > 0\),
\begin{multline*}
\bprob{\Delta \geq x}
    \leq 2\shift^{2p} \log^2\shift \exp\biggl( - \frac{1}{2} \min\Bigl\{ \Bigl( \frac{\shift^{\nicefrac{p}{2}} x\log 2}{48\sqrt{2}} \Bigr)^{\nicefrac{2}{3}}, \frac{\shift^p x^2 \log 2}{144 C_\beta \log\shift} \Bigr\} \biggr) \\
    + 4\exp\biggl( - \frac{x\log 2}{12\sqrt{2C_\beta} \log\shift} \exp\biggl( \frac{1}{4} \Bigl( \frac{\shift^{\nicefrac{p}{2}} x\log 2}{48\sqrt{2}} \Bigr)^{\nicefrac{2}{3}} \biggr) \biggr)
    + 4\exp\Bigl( - \frac{\shift^{2(3\alpha-p)} x^2}{32} \Bigr)
    + 8N\exp\Bigl( - \frac{\shift^p x^2}{64} \Bigr).
\end{multline*}
Notice that the powers of \(\shift\) and \(x\) are balanced if we set \(p \defeq 2\alpha\). With this value for \(p\), if \(x = \shift^{-\alpha+\delta}\), then all exponents are at least of order \(\shift^{\nicefrac{2\delta}{3}}\) for \(\shift\) large enough, and we conclude that
\[
\bprob{\Delta \geq \shift^{-\alpha+\delta}} \leq 3\shift^{4\alpha} \log^2\shift \exp\bigl( - C\shift^{\nicefrac{2\delta}{3}} \bigr),
\]
where \(C > 0\) depends only on \(\beta\), \(\alpha\) and \(\delta\).

\step{Combining all spaces together}

The above construction yields, for a given \(\shift > 0\), a probability space on which are defined a standard real Brownian motion \(\sBM\) and a standard complex Brownian motion \(W\) such that \eqref{eq.probspace} holds if \(\shift\) is large enough. Thus, given a sequence \(\{\shift_n\}_{n\in\mathbb{N}} \subset (0,\infty)\) such that \(\shift_n\to\infty\), it remains to combine all the spaces together in such a way that the couplings between the real and complex Brownian motions are preserved. For the sake of completeness, we build the full probability space explicitly, but this construction is standard.

Let \((\Omega_n, \mathscr{F}_n, \mathbb{P}_n)\) and \(W_n\) denote the probability space and the complex Brownian motion built above for \(\shift = \shift_n\)
\setshift{E_n}
and a standard real Brownian motion \(\sBM\). As the initial Brownian motion \(\sBM\) is arbitrary, we can choose it to be defined as a random variable \(\sBM \colon \Omega_n \to \mathscr{C}\bigl( [0,\infty), \mathbb{R} \bigr)\), using the construction of Brownian motion on a Wiener space. In the same way, taking the Brownian bridges used to build \(W_n\) with continuous paths, \(W_n\) has surely continuous paths, and \(W_n \colon \Omega_n \to \mathscr{C}\bigl( [0,\infty), \mathbb{C} \bigr)\) where the space \(\mathscr{C}\bigl( [0,\infty), \mathbb{C} \bigr)\) is endowed with a product of the Wiener measure with itself. Using these constructions, on each space \((\Omega_n, \mathscr{F}_n, \mathbb{P}_n)\) we have two Brownian motions \(\sBM\) and \(W_n\) with joint law \(\mathbb{P}_{W_n,\sBM}\) on \(\mathscr{C}\bigl( [0,\infty), \mathbb{C} \bigr) \times \mathscr{C}\bigl( [0,\infty), \mathbb{R} \bigr)\) with the completion of its Borel \sigmaalgebra{} under an appropriate metric that makes it complete and separable (and such a metric does exist, see e.g.\@ \cite[\S 2.4]{karatzas_brownian_1991}). This measurable space is thus Borel, and there exists a regular condition distribution \(\rho_{\sBM\given W_n}\colon \mathscr{C}\bigl( [0,\infty), \mathbb{C} \bigr) \times \mathscr{B}\bigl( \mathscr{C}([0,\infty), \mathbb{R}) \bigr) \to [0,1]\) such that \(\mathbb{P}_{W_n,\sBM} = \mathbb{P}_{W_n} \otimes \rho_{\sBM\given W_n}\). 

Now, set \(\Omega \defeq \mathscr{C}\bigl( [0,\infty), \mathbb{C} \bigr) \times \mathscr{C}\bigl( [0,\infty), \mathbb{R} \bigr)^{\mathbb{N}}\) and let \(\mathscr{F} \defeq \mathscr{B}(\Omega)\) be its Borel \sigmaalgebra{}. Then, define kernels \(\mu_n\) on the partial products \(\mathscr{C}\bigl( [0,\infty), \mathbb{C} \bigr) \times \prod_{j=1}^n \mathscr{C}\bigl( [0,\infty), \mathbb{R} \bigr)\) by setting \(\mu_n(w, b_1, \hdots, b_{n-1}, A) \defeq \rho_{\sBM\given W_n}(w,A)\). By the Ionescu--Tulcea theorem, there exists a measure \(\mathbb{P}\) on \((\Omega, \mathscr{F})\) such that for \(A \in \mathscr{B}\bigl( \mathscr{C}([0,\infty), \mathbb{C}) \times \mathscr{C}([0,\infty), \mathbb{R})^n \bigr)\),
\[
\pprob[\Big]{A \times \prod_{j=n+1}^\infty \Omega_j}
    = \int_A \Bigl( \prod_{j=1}^n \rho_{\sBM\given W_n}(w, \diff{b}_n) \Bigr) \mathbb{P}_{W_n}(\diff{w}),
\]
and in particular if \(A \in \mathscr{B}\bigl( \mathscr{C}([0,\infty), \mathbb{C}) \bigr)\) and \(B \in \mathscr{B}\bigl( \mathscr{C}([0,\infty), \mathbb{R}) \bigr)\), then
\[
\pprob[\big]{A \times \Omega_1 \times \cdots \times \Omega_{n-1} \times B \times \Omega_{n+1} \times \cdots}
    = \int_A \int_B \rho_{\sBM\given W_n}(w, \diff{b}) \mathbb{P}_{W_n}(\diff{w})
    = \mathbb{P}_{W_n, \sBM}(A \times B).
\]
Therefore, for any \(n\in\mathbb{N}\) the projection onto coordinates \(1\) and \(n+1\) has the desired distribution for the pair \((W, \sBM)\), and all of the Brownian motions we need are well defined on the space \((\Omega, \mathscr{F}, \mathbb{P})\).
\end{proof}

\section{Asymptotic behavior of some processes}
\label{sec.asymptotics}

This section is devoted to the next step of the proof of Theorem~\ref{thm.CSconvinlaw}. Now that we have an appropriate probability space on which to work with specific real and complex Brownian motions \(\sBM\) and \(W\), we wish to compare the process \(-\exp(-\diffamps - i\diffphases)\) that appears in the first term of the time-changed shifted Airy system's coefficient matrix as written in~\eqref{eq.sAirymatpolar} with a hyperbolic Brownian motion driven by \(W\). We do this separately for the real and imaginary parts of the processes. We start with the imaginary parts, as these are simpler for two reasons. First, because the imaginary part of a hyperbolic Brownian motion driven by \(W\) is a geometric Brownian motion driven by \(\Im W\), and second because by the Wronskian identity~\eqref{eq.ampsphasesWronskianidentity}, the imaginary part of \(-\exp(-\diffamps - i\diffphases)\) can be written in terms of \(\AMPN\) only.

For most of this section, we will work on the probability space \((\Omega, \mathscr{F}, \mathbb{P})\) from Lemma~\ref{lem.probspace}, and therefore we restrict the shift \(\shift\) to be part of some diverging sequence \(\{\shift_n\}_{n\in\mathbb{N}}\). To simplify notation, we will drop the \(n\) subscript and simply talk about limits as \(\shift\to\infty\), although in this context these should be understood as limits taken along a discrete sequence. 

When working on that space, we will still use \(\AMPN\), \(\AMPD\), \(\phaseN\) and \(\phaseD\) to denote the solutions to the SDEs from Proposition~\ref{prop.polarcoords} with \(\AMPN(0) = \AMPD(0) = 0\), \(\phaseD(0) = 0\) and \(\phaseN(0) = \nicefrac{\pi}{2}\), but now these solutions will always be taken to be driven by the corresponding Brownian motion \(\sBM\) from Lemma~\ref{lem.probspace}.

\subsection{Averaging of integrals with oscillatory integrands}

The first step is to control certain integrals whose integrands oscillate quickly, with phases proportional to the phase coordinates of solutions to \(\sAiryop f = 0\), which causes them to average out when \(\shift\) is large.

\begin{lemma}
\label{lem.averaging}
Let \(\phase\) be a solution to the SDE from Proposition~\ref{prop.polarcoords}, driven by a standard Brownian motion \(B\). Let \(X_{\shift}\) solve
\[
\diff{X_{\shift}}(t) = a_{\shift}(t) X_{\shift}(t) \frac{\diln}{1-\diln t} \diff{t} + b_{\shift}(t) X_{\shift}(t) \sqrt{\frac{\diln}{1 - \diln t}} \diff{B}(t)
\]
on \([0, \lasttime]\), where \(a_{\shift}\) and \(b_{\shift}\) are complex stochastic processes on \([0,\lasttime]\). Fix \(T \in (0, \lasttime]\) and a nonzero \(k \in \mathbb{R}\), and suppose that \(a_{\shift}\) and \(b_{\shift}\) are bounded on \([0,T]\) by constants \(m_a, m_b > 0\) independent of \(\shift\). Then, there are constants \(C, C' > 0\) depending only on \(\beta\), \(m_a\), \(m_b\) and \(k\) such that for any \(M, x > 0\),
\begin{multline*}
\pprob[\bigg]{\biggl\{ \sup_{t\in[0,T]} \abs{X_{\shift}} \leq M \biggr\} \cap \biggl\{ \sup_{t\in[0,T]} \abs[\Big]{\int_0^t X_{\shift}(s) e^{ki\phase(s)} \frac{\diln}{1-\diln s} \diff{s}} \geq x + \frac{CM}{\shift^{\nicefrac{3}{2}}(1-\diln T)^3} \biggr\}} \\
    \leq 4\exp\Bigl( - \frac{C' \shift^3(1-\diln T)^6 x^2}{M^2} \Bigr).
\end{multline*}
\end{lemma}

\begin{remark}
The exponential structure of the SDE is in no way essential to this result: it would be straightforward to extend the proof given below to processes satisfying more complicated SDEs with the factors \(a_{\shift} X_{\shift}\) and \(b_{\shift} X_{\shift}\) replaced by other functions that are bounded on compacts. However, the proof is a bit simpler with this exponential structure, and this is all we will need here.
\end{remark}

\begin{proof}
Recall that \(\phase\) satisfies
\[
\diff{\phase}(t) = -2\diln\shift^{\nicefrac{3}{2}} (1 - \diln t)^2 \diff{t} + \phasedrift(t) \frac{\diln}{1 - \diln t} \diff{t} + \phasediff(t) \sqrt{\frac{\diln}{1 - \diln t}} \diff{B}(t)
\]
where
\[
\phasedrift \defeq - \Bigl( \frac{2}{\beta} - \frac{1}{2} \Bigr) \sin 2\phase - \frac{1}{\beta} \sin 4\phase
\qquadtext{and}
\phasediff \defeq \frac{2\sqrt{2}}{\sqrt{\beta}} \cos^2 \phase.
\]
An application of Itô's formula then shows that
\begin{multline*}
\diff{\bigl( e^{ki\phase} \bigr)}(t)
    = e^{ki\phase(t)} \biggl( -2ki\diln \shift^{\nicefrac{3}{2}} (1 - \diln t)^2 \diff{t} + ki \phasediff(t) \sqrt{\frac{\diln}{1 - \diln t}} \diff{B}(t) \\
    + \Bigl( ki\phasedrift(t) - \frac{k^2}{2} \phasediff(t)^2 \Bigr) \frac{\diln}{1 - \diln t} \diff{t} \biggr).
\end{multline*}
Thus, for \(t \in [0,T]\),
\beq{eq.averaging.1}
\begin{aligned}
\int_0^t X_{\shift}(s) e^{ki\phase(s)} \frac{\diln}{1 - \diln s} \diff{s}
    & = - \frac{1}{2ik\shift^{\nicefrac{3}{2}}} \int_0^t \frac{X_{\shift}(s)}{(1 - \diln s)^3} \biggl( \diff{\bigl( e^{ki\phase} \bigr)}(s) \\
    &\hspace*{22mm} - ki \phasediff(s) e^{ki\phase(s)} \sqrt{\frac{\diln}{1 - \diln s}} \diff{B}(s) \\
    &\hspace*{22mm} - e^{ki\phase(s)} \Bigl( ki\phasedrift(s) - \frac{k^2}{2} \phasediff(s)^2 \Bigr) \frac{\diln}{1 - \diln s} \diff{s} \biggr).
\end{aligned}
\eeq
If \(Y_{\shift}(t) \defeq (1 - \diln t)^{-3} X_{\shift}(t)\), then
\[
\diff{Y_{\shift}}(t)
    = X_{\shift}(t) \bigl( 3 + a_{\shift}(t) \bigr) \frac{\diln}{(1 - \diln t)^4} \diff{t} + X_{\shift}(t) b_{\shift}(t) \frac{\sqrt{\diln}}{(1 - \diln t)^{\nicefrac{7}{2}}} \diff{B}(t),
\]
which also implies that
\[
\diff{\crossvar{Y_{\shift}}{e^{ki\phase}}}(t)
    = ki X_{\shift}(t) e^{ki\phase(t)} \phasediff(t) b_{\shift}(t) \frac{\diln}{(1 - \diln t)^4} \diff{t}.
\]
From this, the first term of~\eqref{eq.averaging.1} can be integrated by parts:
\beq{eq.averaging.2}
\begin{aligned}
\int_0^t & X_{\shift}(s) e^{ki\phase(s)} \frac{\diln}{1 - \diln s} \diff{s} \\
    & = - \frac{1}{2ik\shift^{\nicefrac{3}{2}}} \biggl( \frac{X_{\shift}(s) e^{ki\phase(s)}}{(1 - \diln s)^3} \biggr\vert_0^t 
    - \int_0^t X_{\shift}(s) e^{ki\phase(s)} \Bigl( ki\phasediff(s) + b_{\shift}(s) \Bigr) \frac{\sqrt{\diln}}{(1 - \diln s)^{\nicefrac{7}{2}}} \diff{B}(s) \\
    &\hspace*{11mm} - \int_0^t X_{\shift}(s) e^{ki\phase(s)} \Bigl( ki\phasedrift(s) - \frac{k^2}{2} \phasediff(s)^2 + 3 + a_{\shift}(s) + ki\phasediff(s) b_{\shift}(s) \Bigr) \frac{\diln}{(1 - \diln s)^4} \diff{s} \biggr).
\end{aligned}
\eeq
Given \(M > 0\), each of the three terms can easily be bounded on the event \(\mathscr{G} \defeq \bigl\{ \sup_{t\in[0,T]} \abs{X_{\shift}(t)} \leq M \bigr\}\). The first term is bounded on \(\mathscr{G}\) by
\[
\frac{M}{2\abs{k}\shift^{\nicefrac{3}{2}}} \Bigl( \frac{1}{(1 - \diln T)^3} + 1 \Bigr) \leq \frac{M}{\abs{k}\shift^{\nicefrac{3}{2}}(1 - \diln T)^3}.
\]
Then, since \(\abs{\phasedrift} \leq \nicefrac{3}{\beta} + \nicefrac{1}{2}\) and \(\abs{\phasediff} \leq \nicefrac{2\sqrt{2}}{\sqrt{\beta}}\), there is a constant \(\tilde{C} > 0\) depending only on \(\beta\), \(k\), \(m_a\) and \(m_b\) such that the third term of \eqref{eq.averaging.2} is bounded on \(\mathscr{G}\) by
\[
\frac{\tilde{C}M}{\shift^{\nicefrac{3}{2}}} \int_0^t \frac{\diln}{(1 - \diln s)^4} \diff{s}
    \leq \frac{\tilde{C}M}{3\shift^{\nicefrac{3}{2}}(1 - \diln T)^3}.
\]
Finally, there is another constant \(\tilde{C}' > 0\) depending only on \(\beta\), \(k\) and \(m_b\) such that, still on \(\mathscr{G}\), the brackets of the real and imaginary parts of the second term of \eqref{eq.averaging.2} are bounded by
\[
\frac{\tilde{C}'M^2}{\shift^3} \int_0^t \frac{\diln}{(1 - \diln s)^7} \diff{s}
    \leq \frac{\tilde{C}'M^2}{6\shift^3 (1 - \diln T)^6}.
\]
As this holds for any \(t \in [0,T]\), applying Bernstein's inequality for martingales to both the real and the imaginary parts of the stochastic integral shows that for any \(x > 0\),
\begin{multline*}
\pprob[\bigg]{\mathscr{G} \cap \biggl\{ \sup_{t\in [0,T]} \abs[\bigg]{\frac{1}{2ik\shift^{\nicefrac{3}{2}}} \int_0^t X_{\shift}(s) e^{ki\phase(s)} \bigl( ki\phasediff(s) + b_{\shift}(s) \bigr) \frac{\sqrt{\diln}}{(1 - \diln s)^{\nicefrac{7}{2}}} \diff{B}(s)} \geq x \biggr\}} \\
    \leq 4\exp\biggl( - \frac{3\shift^3 (1 - \diln T)^6 x^2}{\tilde{C}'M^2} \biggr).
\end{multline*}
Combining the bounds on the three terms of \eqref{eq.averaging.2} yields the announced result with \(C \defeq \nicefrac{\tilde{C}}{3} + \nicefrac{1}{\abs{k}}\) and \(C' \defeq \nicefrac{3}{\tilde{C}'}\).
\end{proof}

This result has two immediate corollaries.

\begin{corollary}
\label{cor.averaging.bounded}
If \(k \neq 0\) and \(\alpha \in (0, \nicefrac{1}{2})\), there are \(C, C' > 0\) depending only on \(\beta\) and \(k\) such that for any \(x > 0\),
\[
\bprob[\bigg]{\sup_{t\in\timedom} \abs[\Big]{\int_0^t e^{ki\phase(s)} \frac{\diln}{1 - \diln s} \diff{s}} \geq x + C\shift^{-3\alpha}}
    \leq 4e^{-C'\shift^{6\alpha} x^2}
\]
where \(\timedom \defeq [0, (1 - \shift^{\nicefrac{-1}{2}+\alpha}) / \diln]\).
\end{corollary}

\begin{proof}
We apply the lemma with \(X_{\shift} \equiv 1\) so that \(a_{\shift} = b_{\shift} = 0\), and with \(\diln T \defeq 1 - \shift^{\nicefrac{-1}{2}+\alpha}\). As in that case \(\bprob[\big]{\sup_{t\in[0,T]} \abs{X_{\shift}(t)} > 1} = 0\), we can take \(M = 1\) to get that there are \(C, C' > 0\) such that for any \(x > 0\),
\[
\bprob[\bigg]{\sup_{t\in [0,T]} \abs[\Big]{\int_0^t e^{ki\phase(s)} \frac{\diln}{1 - \diln s} \diff{s}} \geq x + \frac{C}{\shift^{\nicefrac{3}{2}}(1 - \diln T)^3}}
    \leq 4e^{-C' \shift^3 (1 - \diln T)^6 x^2}.
\]
This is exactly the announced result since \(\shift^{\nicefrac{3}{2}}(1 - \diln T)^3 = \shift^{3\alpha}\).
\end{proof}

\begin{corollary}
\label{cor.averaging.integrable}
If \(k \neq 0\), then for every \(\varepsilon > 0\) there is a \(C_{\varepsilon} > 0\) depending only on \(\beta\), \(k\) and \(\varepsilon\) such that
\[
\bprob[\bigg]{\sup_{t\in [0,\lasttime]} \abs[\Big]{\int_0^t e^{ki\phase(s)} \frac{\diln}{1 - \diln s} \diff{s}} > C_{\varepsilon}} < \varepsilon.
\]
\end{corollary}

\begin{proof}
Applying the lemma with \(X_{\shift} \equiv 1\) so that \(a_{\shift} = b_{\shift} = 0\) and with \(M = 1\) and \(T = \lasttime\), we get that for some \(C, C' > 0\) depending only on \(\beta\) and \(k\), for any \(x > 0\),
\[
\bprob[\bigg]{\sup_{t\in [0,\lasttime]} \abs[\Big]{\int_0^t e^{ki\phase(s)} \frac{\diln}{1 - \diln s} \diff{s}} > x + C} \leq 4 e^{-C' x^2},
\]
and taking \(x\) large enough so that \(4e^{-C'x^2} < \varepsilon\) yields the result with \(C_{\varepsilon} \defeq x + C\).
\end{proof}

\subsection{The geometric Brownian motion}

We are now ready to tackle the convergence of the process \(e^{-2\AMPN}\) to a geometric Brownian motion driven by \(\Im W\).

\begin{proposition}
\label{prop.GBM}
On the probability space from Lemma~\ref{lem.probspace}, if \(\alpha \in (0,\nicefrac{1}{2})\) and \(\delta \in (0,\alpha)\), there is a \(C > 0\) such that for all \(\shift \in \{\shift_n\}_{n\in\mathbb{N}}\) large enough,
\[
\bprob[\bigg]{\sup_{t\in\timedom} \abs[\Big]{2\AMPN(t) + \log\GBM(t)} \geq \shift^{-\alpha+\delta}}
    \leq 4 \shift^{4\alpha} \log^2\shift \exp\bigl( - C\shift^{\nicefrac{2\delta}{3}} \bigr)
\]
where \(\timedom \defeq [0, (1-\shift^{\nicefrac{-1}{2}+\alpha}) / \diln]\) and \(\GBM(t) \defeq \exp\bigl( \frac{2}{\sqrt{\beta}} \Im W\circ\slogtime(t) - \frac{2}{\beta} \slogtime(t) \bigr)\).
\end{proposition}

\begin{proof}
By definition of \(\AMPN\),
\begin{align*}
&\sup_{t\in\timedom} \abs[\Big]{2\AMPN(t) + \frac{2}{\sqrt{\beta}} \Im W\circ\slogtime(t) - \frac{2}{\beta} \slogtime(t)} \\
    &\hspace*{11mm} \leq \sup_{t\in\timedom} \abs[\bigg]{\frac{2}{\sqrt{\beta}} \int_0^t \sin 2\phaseN(s) \sqrt{\frac{2\diln}{1 - \diln s}} \diff{\sBM}(s) + \frac{2}{\sqrt{\beta}} \Im W\circ\slogtime(t)} \\
    &\hspace*{22mm} + \frac{\abs{4-\beta}}{\beta} \sup_{t\in\timedom} \abs[\bigg]{\int_0^t \cos 2\phaseN(s) \frac{\diln}{1 - \diln s} \diff{s}}
    + \frac{2}{\beta} \sup_{t\in\timedom} \abs[\bigg]{\int_0^t \cos 4\phaseN(s) \frac{\diln}{1 - \diln s} \diff{s}}.
\end{align*}
The first supremum can immediately be bounded by \(\frac{1}{3}\shift^{-\alpha+\delta}\) on a good event by Lemma~\ref{lem.probspace}, simply by taking the imaginary part in \eqref{eq.probspace}. The other two can each be bounded by \(\frac{1}{3} \shift^{-\alpha+\delta}\) on a good event by applying Corollary~\ref{cor.averaging.bounded}. The result then follows by combining the tail bounds, which are dominated by that from Lemma~\ref{lem.probspace}. 
\end{proof}

Using properties of Brownian motion, we can take the exponential and compare \(e^{-2\AMPN}\) and \(e^{2\AMPN}\) to the geometric Brownian motion \(\GBM\) and its reciprocal.

\begin{corollary}
\label{cor.GBM.exp}
In the setting of the proposition, for any \(\delta \in (0,\alpha)\) and any \(\shift\) large enough, 
\beq{eq.GBM.exp.+}
\bprob[\bigg]{\sup_{t\in\timedom} \abs[\Big]{e^{-2\AMPN(t)} - \GBM(t)} \geq \shift^{-\alpha+\delta}} \leq \frac{2}{\log\shift},
\eeq
and if \(\alpha > \frac{1}{\beta+2}\) and \(\delta < \frac{1}{\beta}\bigl( \alpha(\beta+2) - 1 \bigr)\), then for any \(\shift\) large enough,
\beq{eq.GBM.exp.-}
\bprob[\bigg]{\sup_{t\in\timedom} \abs[\Big]{e^{2\AMPN(t)} - \frac{1}{\GBM(t)}} \geq \exp\Bigl( - \frac{\alpha(\beta+2) - \beta\delta - 1}{\beta} \log\shift \Bigr)}
    \leq 2\exp\Bigl( - \frac{\beta\delta^2}{36(1-2\alpha)} \log\shift \Bigr).
\eeq
\end{corollary}

\begin{proof}
Since
\[
\abs[\Big]{e^{\mp 2\AMPN} - \GBM^{\pm 1}}
    \leq \GBM^{\pm 1} \abs[\big]{2\AMPN + \log\GBM} e^{\abs{2\AMPN + \log\GBM}},
\]
given the proposition, what remains to control here is only the supremum of the geometric Brownian motion \(\GBM\) and that of its reciprocal. Recall that
\[
\GBM^{\pm 1}(t)
    = \exp\Bigl( \pm \frac{2}{\sqrt{\beta}} \Im W\circ\slogtime(t) \mp \frac{2}{\beta} \slogtime(t) \Bigr).
\]
We are interested in the supremum of this process over \(\diln t \in [0, 1 - \shift^{\nicefrac{-1}{2}+\alpha}]\), which can be understood as the supremum over \(s = \slogtime(t) \in [0, (\frac{1}{2} - \alpha)\log\shift]\). Now, from the joint density of a Brownian motion and its running maximum, an application of Girsanov's theorem shows that for \(y \geq 0\) (see e.g.\@ \cite[Part~II, \S 2.1, formula~1.1.4]{borodin_handbook_2002} for the precise statement),
\beq{eq.GBM.driftedBMbound}
\begin{multlined}
\bprob[\bigg]{\sup_{s\in [0,(\frac{1}{2}-\alpha)\log\shift]} \Bigl( \pm \frac{2}{\sqrt{\beta}} \Im W(s) \mp \frac{2s}{\beta} \Bigr) \geq y}
    = \frac{1}{2} \biggl( 1 - \erf\Bigl( \frac{1}{2} \Bigl( y \sqrt{\frac{\beta}{(1-2\alpha)\log\shift}} \pm \sqrt{\frac{(1-2\alpha)\log\shift}{\beta}} \Bigr) \Bigr) \biggr) \\
    + \frac{e^{\mp y}}{2} \biggl( 1 - \erf\Bigl( \frac{1}{2} \Bigl( y \sqrt{\frac{\beta}{(1-2\alpha)\log\shift}} \mp \sqrt{\frac{(1-2\alpha)\log\shift}{\beta}} \Bigr) \Bigr) \biggr).
\end{multlined}
\eeq

We can bound this in the two cases. First, for the top sign, the argument of the first error function is nonnegative so we can use the fact that \(1 - \erf x \leq e^{-x^2}\) for \(x \geq 0\), and we can simply bound the second one with the basic property \(\frac{1}{2}(1 - \erf) \leq 1\). We thus obtain for \(y = \log\log\shift\) and \(\shift\) large enough that
\begin{multline*}
\bprob[\bigg]{\sup_{s\in [0,(\frac{1}{2}-\alpha)\log\shift]} \Bigl( \frac{2}{\sqrt{\beta}} \Im W(s) - \frac{2s}{\beta} \Bigr) \geq \log\log\shift} \\
    \leq \frac{1}{2} \exp\biggl( - \frac{1}{4} \Bigl( \log\log\shift \sqrt{\frac{\beta}{(1-2\alpha)\log\shift}} + \sqrt{\frac{(1-2\alpha)\log\shift}{\beta}} \Bigr)^2 \biggr) + \exp\Bigl( - \log\log\shift \Bigr) 
    \leq \frac{1}{2\shift^{\nicefrac{(1-2\alpha)}{4\beta}}} + \frac{1}{\log\shift}.
\end{multline*}
Exponentiating, we get that for \(\shift\) large enough,
\beq{eq.boundsupGBM}
\bprob[\bigg]{\sup_{t\in\timedom} \GBM(t) \geq \log\shift}
    \leq \frac{1}{2\shift^{\nicefrac{(1-2\alpha)}{4\beta}}} + \frac{1}{\log\shift}.
\eeq
On the intersection of the complement of this event with the complement of the event in the proposition with \(\delta\) replaced with \(\nicefrac{\delta}{2}\),
\[
\sup_{t\in\timedom} \abs[\big]{e^{-2\AMPN(t)} - \GBM(t)}
    \leq \shift^{-\alpha+\nicefrac{\delta}{2}} \exp\bigl( \shift^{-\alpha+\nicefrac{\delta}{2}} \bigr) \log\shift
    \leq \shift^{-\alpha+\delta}
\]
for \(\shift\) large enough, and combining the tail bounds (which are dominated by the \(\nicefrac{1}{\log\shift}\) term) yields \eqref{eq.GBM.exp.+}.

With the bottom sign in \eqref{eq.GBM.driftedBMbound}, we use again the bound \(1 - \erf x \leq e^{-x^2}\) valid for \(x \geq 0\), but now on both error functions, so this is only valid for \(y \geq \frac{1-2\alpha}{\beta}\log\shift\). This yields
\[
\bprob[\bigg]{\sup_{s\in [0,(\frac{1}{2}-\alpha)\log\shift]} \Bigl( - \frac{2}{\sqrt{\beta}} \Im W(s) + \frac{2s}{\beta} \Bigr) \geq y}
    \leq \exp\biggl( -\frac{1}{4} \Bigl( y \sqrt{\frac{\beta}{(1-2\alpha)\log\shift}} - \sqrt{\frac{(1-2\alpha)\log\shift}{\beta}} \Bigr)^2 \biggr).
\]
Taking \(y = \bigl( \frac{1-2\alpha}{\beta} + \frac{\delta}{3} \bigr) \log\shift\), we get
\[
\bprob[\bigg]{\sup_{s\in [0,(\frac{1}{2}-\alpha)\log\shift]} \Bigl( - \frac{2}{\sqrt{\beta}} \Im W(s) + \frac{2s}{\beta} \Bigr) \geq \Bigl( \frac{1 - 2\alpha}{\beta} + \frac{\delta}{3} \Bigr) \log\shift}
    \leq \exp\Bigl( - \frac{\beta\delta^2}{36(1-2\alpha)} \log\shift \Bigr).
\]
On the intersection of the complement of this event with the complement of the event in the proposition with \(\delta\) replaced with \(\nicefrac{\delta}{3}\),
\[
\sup_{t\in\timedom} \abs[\Big]{e^{2\AMPN(t)} - \frac{1}{\GBM(t)}}
    \leq \exp\biggl( \shift^{-\alpha+\nicefrac{\delta}{3}} + \Bigl( - \alpha + \frac{\delta}{3} + \frac{1-2\alpha}{\beta} + \frac{\delta}{3} \Bigr) \log\shift \biggr)
\]
for \(\shift\) large enough. If \(\alpha > \frac{1}{\beta+2}\), then taking \(\delta < \frac{1}{\beta}\bigl( \alpha(\beta+2) - 1\bigr)\) ensures that the exponent is negative, and combining the tail bounds (which are dominated by the negative power of \(\shift\)) gives \eqref{eq.GBM.exp.-} for \(\shift\) large enough.
\end{proof}

\subsection{The real part of the hyperbolic Brownian motion}

In the last section we saw that the process \(e^{-2\AMPN}\) becomes a geometric Brownian motion driven by \(\Im W\) as \(\shift\to\infty\). Now, we use this to prove the following result, hence finishing to compare \(-\exp(-\diffamps-i\diffphases)\) with a hyperbolic Brownian motion driven by \(W\).

\begin{proposition}
\label{prop.ReHBM}
On the probability space from Lemma~\ref{lem.probspace}, for any \(\alpha \in (0, \nicefrac{1}{2})\) and \(\delta \in (0, \nicefrac{\alpha}{4})\), there is a \(C > 0\) such that for any \(\shift \in \{\shift_n\}_{n\in\mathbb{N}}\) large enough, 
\[
\bprob[\bigg]{\sup_{t\in\timedom} \abs[\Big]{-e^{-\diffamps(t)}\cos\diffphases(t) - \frac{2}{\sqrt{\beta}} \int_0^t \GBM(s) \diff{(\Re W\circ\slogtime)}(s)} \geq \shift^{\nicefrac{-\alpha}{4}+\delta}} \leq \frac{C}{\log\shift}
\]
where \(\timedom \defeq [0, (1 - \shift^{\nicefrac{-1}{2}+\alpha}) / \diln]\).
\end{proposition}

\begin{proof}
We can get an SDE for \(-e^{-\diffamps}\cos\diffphases\) by taking the real part of the SDE for \(-\exp\bigl( -\diffamps - i\diffphases \bigr)\) given in \eqref{eq.dalmostHBM}. As the Wronskian identity \eqref{eq.ampsphasesWronskianidentity} implies that \(e^{-\diffamps} \sin\diffphases = e^{-2\AMPN}\), this SDE can be written as
\beq{eq.dRealmostHBM}
\begin{aligned}
\diff{\bigl( -e^{-\diffamps} \cos\diffphases \bigr)}(t)
    & = \frac{2}{\sqrt{\beta}} e^{-2\AMPN(t)} \cos 2\phaseN(t) \sqrt{\frac{2\diln}{1 - \diln t}} \diff{\sBM}(t) \\
    & \hspace*{22mm} - e^{-2\AMPN(t)} \biggl( \Bigl( \frac{4}{\beta} - 1 \Bigr) \sin 2\phaseN(t) + \frac{4}{\beta} \sin 4\phaseN(t) \biggr) \frac{\diln}{1 - \diln t} \diff{t}.
\end{aligned}
\eeq
To prove the proposition, we proceed as follows. First, we find a good event on which the growth of \(e^{-2\AMPN}\) can be controlled. From this, we can deduce that the second line in \eqref{eq.dRealmostHBM} vanishes when \(\shift\to\infty\), thus reducing the problem to comparing the first line in \eqref{eq.dRealmostHBM} to the real part of the hyperbolic Brownian motion driven by \(W\). To do this, we discretize the time interval as in the proof of Lemma~\ref{lem.probspace} in order to reduce the problem to a comparison between two discrete time martingales.

\setcounter{step}{0}
\step{Finding a good event on which to control the growth of integrands}

Let
\beq{eq.ReHBM.goodevent}
\mathscr{G} \defeq \biggl\{ \sup_{t\in\timedom} \GBM(t) \leq \log\shift \biggr\} \cap \biggl\{ \sup_{t\in\timedom} \abs[\big]{e^{-2\AMPN}(t) - \GBM(t)} \leq \shift^{-\alpha+\delta} \biggr\}.
\eeq
We have seen in~\eqref{eq.boundsupGBM} in the proof of Corollary~\ref{cor.GBM.exp} and in \eqref{eq.GBM.exp.+} in the corollary itself that the complements of both of these events have probability at most \(\nicefrac{2}{\log\shift}\) for \(\shift\) large enough, so \(\pprob{\mathscr{G}^\complement} \leq \nicefrac{4}{\log\shift}\). Therefore, in order to prove the proposition, it suffices to prove that
\[
\pprob[\bigg]{\mathscr{G} \cap \biggl\{ \sup_{t\in\timedom} \abs[\Big]{-e^{-\diffamps(t)}\cos\diffphases(t) - \frac{2}{\sqrt{\beta}} \int_0^t \GBM(s) \diff{(\Re W\circ\slogtime)}(s)} \geq \shift^{\nicefrac{-\alpha}{4}+\delta} \biggr\}} \lesssim \frac{1}{\log\shift}.
\]

\step{Controlling the oscillatory terms}

We now show that the second line in \eqref{eq.dRealmostHBM} does not contribute. An application of Itô's formula (or just taking the imaginary part of~\eqref{eq.dalmostHBM} and using \(e^{-\diffamps} \sin\diffphases = e^{-2\AMPN}\)) shows that
\begin{multline*}
\diff{\bigl( e^{-2\AMPN} \bigr)}(t)
    = - \frac{2\sqrt{2}}{\sqrt{\beta}} e^{-2\AMPN(t)} \sin 2\phaseN(t) \sqrt{\frac{\diln}{1 - \diln t}} \diff{\sBM}(t) \\
    - e^{-2\AMPN(t)} \biggl( \Bigl( \frac{4}{\beta} - 1 \Bigr) \cos 2\phaseN(t) + \frac{4}{\beta} \cos 4\phaseN(t) \biggr) \frac{\diln}{1 - \diln t} \diff{t}.
\end{multline*}
This SDE has precisely the form appearing in Lemma~\ref{lem.averaging} with
\[
a_{\shift} = - \Bigl( \frac{4}{\beta} - 1 \Bigr) \cos 2\phaseN - \frac{4}{\beta} \cos 4\phaseN
\qquadtext{and}
b_{\shift} = - \frac{2\sqrt{2}}{\beta} \sin 2\phaseN,
\]
which are both bounded processes with bounds depending only on \(\beta\). Applying this lemma for a \(k \neq 0\) with \(\diln T = 1 - \shift^{\nicefrac{-1}{2}+\alpha}\) so that \(\shift^{\nicefrac{3}{2}} (1 - \diln T)^3 = \shift^{3\alpha}\), we see that there are \(C, C' > 0\) depending only on \(\beta\) and \(k\) such that for any \(x > 0\),
\[
\pprob[\bigg]{\biggl\{ \sup_{t\in\timedom} e^{-2\AMPN(t)} \leq 2\log\shift \biggr\} \cap \biggl\{ \sup_{t\in\timedom} \abs[\Big]{\int_0^t e^{-2\AMPN(s) + ki\phaseN(s)} \frac{\diln}{1 - \diln s} \diff{s}} \geq x + \frac{2C\log\shift}{\shift^{3\alpha}} \biggr\}}
    \leq 4 \exp\Bigl( - \frac{C' \shift^{6\alpha} x^2}{4\log^2\shift} \Bigr).
\]
By definition, \(e^{-2\AMPN} \leq 2\log\shift\) on \(\mathscr{G}\), so the first of the events in the above can be replaced with \(\mathscr{G}\) without changing the tail bound. Taking \(x = \shift^{\nicefrac{-\alpha}{4}+\delta}\) then yields an exponential tail bound, which vanishes faster than \(\nicefrac{1}{\log\shift}\) as \(\shift\to\infty\). Hence, using this with \(k = -2\) and \(k = -4\) gives the control we need on both terms of the second line of \eqref{eq.dRealmostHBM}.

\step{Reducing to a discrete-time problem}

It remains to compare the first line of~\eqref{eq.dRealmostHBM} with the real part of a hyperbolic Brownian motion driven by \(W\) and run in logarithmic time. To do so, we reduce the comparison between the two stochastic integrals to a comparison between discrete-time martingales. As in the proof of Lemma~\ref{lem.probspace}, we partition the time interval with a sequence \(\{t_j\}_{j=0}^N\) defined with \(t_0 \defeq 0\), \(\diln t_N \defeq 1 - \nicefrac{1}{\sqrt{\shift}}\), and with
\[
\diln t_j \defeq 1 - \frac{1}{(1 + \shift^{-p})^j}
\quadtext{for}
0 < j < N
\]
where here we set \(p \defeq \nicefrac{\alpha}{2}\). As we have seen in the proof of Lemma~\ref{lem.probspace}, this means that \(N\) is given by~\eqref{eq.probspace.Ndefandbound}, so that in particular \(N \leq \shift^p\log\shift\) for \(\shift\) large enough. Now, with \(S_{\shift,\alpha} \defeq \{j \in \mathbb{N} : 0 < \diln t_{j-1} \leq 1 - \shift^{\nicefrac{-1}{2}+\alpha}\}\), we can write
\begin{subequations}
\label{eq.ReHBM.supwithincrements}
\begin{align}
&\sup_{t\in\timedom} 
    \abs[\bigg]{\int_0^t e^{-2\AMPN(s)} \cos 2\phaseN(s) \sqrt{\frac{2\diln}{1 - \diln s}} \diff{\sBM}(s)
    - \int_0^t \GBM(s) \diff{\bigl( \Re W\circ\slogtime \bigr)}(s)} \notag \\
& \hspace*{11mm} \leq \sup_{j \in S_{\shift,\alpha}}
    \abs[\bigg]{\int_0^{t_{j-1}} e^{-2\AMPN(s)} \cos 2\phaseN(s) \sqrt{\frac{2\diln}{1 - \diln s}} \diff{\sBM}(s)
    - \int_0^{t_{j-1}} \GBM(s) \diff{\bigl( \Re W\circ\slogtime \bigr)}(s)} 
    \label{eq.ReHBM.supwithincrements.discrete} \\
& \hspace*{22mm} + \sup_{j \in S_{\shift,\alpha}} \sup_{t\in [t_{j-1}, t_j] \cap \timedom} 
    \abs[\bigg]{\int_{t_{j-1}}^t e^{-2\AMPN(s)} \cos 2\phaseN(s) \sqrt{\frac{2\diln}{1 - \diln s}} \diff{\sBM}(s)} 
    \label{eq.ReHBM.supwithincrements.B} \\
& \hspace*{22mm} + \sup_{j \in S_{\shift,\alpha}} \sup_{t\in [t_{j-1}, t_j] \cap \timedom} 
    \abs[\bigg]{\int_{t_{j-1}}^t \GBM(s) \diff{\bigl( \Re W\circ\slogtime \bigr)}(s)}.
    \label{eq.ReHBM.supwithincrements.W}
\end{align}
\end{subequations}
Here, the two last lines are easily controlled on \(\mathscr{G}\). Indeed, on this event \(e^{-2\AMPN} \leq 2\log\shift\) uniformly for \(\diln t \in [0, 1-\shift^{\nicefrac{-1}{2}+\alpha}]\), so the bracket of the integral in \eqref{eq.ReHBM.supwithincrements.B} is bounded for any \(t \in [t_{j-1}, t_j] \cap \timedom\) by
\[
4 \log^2\shift \int_{t_{j-1}}^{t_j} \frac{2\diln}{1 - \diln s} \diff{s}
    = 8 \log^2\shift \log(1 + \shift^{-p})
    \leq 8 \shift^{-p} \log^2\shift.
\]
Thus, by Bernstein's inequality, for any \(x > 0\),
\[
\pprob[\bigg]{\mathscr{G} \cap \biggl\{ \sup_{t\in [t_{j-1}, t_j] \cap \timedom} \abs[\Big]{\int_{t_{j-1}}^t e^{-2\AMPN(s)} \cos 2\phaseN(s) \sqrt{\frac{2\diln}{1 - \diln s}} \diff{\sBM}(s)} \geq x \biggr\}}
    \leq 2 \exp\Bigl( - \frac{\shift^p x^2}{16\log^2\shift} \Bigr).
\]
Taking the supremum over \(j \in S_{\shift,\alpha}\) adds a prefactor of \(N\) to the tail bound, so because \(N \leq \shift^p\log\shift = \shift^{\nicefrac{\alpha}{2}} \log\shift\), with \(x = \shift^{\nicefrac{-\alpha}{4}+\delta}\) the tail bound decreases exponentially in a power of \(\shift\) and is certainly \(O(\nicefrac{1}{\log\shift})\).

Likewise, on \(\mathscr{G}\) the geometric Brownian motion \(\GBM\) is bounded by \(\log\shift\), so the bracket of the integral in \eqref{eq.ReHBM.supwithincrements.W} is bounded by
\[
\log^2\shift \int_{t_{j-1}}^{t_j} \diff{\quadvar{\Re W\circ\slogtime}}(s)
    = \log^2\shift \int_{t_{j-1}}^{t_j} \frac{\diln}{1 - \diln s} \diff{s}
    \leq \shift^{-p} \log^2\shift.
\]
In the same way as in the last case, Bernstein's inequality then yields a tail bound for~\eqref{eq.ReHBM.supwithincrements.W} on \(\mathscr{G}\) that decreases exponentially in a power of \(\shift\). 

The problem is thus reduced to controlling~\eqref{eq.ReHBM.supwithincrements.discrete}.

\step{Controlling the difference of discrete-time martingales}

It remains to control~\eqref{eq.ReHBM.supwithincrements.discrete}, the supremum of the difference of discrete martingales. To do this, given a process \(X\), we let \(R(X)\) denote the process \(X\) modified to be reset at every time \(t_k\) of the partition, that is, \(R(X)(s) \defeq X(s) - X(t_{k-1})\) for \(s \in [t_{k-1}, t_k)\). Using this notation, we can split each difference of increments in~\eqref{eq.ReHBM.supwithincrements.discrete} in four differences, so that the full difference takes the form
\begin{subequations}
\label{eq.ReHBM.supsumincrements}
\begin{align}
&\int_0^{t_{j-1}} e^{-2\AMPN(s)} \cos 2\phaseN(s) \sqrt{\frac{2\diln}{1 - \diln s}} \diff{\sBM}(s)
    - \int_0^{t_{j-1}} \GBM(s) \diff{\bigl( \Re W\circ\slogtime \bigr)}(s) \notag \\
    & \hspace*{6mm} = \int_0^{t_{j-1}} R\bigl( e^{-2\AMPN} \bigr)(s) \cos 2\phaseN(s) \sqrt{\frac{2\diln}{1 - \diln s}} \diff{\sBM}(s) 
    \label{eq.ReHBM.supsumincrements.RB} \\
    & \hspace*{12mm} + \sum_{k=1}^{j-1} \Bigl( e^{-2\AMPN(t_{k-1})} - \GBM(t_{k-1}) \Bigr) \int_{t_{k-1}}^{t_k} \cos 2\phaseN(s) \sqrt{\frac{2\diln}{1 - \diln s}} \diff{\sBM}(s) 
    \label{eq.ReHBM.supsumincrements.diffGBMs} \\
\shortintertext{}
    & \hspace*{12mm} + \sum_{k=1}^{j-1} \GBM(t_{k-1}) \biggl( \int_{t_{k-1}}^{t_k} \cos 2\phaseN(s) \sqrt{\frac{2\diln}{1 - \diln s}} \diff{\sBM}(s) - \Re W\circ\slogtime(t_k) + \Re W\circ\slogtime (t_{k-1}) \biggr) 
    \label{eq.ReHBM.supsumincrements.diffBMs} \\
    & \hspace*{12mm} - \int_0^{t_{j-1}} R(\GBM)(s) \diff{\bigl( \Re W\circ\slogtime \bigr)}(s).
    \label{eq.ReHBM.supsumincrements.RW}
\end{align}
\end{subequations}
We will control each of these four lines independently.

We start with the easiest ones. First, the integral in a single term of \eqref{eq.ReHBM.supsumincrements.diffGBMs} has bracket
\[
2 \int_{t_{k-1}}^{t_k} \cos^2 2\phaseN(s) \frac{\diln}{1 - \diln s} \diff{s}
    \leq 2\log(1 + \shift^{-p})
    \leq 2\shift^{-p},
\]
so by Bernstein's inequality, for any \(x > 0\),
\[
\bprob[\bigg]{\abs[\Big]{\int_{t_{k-1}}^{t_k} \cos 2\phaseN(s) \sqrt{\frac{2\diln}{1 - \diln s}} \diff{\sBM}(s)} \geq x}
    \leq 2 \exp\Bigl( - \frac{\shift^p x^2}{4} \Bigr).
\]
Since \(\abs{e^{-2\AMPN} - \GBM} \leq \shift^{-\alpha+\delta}\) on \(\mathscr{G}\), it follows that for any \(k\) and any \(x > 0\),
\[
\pprob[\bigg]{\mathscr{G} \cap \Bigl\{ \abs[\big]{e^{-2\AMPN(t_{k-1})} - \GBM(t_{k-1})} \abs[\Big]{\int_{t_{k-1}}^{t_k} \cos 2\phaseN(s) \sqrt{\frac{2\diln}{1 - \diln s}} \diff{\sBM}(s)} \geq x \Bigr\}}
    \leq 2 \exp\Bigl( - \frac{\shift^{\nicefrac{5\alpha}{2}-2\delta} x^2}{4} \Bigr).
\]
Combining the tail bounds for all of the increments and using that \(N \leq \shift^{\nicefrac{\alpha}{2}} \log\shift\) for \(\shift\) large enough, we get
\begin{multline*}
\pprob[\bigg]{\mathscr{G} \cap \biggl\{ \sup_{j\in S_{\shift,\alpha}} \sum_{k=1}^{j-1} \abs[\big]{e^{-2\AMPN(t_{k-1})} - \GBM(t_{k-1})} \abs[\Big]{\int_{t_{k-1}}^{t_k} \cos 2\phaseN(s) \sqrt{\frac{2\diln}{1 - \diln s}} \diff{\sBM}(s)} \geq x \biggr\}} \\
    \leq 2\shift^{\nicefrac{\alpha}{2}} \log\shift \exp\Bigl( - \frac{\shift^{\nicefrac{3\alpha}{2}-2\delta} x^2}{4\log^2\shift} \Bigr).
\end{multline*}
With \(x = \shift^{\nicefrac{-\alpha}{4}+\delta}\), this bound is \(O(\nicefrac{1}{\log\shift})\) as needed.

Then, to control \eqref{eq.ReHBM.supsumincrements.diffBMs}, we use the tail bound \eqref{eq.probspace} we got in Lemma~\ref{lem.probspace}. Indeed, as \(\GBM \leq \log\shift\) on \(\mathscr{G}\), for any \(k\) we get that for \(\shift\) large enough,
\[
\frac{\shift^{\nicefrac{-\alpha}{4}+\delta}}{N\GBM(t_{k-1})}
    \geq \frac{\shift^{\nicefrac{-3\alpha}{4}+\delta}}{\log^2\shift}
    \geq \shift^{\nicefrac{-3\alpha}{4}},
\]
so by Lemma~\ref{lem.probspace},
\begin{multline*}
\pprob[\bigg]{\mathscr{G} \cap \Bigl\{ \GBM(t_{k-1}) \abs[\Big]{\int_{t_{k-1}}^{t_k} \cos 2\phaseN(s) \sqrt{\frac{2\diln}{1 - \diln s}} \diff{\sBM}(s) - \Re W\circ\slogtime(t_k) + \Re W\circ\slogtime(t_{k-1})} \geq \frac{\shift^{\nicefrac{-\alpha}{4}+\delta}}{N} \Bigr\}} \\
    \leq 3\shift^{4\alpha} \log^2\shift \exp\bigl( - C\shift^{\nicefrac{\alpha}{6}} \bigr).
\end{multline*}
In a similar way as in the last case, it follows that the supremum of~\eqref{eq.ReHBM.supsumincrements.diffBMs} over \(j \in S_{\shift,\alpha}\) only exceeds \(\shift^{\nicefrac{-\alpha}{4}+\delta}\) on \(\mathscr{G}\) with probability at most \(3 \shift^{\nicefrac{9\alpha}{2}} \log^3\shift \exp\bigl( - C\shift^{\nicefrac{\alpha}{6}} \bigr)\), which again vanishes faster than \(\nicefrac{1}{\log\shift}\).

Now, the stochastic integral in \eqref{eq.ReHBM.supsumincrements.RW} has quadratic variation
\begin{align}
\label{eq.ReHBM.quadvarRW}
\int_0^{t_{j-1}} R^2(\GBM)(s) \diff{\quadvar{\Re W\circ\slogtime}}(s)
    & = \sum_{k=1}^{j-1} \int_{t_{k-1}}^{t_k} \bigl( \GBM(s) - \GBM(t_{k-1}) \bigr)^2 \frac{\diln}{1 - \diln s} \diff{s} \\
    & = \frac{4}{\beta} \sum_{k=1}^{j-1} \int_{t_{k-1}}^{t_k} \biggl( \int_{t_{k-1}}^s \GBM(u) \diff{\bigl( \Im W\circ\slogtime \bigr)}(u) \biggr)^2 \frac{\diln}{1 - \diln s} \diff{s}.
    \notag
\end{align}
For any fixed \(k\), for \(s \in [t_{k-1}, t_k]\) the bracket of the remaining stochastic integral is bounded on \(\mathscr{G}\) by
\[
\log^2\shift \int_{t_{k-1}}^{t_k} \frac{\diln}{1 - \diln s} \diff{s}
    = \log^2\shift \log(1 + \shift^{-p})
    \leq \shift^{-p} \log^2\shift,
\]
so by Bernstein's inequality, for any \(y > 0\),
\[
\pprob[\bigg]{\mathscr{G} \cap \Bigl\{ \max_{1\leq k\leq j-1} \sup_{s\in [t_{k-1}, t_k]} \abs[\Big]{\int_{t_{k-1}}^s \GBM(u) \diff{\bigl( \Im W\circ\slogtime \bigr)}(u)} \geq y \Bigr\}}
    \leq 2(j-1) \exp\Bigl( - \frac{\shift^p y^2}{2\log^2\shift} \Bigr).
\]
On the complementary event (relative to \(\mathscr{G}\)), the quadratic variation of the integral in \eqref{eq.ReHBM.supsumincrements.RW} is bounded by
\[
\frac{4y^2}{\beta} \sum_{k=1}^{j-1} \int_{t_{k-1}}^{t_k} \frac{\diln}{1 - \diln s} \diff{s}
    = \frac{4y^2}{\beta} \int_0^{t_{j-1}} \frac{\diln}{1 - \diln s} \diff{s}
    = - \frac{4y^2}{\beta} \log(1 - \diln t_{j-1})
    \leq \frac{2y^2}{\beta} \log\shift.
\]
Hence, applying again Bernstein's inequality, we obtain that for any \(x > 0\),
\[
\pprob[\bigg]{\mathscr{G} \cap \Bigl\{ \abs[\Big]{\int_0^{t_{j-1}} R(\GBM)(s) \diff{\bigl( \Re W\circ\slogtime \bigr)}(s)} \geq x \Bigr\}}
    \leq 2(j-1) \exp\Bigl( - \frac{\shift^p y^2}{2\log^2\shift} \Bigr) + 2\exp\Bigl( - \frac{\beta x^2}{4y^2\log\shift} \Bigr).
\]
We can get the exponents to match by setting \(y^2 = \sqrt{\frac{\beta\log\shift}{2\shift^p}} x\), which yields
\[
\pprob[\bigg]{\mathscr{G} \cap \Bigl\{ \abs[\Big]{\int_0^{t_{j-1}} R(\GBM)(s) \diff{\bigl( \Re W\circ\slogtime \bigr)}(s)} \geq x \Bigr\}}
    \leq 2j \exp\Bigl( - \frac{\sqrt{\beta} \shift^{\nicefrac{p}{2}} x}{2\sqrt{2}\log^{\nicefrac{3}{2}}\shift} \Bigr).
\]
Taking the supremum over \(j \in S_{\shift,\alpha}\) replaces the prefactor with \(N^2 \leq \shift^\alpha \log^2\shift\), and with \(x = \shift^{\nicefrac{-\alpha}{4}+\delta}\) the power of \(\shift\) in the exponent remains positive, so the upper bound is \(O(\nicefrac{1}{\log\shift})\).

It only remains to bound \eqref{eq.ReHBM.supsumincrements.RB}, which we do in the same way as in the last case. The stochastic integral in \eqref{eq.ReHBM.supsumincrements.RB} quadratic variation
\[
2 \sum_{k=1}^{j-1} \int_{t_{k-1}}^{t_k} \Bigl( e^{-2\AMPN(s)} - e^{-2\AMPN(t_{k-1})} \Bigr)^2 \cos^2 2\phaseN(s) \frac{\diln}{1 - \diln s} \diff{s}.
\]
Now,
\[
\abs[\big]{e^{-2\AMPN(s)} - e^{-2\AMPN(t_{k-1})}}
    \leq \abs[\big]{e^{-2\AMPN(s)} - \GBM(s)} + \abs[\big]{\GBM(s) - \GBM(t_{k-1})} + \abs[\big]{\GBM(t_{k-1}) - e^{-2\AMPN(t_{k-1})}}.
\]
The square of this can be bounded using Young's inequality together with the fact that \(\abs{e^{-2\AMPN} - \GBM} \leq \shift^{-\alpha+\delta}\) on \(\mathscr{G}\), and this shows that the above quadratic variation is bounded on \(\mathscr{G}\) by
\[
4 \sum_{k=1}^{j-1} \int_{t_{k-1}}^{t_k} \Bigl( \abs[\big]{\GBM(s) - \GBM(t_{k-1})}^2 + 4\shift^{-2\alpha+2\delta} \Bigr) \frac{\diln}{1 - \diln s} \diff{s}.
\]
Here, we can bound the second term by \(16\shift^{-2\alpha+2\delta} \log(1 - \diln t_{j-1}) \leq 8\shift^{-2\alpha+2\delta} \log\shift\) once we sum up the increments. Then, the first term is exactly four times the quadratic variation~\eqref{eq.ReHBM.quadvarRW} that we bounded for the case of~\eqref{eq.ReHBM.supsumincrements.RW}. Using what we found in that case, we get that for any \(y > 0\)
\begin{multline*}
\pprob[\bigg]{\mathscr{G} \cap \Bigl\{ 4 \sum_{k=1}^{j-1} \int_{t_{k-1}}^{t_k} \Bigl( \abs[\big]{\GBM(s) - \GBM(t_{k-1})}^2 + 4\shift^{-2\alpha+2\delta} \Bigr) \frac{\diln}{1 - \diln s} \diff{s} \geq 8 \Bigl( \frac{y^2}{\beta} + \shift^{-2\alpha+2\delta} \Bigr) \log\shift \Bigr\}} \\
    \leq 2(j-1) \exp\Bigl( - \frac{\shift^p y^2}{2\log^2\shift} \Bigr).
\end{multline*}
This controls the quadratic variation of the stochastic integral in \eqref{eq.ReHBM.supsumincrements.RB}, so a further application of Bernstein's inequality shows that for any \(x > 0\),
\begin{multline*}
\pprob[\bigg]{\mathscr{G} \cap \Bigl\{ \abs[\Big]{\int_0^{t_{j-1}} R(e^{-2\AMPN})(s) \cos 2\phaseN(s) \sqrt{\frac{2\diln}{1 - \diln s}} \diff{\sBM}(s)} \geq x \Bigr\}} \\
    \leq 2(j-1) \exp\Bigl( - \frac{\shift^p y^2}{2\log^2\shift} \Bigr) + 2\exp\Bigl( - \frac{\beta x^2}{16(y^2 + \beta\shift^{-2\alpha+2\delta}) \log\shift} \Bigr).
\end{multline*}
Setting \(y^2 = \sqrt{\frac{\beta\log\shift}{8\shift^p}} x\) would make the exponents match if the extra term \(\beta\shift^{-2\alpha+2\delta}\) did not appear in the denominator of the second exponent. Nevertheless, if the \(y^2\) part of the quadratic variation dominates, this still gives the best order for \(y^2\), and taking this we obtain
\[
\pprob[\bigg]{\mathscr{G} \cap \Bigl\{ \abs[\Big]{\int_0^{t_{j-1}} \!\! R(e^{-2\AMPN})(s) \cos 2\phaseN(s) \sqrt{\frac{2\diln}{1 - \diln s}} \diff{\sBM}(s)} \geq x \Bigr\}}
    \leq 2j \exp\biggl( - \frac{\sqrt{\beta}\shift^{\nicefrac{p}{2}} x}{4\sqrt{2}\log^{\nicefrac{3}{2}}\!\shift} \Bigl( 1 + \frac{2\sqrt{2\beta}\shift^{\nicefrac{p}{2}-2\alpha+2\delta}}{x\sqrt{\log\shift}} \Bigr)^{-1} \biggr).
\]
Taking the supremum only converts the \(2j\) factor into an \(N^2\) factor, so with \(x = \shift^{\nicefrac{-\alpha}{4}+\delta}\), the bound remains \(O(\nicefrac{1}{\log\shift})\) like in the other cases. 

This finishes to control the supremum of the difference between the two discrete-time martingales in \eqref{eq.ReHBM.supwithincrements.discrete}, and therefore concludes the proof.
\end{proof}

\section{Vague convergence of the canonical systems and convergence of solutions}
\label{sec.CSconv}

In the last section, we showed that the process \(-\exp(-\diffamps-i\diffphases)\), which appears in the first term of the Airy system's coefficient matrix~\eqref{eq.sAirymatpolar}, becomes as \(\shift\to\infty\) a hyperbolic Brownian motion driven by \(W\). Heuristically, this shows that the first term of the coefficient matrix~\eqref{eq.sAirymatpolar} essentially becomes that of the sine system in the limit.

In this section, we complete this argument to prove the vague convergence of the canonical systems. An important ingredient that we are missing, however, is a better control on the size of the entries of the Airy system's coefficient matrix. We start by filling this gap, and then we complete the proof of the vague convergence. Finally, we conclude this section with a proof of an important consequence of this: the compact convergence of the canonical systems' transfer matrices, which shows that their solutions also converge.

\subsection{Controlling the magnitude of the entries of the coefficient matrix}

In order to control the magnitude of the entries of the coefficient matrix, we will use the following property of continuous martingales.

\begin{proposition}
\label{prop.martingalebound}
Let \(M\) be a continuous martingale on an interval \([0,T)\) with \(M(0) = 0\). For any \(\varepsilon, \delta > 0\), there is a \(y > 0\) such that
\[
\bprob[\bigg]{\sup_{t \in [0,T)} \frac{\abs{M(t)}}{1 + \quadvar{M}(t)^{\nicefrac{1}{2} + \delta}} \geq y} < \varepsilon.
\]
\end{proposition}

\begin{proof}
Setting \(T_t \defeq \inf\{s \geq 0 : \quadvar{M}(s) > t\}\), by the Dambis--Dubins--Schwarz theorem, \(B(t) \defeq M(T_t)\) is a standard Brownian motion on \(\bigl[ 0, \quadvar{M}(T) \bigr)\) and \(M(t) = B\bigl( \quadvar{M}(t) \bigr)\) for any \(t \in [0,T)\). Therefore,
\[
\sup_{t \in [0,T)} \frac{\abs{M(t)}}{1 + \quadvar{M}(t)^{\nicefrac{1}{2} + \delta}}
    = \sup_{t \in [0,\quadvar{M}(T))} \frac{\abs{B(t)}}{1 + t^{\nicefrac{1}{2}+\delta}},
\]
and the result will follow if we can prove that for any \(\varepsilon, \delta > 0\), there is a \(y > 0\) such that
\[
\bprob[\bigg]{\sup_{t\in [0,\infty)} \frac{\abs{B(t)}}{1 + t^{\nicefrac{1}{2}+\delta}} \geq y} < \varepsilon.
\]
Now, since \((1 + t^{\nicefrac{1}{2}+\delta})^{-1} \sqrt{2t\log\log t} \to 0\) as \(t\to\infty\),
\[
\limsup_{t\to\infty} \frac{\abs{B(t)}}{1+t^{\nicefrac{1}{2}+\delta}}
    \leq \Bigl( \limsup_{t\to\infty} \frac{\abs{B(t)}}{\sqrt{2t\log\log t}} \Bigr) \Bigl( \limsup_{t\to\infty} \frac{\sqrt{2t\log\log t}}{1+t^{\nicefrac{1}{2}+\delta}} \Bigr)
    = 0
\]
a.s.\@ by the law of the iterated logarithm. Since the map \(t \mapsto \abs{B(t)} (1 + t^{\nicefrac{1}{2}+\delta})^{-1}\) is a.s.\@ continuous, it follows that the random variable \(\sup_{t \in [0,\infty)} \abs{B(t)} (1+t^{\nicefrac{1}{2}+\delta})^{-1}\) is a.s.\@ finite, and because
\[
\biggl\{ \sup_{t\in [0,\infty)} \frac{\abs{B(t)}}{1+t^{\nicefrac{1}{2}+\delta}} < \infty \biggr\} = \bigcup_{n\in\mathbb{N}} \biggl\{ \sup_{t\in [0,\infty)} \frac{\abs{B(t)}}{1+t^{\nicefrac{1}{2}+\delta}} < n \biggr\},
\]
the result follows by continuity of measures.
\end{proof}

From this property of continuous martingales, we can prove the following result, which will allow us to control the magnitude of the entries of the shifted Airy system's coefficient matrix.

\begin{proposition}
\label{prop.tracebound}
Let \(\AMP\) solve the SDE from Proposition~\ref{prop.polarcoords} with \(\AMP(0) = 0\). For any \(\varepsilon > 0\), there are constants \(C, C' > 0\) depending only on \(\beta\) and \(\varepsilon\) such that
\[
\bprob[\Big]{\forall t \in [0,\lasttime], \abs[\Big]{\AMP(t) + \frac{1}{\beta} \log(1 - \diln t)} \leq C + C'\bigl( -\log(1 - \diln t) \bigr)^{\nicefrac{3}{4}}} > 1 - \varepsilon.
\]
\end{proposition}

\begin{remark}
\label{rmk.tracebound}
This gives a good control on \(e^{2\AMPN}\), \(e^{\AMPN+\AMPD}\) and \(e^{2\AMPD}\), which are the amplitudes of the entries of the shifted Airy system's coefficient matrix. Indeed, it shows that for \(k,l \in \mathbb{R}\) and \(\varepsilon > 0\), there are \(C, C' > 0\) such that \(k\AMPN(t) + l\AMPD(t) \leq C + C'\bigl( -\log(1-\diln t) \bigr)^{\nicefrac{3}{4}} - \frac{k+l}{\beta} \log(1 - \diln t)\) with probability at least \(1 - \varepsilon\). If \(k + l = 2\), it follows that for any \(\gamma > 0\), there is a \(C'' > 0\) depending only on \(\beta\), \(\gamma\) and \(\varepsilon\) such that for all \(t \in [0,\lasttime]\),
\beq{eq.tracebound.integrable}
e^{k\AMPN(t)+l\AMPD(t)}
    \leq \frac{C''e^C}{(1 - \diln t)^{\nicefrac{2}{\beta}+\gamma}}
    \leq \frac{C''e^C}{(1 - t)^{\nicefrac{2}{\beta}+\gamma}}.
\eeq
When \(\beta > 2\), one can choose \(\gamma\) so that \(\nicefrac{2}{\beta} + \gamma < 1\), and then this last upper bound is integrable on \([0,1)\) and does not depend on \(\shift\).
\end{remark}

\begin{proof}
Recall from Proposition~\ref{prop.polarcoords} that \(\AMP\) satisfies
\[
\diff{\AMP}(t)
    = \biggl( \frac{1}{\beta} + \Bigl( \frac{2}{\beta} - \frac{1}{2} \Bigr) \cos 2\phase(t) + \frac{1}{\beta} \cos 4\phase(t) \biggr) \frac{\diln}{1 - \diln t} \diff{t} + \sqrt{\frac{2}{\beta}} \sin 2\phase(t) \sqrt{\frac{\diln}{1 - \diln t}} \diff{\sBM}(t)
\]
with \(\AMP(0) = 0\), where \(\phase\) solves the other SDE from Proposition~\ref{prop.polarcoords}. Fix \(\varepsilon > 0\). By Corollary~\ref{cor.averaging.integrable}, the cosine terms are bounded by constants with probability at least \(1 - \nicefrac{\varepsilon}{3}\), so it suffices to show that there is a \(\tilde{C} > 0\) such that
\beq{eq.tracebound.2}
\bprob[\bigg]{\sup_{t\in [0,\lasttime]} \frac{1}{1 + \bigl( -\log(1 - \diln t) \bigr)^{\nicefrac{3}{4}}} \abs[\Big]{\int_0^t \sin 2\phase(s) \sqrt{\frac{\diln}{1 - \diln s}} \diff{\sBM}(s)} > \tilde{C}} < \frac{\varepsilon}{3}.
\eeq
But
\[
M(t) \defeq \int_0^t \sin 2\phase(s) \sqrt{\frac{\diln}{1 - \diln s}} \diff{\sBM}(s)
\]
is a continuous martingale on \([0,\lasttime]\) with \(M(0) = 0\) and bracket
\[
\quadvar{M}(t)
    = \int_0^t \sin^2 2\phase(s) \frac{\diln}{1 - \diln s} \diff{s}
    \leq - \log(1 - \diln t),
\]
so \eqref{eq.tracebound.2} follows from Proposition~\ref{prop.martingalebound}.
\end{proof}

\subsection{Vague convergence of canonical systems}

We are now ready to complete the proof of Theorem~\ref{thm.CSconvinlaw}. What we prove is the following stronger result.

\begin{theorem}
\label{thm.Airytosine.vaguely}
\setshift{E_n}
Let \(\{\sBM\}_{n\in\mathbb{N}}\) and \(W\) be the Brownian motions from Lemma~\ref{lem.probspace}, and let \(\sAirymat\) and \(\sinemat\) be the coefficient matrices of the shifted Airy system and of the sine system built from these Brownian motions as described in Remark~\ref{rmk.probspace}. Let \(\mathcal{I} \defeq [0,1)\) if \(\beta \leq 2\) and \(\mathcal{I} \defeq [0,1]\) if \(\beta > 2\). Then, for any \(\testfunc \in \mathscr{C}_c(\mathcal{I}, \mathbb{C}^2)\),
\beq{eq.Airytosine.vaguely}
\int_{\mathcal{I}} \testfunc^*(t) \bigl( \timechange'(t) \sAirymat\circ\timechange(t) - \sinemat\circ\logtime(t) \bigr) \testfunc(t) \diff{t}
\probto[n\to\infty] 0.
\eeq
In particular, \(\timechange' (\sAirymat\circ\timechange) \to \sinemat\circ\logtime\) vaguely on \(\mathcal{I}\) in probability and in distribution.
\end{theorem}

\begin{proof}
Recall that the time-changed coefficient matrix of the shifted Airy system can be written as
\[
\timechange' (\sAirymat\circ\timechange) = \frac{\diln}{2e^{-\diffamps}\sin\diffphases} \begin{pmatrix} 1 & e^{-\diffamps}\cos\diffphases \\ e^{-\diffamps}\cos\diffphases & e^{-2\diffamps} \end{pmatrix} + \frac{\diln}{2} \begin{pmatrix} e^{2\AMPN} \cos 2\phaseN & e^{\sumamps} \cos\sumphases \\ e^{\sumamps} \cos\sumphases & e^{2\AMPD} \cos 2\phaseD \end{pmatrix}
\]
as it was done in \eqref{eq.sAirymatpolar}. To prove that \eqref{eq.Airytosine.vaguely} holds, we first simplify the problem in two ways: we remove a short part at the end of the time interval, and we replace the sine system's logarithmic time \(\logtime(t) = -\log(1-t)\) with \(\slogtime(t) = -\log(1 - \diln t)\). Then, the rest of the proof is split in two main parts: first we show that \eqref{eq.Airytosine.vaguely} holds when \(\timechange' (\sAirymat\circ\timechange)\) is replaced with the first term of~\eqref{eq.sAirymatpolar}, and then we show that the second term converges to zero.

Note that once \eqref{eq.Airytosine.vaguely} is proven, the fact that \(\timechange' (\sAirymat\circ\timechange) \to \sinemat\circ\logtime\) vaguely on \(\mathcal{I}\) in probability follows directly from Proposition~\ref{prop.CS.random.convergence}.

As we have done before, in order to simplify the notation throughout the proof we omit the explicit indexing of the shift sequence \(\{\shift_n\}_{n\in\mathbb{N}}\). In particular, any limit as \(\shift\to\infty\) is taken along that sequence and should really be understood as a limit as \(n\to\infty\). 

\setcounter{step}{0}
\step{Shortening the time interval}

We start by showing that instead of integrating up to time \(1\), we can equivalently integrate only up to time \(1 - \shift^{\nicefrac{-1}{2}+\alpha}\) for some fixed \(\alpha \in (\frac{1}{\beta+2}, \frac{1}{2})\), i.e., that
\beq{eq.Airytosine.vaguely.shortening}
\int_{1-\shift^{\nicefrac{-1}{2}+\alpha}}^1 \testfunc^*(t) \bigl( \timechange'(t) \sAirymat\circ\timechange(t) - \sinemat\circ\logtime(t) \bigr) \testfunc(t) \diff{t} \probto[\shift\to\infty] 0.
\eeq
Notice that if \(\beta \leq 2\), then \(\testfunc\) is compactly supported in \([0,1)\) so this statement is trivial as for any \(\shift\) large enough, the whole interval \([1 - \shift^{\nicefrac{-1}{2}+\alpha}, 1]\) has left the support of \(\testfunc\). 

To show that \eqref{eq.Airytosine.vaguely.shortening} also holds when \(\beta > 2\), recall that in that case \(\tr(\sinemat\circ\logtime)\) is a.s.\@ integrable on \([0,1)\), so \(\testfunc^* (\sinemat\circ\logtime) \testfunc\) is a.s.\@ integrable on \([0,1)\), and therefore
\[
\int_{1-\shift^{\nicefrac{-1}{2}+\alpha}}^1 \testfunc^*(t) \sinemat\circ\logtime(t) \testfunc(t) \diff{t} \to 0
\quadtext{as}
\shift\to\infty
\]
a.s.\@ and in probability. The same is true with the other coefficient matrix, as its entries are bounded by \(e^{2\AMPN}\), \(e^{2\AMPD}\) or \(e^{\AMPN+\AMPD}\). Indeed, we have seen in \eqref{eq.tracebound.integrable} in Remark~\ref{rmk.tracebound} that given \(\gamma > 0\), for any \(\varepsilon > 0\), there is a \(C_\varepsilon > 0\) depending only on \(\beta\), \(\gamma\) and \(\varepsilon\) such that if \(k+l = 2\), then
\[
\bprob[\Big]{\forall t \in [0,1], e^{k\AMPN(t)+l\AMPD(t)} \leq \frac{C_\varepsilon}{(1 - t)^{\nicefrac{2}{\beta}+\gamma}}} \geq 1 - \varepsilon,
\]
and with \(\beta > 2\) this upper bound can be taken to be integrable by choosing \(\gamma\) so that \(\nicefrac{2}{\beta} + \gamma < 1\). Thus, on this event,
\[
\int_{1-\shift^{\nicefrac{-1}{2}+\alpha}}^1 \testfunc^*(t) \timechange'(t) \sAirymat\circ\timechange(t) \testfunc(t) \diff{t}
    \leq 2C_\varepsilon \norm{\testfunc}_\infty^2 \int_{1-\shift^{\nicefrac{-1}{2}+\alpha}}^1 \frac{1}{(1-t)^{\nicefrac{2}{\beta}+\gamma}} \diff{t},
\]
which vanishes as \(\shift\to\infty\) by dominated convergence. It follows that the integral on the left-hand side converges to \(0\) in probability as \(\shift\to\infty\), which finishes to prove~\eqref{eq.Airytosine.vaguely.shortening}. 

So for any \(\beta > 0\), the problem is reduced to showing that
\[
\int_0^{1-\shift^{\nicefrac{-1}{2}+\alpha}} \testfunc^*(t) \bigl( \timechange'(t) \sAirymat\circ\timechange(t) - \sinemat\circ\logtime(t) \bigr) \testfunc(t) \diff{t}
\probto[\shift\to\infty] 0.
\]

\step{Changing the logarithmic time scale of the sine system}

We now change the logarithmic time scale of the sine system from \(\logtime(t) = -\log(1-t)\) to \(\slogtime(t) = -\log(1-\diln t)\). Recall that \(\diln = 1\) when \(\beta \leq 2\), so \(\slogtime = \logtime\) in that case, and again there is only something to prove when \(\beta > 2\). 

To see that the time scale can also be changed when \(\beta > 2\), note that for \(\testfunc \in \mathscr{C}_c\bigl( [0,1], \mathbb{C}^2 \bigr)\),
\begin{align}
\label{eq.Airytosine.vaguely.timescale}
&\int_0^{1-\shift^{\nicefrac{-1}{2}+\alpha}} \testfunc^*(t) \Bigl( \sinemat\circ\logtime(t) - \sinemat\circ\slogtime(t) \Bigr) \testfunc(t) \diff{t} \\
    &\quad= \int_0^{\diln(1-\shift^{\nicefrac{-1}{2}+\alpha})} \Bigl( \testfunc^*(t) \sinemat\circ\logtime(t) \testfunc(t) - \frac{1}{\diln} \testfunc^*\bigl( \nicefrac{t}{\diln} \bigr) \sinemat\circ\logtime(t) \testfunc\bigl( \nicefrac{t}{\diln} \bigr) \Bigr) \diff{t} 
    + \int_{\diln(1-\shift^{\nicefrac{-1}{2}+\alpha})}^{1-\shift^{\nicefrac{-1}{2}+\alpha}} \testfunc^*(t) \sinemat\circ\logtime(t) \testfunc(t) \diff{t},
    \notag
\end{align}
where we changed variables from \(t\) to \(\logtime^{-1}\circ\slogtime(t) = \diln t\) in the second term of the first integral. Since \(\testfunc^* (\sinemat\circ\logtime) \testfunc\) is a.s.\@ integrable, the second integral goes to zero a.s.\@ since the size of the interval of integration is \((1-\shift^{\nicefrac{-1}{2}+\alpha})/\sqrt{\shift} \to 0\). Then, since \(\sinemat\) is real and positive semi-definite, the first integral can be written as
\[
\Re \int_0^{\diln(1-\shift^{\nicefrac{-1}{2}+\alpha})} \Bigl( \testfunc^*(t) + \frac{1}{\sqrt{\diln}} \testfunc^*\bigl( \nicefrac{t}{\diln} \bigr) \Bigr) \sinemat\circ\logtime(t) \Bigl( \testfunc(t) - \frac{1}{\sqrt{\diln}} \testfunc\bigl( \nicefrac{t}{\diln} \bigr) \Bigr) \diff{t}.
\]
The integrand is dominated by \(\frac{4}{\diln} \norm{\testfunc}_\infty^2 \tr(\sinemat\circ\logtime) \leq 8\norm{\testfunc}_\infty^2 \tr(\sinemat\circ\logtime)\) for \(\shift\) large enough, and this bound is a.s.\@ integrable. Hence, by dominated convergence, the integral vanishes in the limit since \(\frac{1}{\sqrt{\diln}} \testfunc\bigl( \nicefrac{t}{\diln} \bigr) \to \testfunc(t)\) as \(\shift\to\infty\) for every \(t \in [0,1)\). Thus, as \(\shift\to\infty\), \eqref{eq.Airytosine.vaguely.timescale} converges to \(0\) a.s., so also in probability.

For any \(\beta > 0\), this reduces the problem to proving that
\[
\int_0^{1-\shift^{\nicefrac{-1}{2}+\alpha}} \testfunc^*(t) \bigl( \timechange'(t) \sAirymat\circ\timechange(t) - \sinemat\circ\slogtime(t) \bigr) \testfunc(t) \diff{t} \probto[\shift\to\infty] 0.
\]

\step{Convergence of the first term to the sine system's coefficient matrix}

We now compare the first term of the Airy system's coefficient matrix in the representation \eqref{eq.sAirymatpolar} to that of the sine system. Recall that by definition,
\[
\sinemat = \frac{1}{2\det\ABM} \ABM^\transpose \ABM
\qquadtext{with}
\ABM \defeq \begin{pmatrix} 1 & -\Re\HBM \\ 0 & \Im\HBM \end{pmatrix}
\]
where \(\HBM\) is a hyperbolic Brownian motion with variance \(\nicefrac{4}{\beta}\) started at \(i\) in the upper half-plane and driven by the complex Brownian motion \(W\). The first term of the coefficient matrix \eqref{eq.sAirymatpolar} has the same structure:
\[
\frac{\diln}{2 e^{-\diffamps} \sin\diffphases} \begin{pmatrix} 1 & e^{-\diffamps} \cos\diffphases \\ e^{-\diffamps} \cos\diffphases & e^{-2\diffamps} \end{pmatrix} = \frac{\diln}{2\det\AiryABM} \AiryABM^\transpose \AiryABM
\quadtext{with}
\AiryABM \defeq \begin{pmatrix} 1 & e^{-\diffamps} \cos\diffphases \\ 0 & e^{-\diffamps} \sin\diffphases \end{pmatrix}.
\]
These representations allow to write, with \(\ABM* \defeq \ABM\circ\slogtime\),
\begin{multline*}
\frac{\diln}{2\det\AiryABM} \AiryABM^\transpose \AiryABM - \frac{1}{2\det\ABM*} \ABM*^\transpose \ABM*
    = \Bigl( \frac{\diln}{2\det\AiryABM} - \frac{1}{2\det\ABM*} \Bigr) \ABM*^\transpose \ABM* \\
    + \frac{\diln}{2\det\AiryABM} \Bigl( (\AiryABM - \ABM*)^\transpose (\AiryABM - \ABM*) + (\AiryABM - \ABM*)^\transpose \ABM* + \ABM*^\transpose (\AiryABM - \ABM*) \Bigr).
\end{multline*}
If we sandwich this matrix between \(\testfunc^*\) and \(\testfunc\), we get
\begin{align}
&\abs[\bigg]{\testfunc^* \Bigl( \frac{\diln}{2\det\AiryABM} \AiryABM^\transpose \AiryABM - \frac{1}{2\det\ABM*} \ABM*^\transpose \ABM* \Bigr) \testfunc} \notag \\
    &\qquad \leq \abs[\Big]{\frac{\diln}{2\det\AiryABM} - \frac{1}{2\det\ABM*}} \norm[\big]{\ABM*\testfunc}_2^2
    + \frac{\diln}{2\det\AiryABM} \Bigl( \norm[\big]{(\AiryABM - \ABM*)\testfunc}_2^2 + 2\abs[\Big]{\Re\inprod[\big]{\ABM*\testfunc}{(\AiryABM-\ABM*)\testfunc}} \Bigr) \notag \\
    \begin{split}
    \label{eq.Airytosine.vaguely.firstterm}
    &\qquad \leq 2\norm{\testfunc}_2^2 \abs[\Big]{\frac{1}{\det\AiryABM} - \frac{1}{\det\ABM*}} \norm{\ABM*}_{\mathrm{max}}^2 + \frac{2\norm{\testfunc}_2^2}{\sqrt{\shift} \det\AiryABM} \norm{\ABM*}_{\mathrm{max}}^2 \\
    &\qquad\hspace*{44mm} + \frac{2\norm{\testfunc}_2^2}{\det\AiryABM} \norm{\AiryABM - \ABM*}_{\mathrm{max}}^2
    + \frac{4\norm{\testfunc}_2^2}{\det\AiryABM} \norm{\ABM*}_{\mathrm{max}} \norm{\AiryABM - \ABM*}_{\mathrm{max}}
    \end{split}
\end{align}
where \(\norm{A}_{\mathrm{max}} \defeq \max_{j,k\in\{1,2\}} \abs{A_{jk}}\) for \(A \in \mathbb{R}^{2\times 2}\), and where we used that \(\diln\) is either \(1\) or \(1 - \nicefrac{1}{\sqrt{\shift}}\) to split the first term. From the results from Section~\ref{sec.asymptotics}, it is not hard to find a good event on which this difference is small.

Indeed, recall that solving the SDE for \(\Im\HBM\) shows that \(\Im\HBM\circ\slogtime = \GBM\), with \(\GBM\) defined as in Proposition~\ref{prop.GBM}, and then the SDE for \(\Re\HBM\) gives
\[
\Re\HBM\circ\slogtime(t) = \frac{2}{\sqrt{\beta}} \int_0^t \GBM(s) \diff{(\Re W\circ\slogtime)}(s).
\]
\begin{subequations}
\label{eq.Airytosine.vaguely.tailbounds}%
Thus, Corollary~\ref{cor.GBM.exp} shows that for \(\shift\) large enough, with \(\timedom \defeq [0, (1 - \shift^{\nicefrac{-1}{2}+\alpha}) / \diln]\) as before,
\beq{eq.Airytosine.vaguely.tailbounddiffIm}
\bprob[\bigg]{\sup_{t\in\timedom} \abs[\Big]{e^{-\diffamps(t)}\sin\diffphases(t) - \Im\HBM\circ\slogtime(t)} \geq \shift^{\nicefrac{-\alpha}{2}}}
    \leq \frac{2}{\log\shift}.
\eeq
Likewise, Proposition~\ref{prop.ReHBM} shows that for some absolute \(C > 0\) and \(\shift\) large enough,
\beq{eq.Airytosine.vaguely.tailbounddiffRe}
\bprob[\bigg]{\sup_{t\in\timedom} \abs[\Big]{e^{-\diffamps(t)}\cos\diffphases(t) + \Re\HBM\circ\slogtime(t)} \geq \shift^{\nicefrac{-\alpha}{5}}}
    \leq \frac{C}{\log\shift}.
\eeq
Together, these two statements allow to control \(\norm{\AiryABM - \ABM*}_{\mathrm{max}}\). Similarly, Corollary~\ref{cor.GBM.exp} also shows that for \(\shift\) large enough,
\beq{eq.Airytosine.vaguely.tailbounddiffdets}
\bprob[\bigg]{\sup_{t\in\timedom} \abs[\Big]{\frac{1}{\det\AiryABM(t)} - \frac{1}{\det\ABM*(t)}} \geq \shift^{-\delta}}
    \leq 2\exp\Bigl( - \frac{\beta\delta^2}{36(1-2\alpha)} \log\shift \Bigr)
\eeq
where \(\delta \defeq \frac{1}{2\beta}\bigl( \alpha(\beta+2) - 1 \bigr) > 0\). 

Then, recall that we have seen in \eqref{eq.boundsupGBM} that for \(\shift\) large enough,
\[
\bprob[\bigg]{\sup_{t\in\timedom} \Im\HBM\circ\slogtime(t) \geq \log\shift}
    \leq \frac{2}{\log\shift}.
\]
On the complementary event, for any \(t \in [0,\lasttime]\),
\[
\quadvar{\Re\HBM\circ\slogtime}(t)
    = \frac{4}{\beta} \int_0^t \abs[\big]{\Im\HBM\circ\slogtime(s)}^2 \diff{\quadvar{\Re W\circ\slogtime}}(s)
    \leq \frac{4}{\beta} \log^2\shift \int_0^t \frac{\diln}{1 - \diln s} \diff{s}
    \leq \frac{2}{\beta} \log^3\shift,
\]
so by Bernstein's inequality, for any \(x > 0\),
\[
\pprob[\bigg]{\biggl\{ \sup_{t\in\timedom} \abs[\big]{\Re\HBM\circ\slogtime(t)} \geq x \biggr\} \cap \biggl\{ \sup_{t\in\timedom} \Im\HBM\circ\slogtime(t) < \log\shift \biggr\}} \leq 2\exp\Bigl( - \frac{\beta x^2}{4\log^3\shift} \Bigr).
\]
Taking \(x = \log^2\shift\), the bound on the probability becomes a negative power of \(\shift\), and it follows that for \(\shift\) large enough,
\beq{eq.Airytosine.vaguely.tailboundX}
\bprob[\bigg]{\sup_{t\in\timedom} \norm{\ABM*(t)}_{\mathrm{max}} \geq \log^2\shift} \leq \frac{3}{\log\shift}.
\eeq

Finally, Proposition~\ref{prop.tracebound} (in particular \eqref{eq.tracebound.integrable} in Remark~\ref{rmk.tracebound}) shows that for any \(\varepsilon, \gamma > 0\), there is a \(C_\varepsilon > 0\) depending only on \(\beta\), \(\gamma\) and \(\varepsilon\) such that
\beq{eq.Airytosine.vaguely.tracebound}
\bprob[\Big]{\forall t \in [0,\lasttime], \frac{1}{\det\AiryABM(t)} \leq \frac{C_\varepsilon}{(1-t)^{\nicefrac{2}{\beta}+\gamma}}} \geq 1 - \frac{\varepsilon}{2}.
\eeq
\end{subequations}

So, by combining all of the tail bounds in \eqref{eq.Airytosine.vaguely.tailbounds}, given \(\varepsilon,\gamma > 0\) we can set
\begin{multline*}
\mathscr{G}_{\shift,\varepsilon,\gamma} \defeq
    \biggl\{ \sup_{t\in\timedom} \norm{\AiryABM(t) - \ABM*(t)}_{\mathrm{max}} < \shift^{\nicefrac{-\alpha}{5}} \biggr\}
    \cap \biggl\{ \sup_{t\in\timedom} \abs[\Big]{\frac{1}{\det\AiryABM(t)} - \frac{1}{\det\ABM*(t)}} < \shift^{-\delta} \biggr\} \\
    \cap \biggl\{ \sup_{t\in\timedom} \norm{\ABM*(t)}_{\mathrm{max}} < \log^2\shift \biggr\} 
    \cap \biggl\{ \forall t\in [0,\lasttime], \frac{1}{\det\AiryABM(t)} \leq \frac{C_\varepsilon}{(1 - t)^{\nicefrac{2}{\beta}+\gamma}} \biggr\}
\end{multline*}
and \(\pprob{\mathscr{G}_{\shift,\varepsilon,\gamma}} \geq 1 - \nicefrac{C}{\log\shift} - \nicefrac{\varepsilon}{2}\) for any \(\shift\) large enough, where \(C > 0\) is an absolute constant. 
To conclude, take \(\gamma > 0\) arbitrary if \(\beta \leq 2\) but such that \(\nicefrac{2}{\beta}+\gamma < 1\) if \(\beta > 2\). On \(\mathscr{G}_{\shift,\varepsilon,\gamma}\), using the representation \eqref{eq.Airytosine.vaguely.firstterm}, we can bound
\begin{align}
\label{eq.Airytosine.vaguely.firstterm2}
&\abs[\bigg]{\int_0^{1-\shift^{\nicefrac{-1}{2}+\alpha}} \testfunc^*(t) \Bigl( \frac{\diln}{2\det\AiryABM} \AiryABM^\transpose \AiryABM - \frac{1}{2\det\ABM*} \ABM*^\transpose \ABM* \Bigr)(t) \testfunc(t) \diff{t}} \\
    &\hspace*{16mm} \leq 2\shift^{-\delta} \log^4\shift \int_0^{1-\shift^{\nicefrac{-1}{2}+\alpha}} \norm{\testfunc(t)}_2^2 \diff{t}
    + 2C_\varepsilon \shift^{\nicefrac{-1}{2}} \log^4\shift \int_0^{1-\shift^{\nicefrac{-1}{2}+\alpha}} \frac{\norm{\testfunc(t)}_2^2}{(1 - t)^{\nicefrac{2}{\beta}+\gamma}} \diff{t} \notag \\
    &\hspace*{32mm} + 2C_\varepsilon \shift^{\nicefrac{-2\alpha}{5}} \int_0^{1-\shift^{\nicefrac{-1}{2}+\alpha}} \frac{\norm{\testfunc(t)}_2^2}{(1 - t)^{\nicefrac{2}{\beta}+\gamma}} \diff{t}
    + 4C_\varepsilon \shift^{\nicefrac{-\alpha}{5}} \log^2\shift \int_0^{1-\shift^{\nicefrac{-1}{2}+\alpha}} \frac{\norm{\testfunc(t)}_2^2}{(1 - t)^{\nicefrac{2}{\beta}+\gamma}} \diff{t}. \notag
\end{align}
All of these integrals are finite. Indeed, if \(\beta \leq 2\), then the integrands are all bounded because the support of \(\testfunc\) is compact in \([0,1)\), while if \(\beta > 2\), then the integrands are all dominated by \(\norm{\testfunc}_\infty^2 (1 - t)^{\nicefrac{-2}{\beta} - \gamma}\), which is integrable by our choice of \(\gamma\). Therefore, the right-hand side here vanishes as \(\shift\to\infty\). Since \(\pprob{\mathscr{G}_{\shift,\varepsilon,\gamma}} \geq 1 - \nicefrac{C}{\log\shift} - \nicefrac{\varepsilon}{2}\), it follows that~\eqref{eq.Airytosine.vaguely.firstterm2} converges to \(0\) in probability as \(\shift\to\infty\).

\step{Convergence of the second term to zero}

It remains to prove that the second term of the coefficient matrix~\eqref{eq.sAirymatpolar} vanishes in the limit. More precisely, given the simplifications made in the first steps of the proof, it remains to prove that for \(\testfunc \in \mathscr{C}_c(\mathcal{I}, \mathbb{C}^2)\),
\[
\int_0^{1-\shift^{\nicefrac{-1}{2}+\alpha}} \testfunc^*(t) \begin{pmatrix}
    e^{2\AMPN(t)} \cos 2\phaseN(t) &
    e^{\AMPN(t)+\AMPD(t)} \cos\bigl( \phaseN(t) + \phaseD(t) \bigr) \\
    e^{\AMPN(t)+\AMPD(t)} \cos\bigl( \phaseN(t) + \phaseD(t) \bigr) &
    e^{2\AMPD(t)} \cos 2\phaseD(t)
\end{pmatrix} \testfunc(t) \diff{t}
\probto[\shift\to\infty] 0
\]
which, by linearity, follows if for any \(\testfunc \in \mathscr{C}_c(\mathcal{I}, \mathbb{R})\),
\beq{eq.Airytosine.vaguely.vaguepart}
\int_0^{\lasttime} \testfunc(t) e^{k\AMPN(t)+l\AMPD(t)} \cos\bigl( k\phaseN(t) + l\phaseD(t) \bigr) \diff{t}
\probto[\shift\to\infty] 0
\eeq
when \(k,l \in \mathbb{N}_0\) satisfy \(k+l = 2\). To prove \eqref{eq.Airytosine.vaguely.vaguepart}, we will work on an event on which \(\exp(k\AMPN+l\AMPD)\) can be controlled. As in other cases, Proposition~\ref{prop.tracebound} guarantees that given \(\varepsilon > 0\), there is a \(C_\varepsilon > 0\) depending on \(\beta\), \(\gamma\) and \(\varepsilon\) such that the event
\beq{eq.Airytosine.vaguely.traceboundevent}
\mathscr{H}_{\shift,\varepsilon} \defeq \biggl\{ \forall t \in [0,\lasttime], e^{k\AMPN(t)+l\AMPD(t)} \leq \frac{C_\varepsilon}{(1- \diln t)^{\nicefrac{2}{\beta}+\gamma}} \biggr\}
\qquadtext{where}
\gamma \defeq \begin{cases}
    \nicefrac{1}{4} - \nicefrac{1}{2\beta} & \text{if } \beta > 2, \\
    \nicefrac{1}{4} & \text{if } \beta \leq 2
\end{cases}
\eeq
has \(\pprob{\mathscr{H}_{\shift,\varepsilon}} \geq 1 - \nicefrac{\varepsilon}{2}\). Note that when \(\beta > 2\), \(\gamma\) is chosen so that \(\nicefrac{2}{\beta} + \gamma = \nicefrac{1}{4} + \nicefrac{3}{2\beta} < 1\) and \(t \mapsto (1-t)^{\nicefrac{-2}{\beta}-\gamma}\) is integrable on \([0,1]\).

Now, we start by replacing the test function \(\testfunc\) by a piecewise constant approximation. To do so, we discretize the time interval with a partition \(\{t_j\}_{j=0}^N\) of \([0, \lasttime]\) defined by setting \(t_0 \defeq 0\), and then recursively
\beq{eq.Airytosine.vaguely.partition}
\diln t_j \defeq \diln t_{j-1} + \frac{\pi \shift^p}{2\diln \shift^{\nicefrac{3}{2}} (1 - \diln t_{j-1})^2}
\qquadtext{where}
p \defeq \begin{cases}
    \nicefrac{1}{4} + \nicefrac{1}{2\beta} & \text{if } \beta > 2, \\
    \nicefrac{1}{4} & \text{if } \beta \leq 2,
\end{cases}
\eeq
until this would give \(\diln t_j > 1 - \nicefrac{1}{\sqrt{\shift}}\), in which case we set \(j \eqdef N\) and \(t_N \defeq \lasttime\) (note that this does happen in a finite number of steps because what is added to \(\diln t_{j-1}\) is always at least \(\frac{\pi}{2} \shift^{\nicefrac{-3}{2}+p}\)). Then, define \(\hat{\testfunc} \colon \mathcal{I}\to\mathbb{R}\) as
\beq{eq.Airytosine.vaguely.piecewiseconstanttestfunc}
\hat{\testfunc} \equiv \hat{\testfunc}_j \defeq \frac{1}{t_j - t_{j-1}} \int_{t_{j-1}}^{t_j} \testfunc(s) \diff{s}
\qquadtext{on} [t_{j-1}, t_j)
\eeq
for each \(j \leq N\), and then \(\hat{\testfunc}(t_N) \defeq \testfunc(\lasttime)\) and \(\hat{\testfunc} \equiv 0\) on \((\lasttime, 1]\). 

Since \(\testfunc\) is continuous and compactly supported in \(\mathcal{I}\), \(\hat{\testfunc}\) is compactly supported in \(\mathcal{I}\), and it is also bounded with \(\norm{\hat{\testfunc}}_\infty \leq \norm{\testfunc}_\infty\). Moreover, since \(\hat{\testfunc}_j\) is the average of \(\testfunc\) in \([t_{j-1}, t_j)\), then for \(t \in [t_{j-1}, t_j)\), it is clear that \(\abs{\testfunc(t) - \hat{\testfunc}_j} \leq \omega_{\testfunc}\bigl( \abs{t_j - t_{j-1}} \bigr)\) where \(\omega_{\testfunc}\) denotes a modulus of continuity for \(\testfunc\). By definition of the \(t_j\)'s, this means that
\[
\abs[\big]{\testfunc(t) - \hat{\testfunc}_j}
    \leq \omega_{\testfunc}\Bigl( \frac{\pi \shift^p}{2\diln^2 \shift^{\nicefrac{3}{2}} (1 - \diln t_{j-1})^2} \Bigr)
    \leq \omega_{\testfunc}\Bigl( \frac{\pi}{2\diln^2 \shift^{\nicefrac{1}{2}-p}} \Bigr),
\]
and as this estimate does not depend on \(j\), it is in fact a bound on \(\norm{(\testfunc - \hat{\testfunc}) \charf{[0,\lasttime]}}_\infty\). It follows that, on \(\mathscr{H}_{\shift,\varepsilon}\), we can indeed replace \(\testfunc\) with \(\hat{\testfunc}\):
\[
\abs[\bigg]{\int_0^{\lasttime} e^{k\AMPN(t)+l\AMPD(t)} \cos\bigl( k\phaseN(t) + l\phaseD(t) \bigr) \bigl( \testfunc(t) - \hat{\testfunc}(t) \bigr) \diff{t}}
    \leq C_\varepsilon \omega_{\testfunc}\Bigl( \frac{\pi}{2\diln^2 \shift^{\nicefrac{1}{2}-p}} \Bigr) \int_0^{\lasttime} \frac{\charf{\supp\hat{\testfunc}}(t)}{(1-t)^{\nicefrac{2}{\beta}+\gamma}} \diff{t},
\]
and the remaining integral is always finite, either because \(\supp\hat{\testfunc}\) is compact (when \(\beta \leq 2\)) or because \(\nicefrac{2}{\beta} + \gamma < 1\) (when \(\beta > 2\)). As the right-hand side vanishes as \(\shift\to\infty\) and \(\pprob{\mathscr{H}_{\shift,\varepsilon}} \geq 1 - \nicefrac{\varepsilon}{2}\), this shows that the integral converges to zero in probability as \(\shift\to\infty\), and therefore it suffices to prove that
\beq{eq.Airytosine.vaguely.vaguepartdiscretized}
\int_0^{\lasttime} \hat{\testfunc}(t) e^{k\AMPN(t)+l\AMPD(t)} \cos\bigl( k\phaseN(t) + l\phaseD(t) \bigr) \diff{t}
\probto[\shift\to\infty] 0.
\eeq

On any interval \([t_{j-1}, t_j)\), since \(\hat{\testfunc}\) is constant, we can integrate by parts. Recall that
\begin{gather*}
\diff{\AMP}(t) = \ampdrift(t) \frac{\diln}{1-\diln t} \diff{t} + \ampdiff(t) \sqrt{\frac{\diln}{1-\diln t}} \diff{\sBM}(t)
\shortintertext{and}
\diff{\phase}(t) = -2\diln \shift^{\nicefrac{3}{2}} (1-\diln t)^2 \diff{t} + \phasedrift(t) \frac{\diln}{1-\diln t} \diff{t} + \phasediff(t) \sqrt{\frac{\diln}{1-\diln t}} \diff{\sBM}(t)
\end{gather*}
where \(\ampdrift\), \(\ampdiff\), \(\phasedrift\) and \(\phasediff\) are constants plus linear combinations of trigonometric functions of multiples of \(\phase\), which are always bounded by constants. Starting from these Itô differentials, an application of Itô's formula shows that if \(k + l = 2\), then
\beq{eq.Airytosine.vaguely.IBPdiff}
\begin{multlined}
e^{-k\AMPN(t)-l\AMPD(t)} \diff{\Bigl( e^{k\AMPN+l\AMPD} \sin\bigl( k\phaseN + l\phaseD \bigr) \Bigr)}(t) \\
    = -4\diln \shift^{\nicefrac{3}{2}} (1 - \diln t)^2 \cos\bigl( k\phaseN(t) + l\phaseD(t) \bigr) \diff{t}
    + \sumdrift(t) \frac{\diln}{1 - \diln t} \diff{t}
    + \sumdiff(t) \sqrt{\frac{\diln}{1-\diln t}} \diff{\sBM}(t)
\end{multlined}
\eeq
where \(\sumdrift\) and \(\sumdiff\) are processes which are bounded by constants depending only on \(\beta\). It follows that for each \(j\),
\begin{align*}
&\int_{t_{j-1}}^{t_j} e^{k\AMPN(t)+l\AMPD(t)} \cos\bigl( k\phaseN(t) + l\phaseD(t) \bigr) \diff{t} \\
    &\hspace*{22mm} = - \frac{1}{4\diln\shift^{\nicefrac{3}{2}}} \int_{t_{j-1}}^{t_j} \frac{1}{(1-\diln t)^2} \biggl( \diff{\Bigl( e^{k\AMPN+l\AMPD} \sin\bigl( k\phaseN+l\phaseD \bigr) \Bigr)}(t) \\
    &\hspace*{44mm} - e^{k\AMPN(t)+l\AMPD(t)} \sumdrift(t) \frac{\diln}{1-\diln t} \diff{t} 
    - e^{k\AMPN(t)+l\AMPD(t)} \sumdiff(t) \sqrt{\frac{\diln}{1-\diln t}} \diff{\sBM}(t) \biggr).
\end{align*}
Integrating by parts and summing up the increments, we get
\begin{subequations}
\label{eq.Airytosine.vaguely.integraldiscretized}
\begin{align}
&\int_0^{\lasttime} \hat{\testfunc}(t) e^{k\AMPN(t)+l\AMPD(t)} \cos\bigl( k\phaseN(t) + l\phaseD(t) \bigr) \hat{\testfunc}(t) \diff{t} \notag \\
    &\hspace*{11mm} = - \frac{1}{4\diln\shift^{\nicefrac{3}{2}}} \sum_{j=1}^N \hat{\testfunc}_j \frac{e^{k\AMPN(t)+l\AMPD(t)} \sin\bigl( \phaseN(t) + l\phaseD(t) \bigr)}{(1-\diln t)^2} \biggr\rvert_{t_{j-1}}^{t_j}
    \label{eq.Airytosine.vaguely.integraldiscretized.IBP} \\
    &\hspace*{22mm} + \frac{1}{4\diln\shift^{\nicefrac{3}{2}}} \int_0^{\lasttime} \hat{\testfunc}(t) e^{k\AMPN(t)+l\AMPD(t)} \Bigl( 2\sin\bigl( k\phaseN(t) + l\phaseD(t) \bigr) + \sumdrift(t) \Bigr) \frac{\diln}{(1 - \diln t)^3} \diff{t}
    \label{eq.Airytosine.vaguely.integraldiscretized.drift} \\
    &\hspace*{22mm} + \frac{1}{4\diln\shift^{\nicefrac{3}{2}}} \int_0^{\lasttime} \hat{\testfunc}(t) e^{k\AMPN(t)+l\AMPD(t)} \sumdiff(t) \frac{\sqrt{\diln}}{(1 - \diln t)^{\nicefrac{5}{2}}} \diff{\sBM}(t).
    \label{eq.Airytosine.vaguely.integraldiscretized.diff}
\end{align}
\end{subequations}

To control \eqref{eq.Airytosine.vaguely.integraldiscretized}, recall that on \(\mathscr{H}_{\shift,\varepsilon}\), for any \(t \in [0, \lasttime]\),
\[
e^{k\AMPN(t)+l\AMPD(t)}
    \leq \frac{C_\varepsilon}{(1 - \diln t)^{\nicefrac{2}{\beta}+\gamma}}.
\]
By definition of \(\diln\) and \(\lasttime\), this upper bound is always bounded by \(C_\varepsilon \shift^{\nicefrac{1}{\beta}+\nicefrac{\gamma}{2}}\). When \(\beta > 2\), we defined \(\gamma = \nicefrac{1}{4} - \nicefrac{1}{2\beta}\) and \(p = \nicefrac{1}{4} + \nicefrac{1}{2\beta}\), so \(\nicefrac{1}{\beta} + \nicefrac{\gamma}{2} = p - \nicefrac{\gamma}{2}\). When \(\beta \leq 2\), we took \(p = \gamma = \nicefrac{1}{4}\) so \(p - \nicefrac{\gamma}{2} > 0\) and it is still true that \((1 - \diln t)^{\nicefrac{-2}{\beta}-\gamma} \leq \shift^{p-\nicefrac{\gamma}{2}}\) for \(\shift\) large enough and \(t \in \supp\hat{\testfunc}\), as the latter set is compact in \([0,1)\). This shows that in any case,
\beq{eq.Airytosine.vaguely.expestimate}
e^{k\AMPN(t)+l\AMPD(t)}
    \leq C_\varepsilon \shift^{p - \nicefrac{\gamma}{2}}
\qquadtext{where}
p - \frac{\gamma}{2} = \begin{cases}
    \nicefrac{1}{8} + \nicefrac{3}{4\beta} & \text{if } \beta > 2, \\
    \nicefrac{1}{8} & \text{if } \beta \leq 2
\end{cases}
\eeq
for \(t \in \supp\hat{\testfunc}\) and \(\shift\) large enough.

Now, on \(\mathscr{H}_{\shift,\varepsilon}\), the estimate \eqref{eq.Airytosine.vaguely.expestimate} shows that \eqref{eq.Airytosine.vaguely.integraldiscretized.IBP} is bounded for \(\shift\) large enough by
\[
\frac{C_\varepsilon \norm{\testfunc}_\infty \shift^{p - \nicefrac{\gamma}{2}}}{4\diln\shift^{\nicefrac{3}{2}}} \sum_{j=1}^N \Bigl( \frac{1}{(1 - \diln t_j)^2} + \frac{1}{(1 - \diln t_{j-1})^2} \Bigr)
    = \frac{C_\varepsilon \norm{\testfunc}_\infty \shift^{p-\nicefrac{\gamma}{2}}}{4\diln\shift^{\nicefrac{3}{2}}} \biggl( 1 + 2\sum_{j=1}^{N-1} \frac{1}{(1 - \diln t_j)^2} + \shift \biggr).
\]
We directly see that the first and last terms vanish as \(\shift\to\infty\) since they are of order \(\shift^{\nicefrac{-3}{2}+p-\nicefrac{\gamma}{2}}\) and \(\shift^{\nicefrac{-1}{2}+p-\nicefrac{\gamma}{2}}\) respectively while \(p - \nicefrac{\gamma}{2} < \nicefrac{1}{2}\) in any case. Then, using the definition \eqref{eq.Airytosine.vaguely.partition} of the times \(t_j\), we can easily see that the second term also vanishes as \(\shift\to\infty\):
\[
\frac{C_\varepsilon \norm{\testfunc}_\infty}{\pi\shift^{\nicefrac{\gamma}{2}}} \sum_{j=1}^{N-1} \frac{\pi\shift^p}{2\diln\shift^{\nicefrac{3}{2}}(1-\diln t_j)^2}
    \leq \frac{C_\varepsilon \norm{\testfunc}_\infty}{\pi\shift^{\nicefrac{\gamma}{2}}}\, \diln t_N
    \to 0.
\]
As this holds on \(\mathscr{H}_{\shift,\varepsilon}\), this shows that \eqref{eq.Airytosine.vaguely.integraldiscretized.IBP} converges to \(0\) in probability as \(\shift\to\infty\).

To control \eqref{eq.Airytosine.vaguely.integraldiscretized.drift}, note that the estimate \eqref{eq.Airytosine.vaguely.expestimate} implies that this term is bounded on \(\mathscr{H}_{\shift,\varepsilon}\) for \(\shift\) large enough by
\[
\frac{C_\varepsilon \norm{\testfunc}_\infty (2 + \norm{\sumdrift}_\infty) \shift^{p-\nicefrac{\gamma}{2}}}{4\diln\shift^{\nicefrac{3}{2}}} \int_0^{\lasttime} \frac{\diln}{(1 - \diln t)^3} \diff{t}
    \leq \frac{C_\varepsilon \norm{\testfunc}_\infty (2 + \norm{\sumdrift}_\infty)}{2\diln} \shift^{\nicefrac{-1}{2}+p-\nicefrac{\gamma}{2}}.
\]
Since \(p - \nicefrac{\gamma}{2} < \nicefrac{1}{2}\), this vanishes as \(\shift\to\infty\), and it follows that \(\eqref{eq.Airytosine.vaguely.integraldiscretized.drift}\) converges to \(0\) in probability as \(\shift\to\infty\).

Finally, using again the estimate \eqref{eq.Airytosine.vaguely.expestimate}, we see that the quadratic variation of \eqref{eq.Airytosine.vaguely.integraldiscretized.diff} is bounded on \(\mathscr{H}_{\shift,\varepsilon}\) for \(\shift\) large enough by
\[
\frac{C_\varepsilon^2 \norm{\testfunc}_\infty^2 \norm{\sumdiff}_\infty^2 \shift^{2p-\gamma}}{16\diln^2\shift^3} \int_0^{\lasttime} \frac{\diln}{(1 - \diln t)^5} \diff{t}
    \leq \frac{C_\varepsilon^2 \norm{\testfunc}_\infty^2 \norm{\sumdiff}_\infty^2}{16\diln^2} \shift^{-1+2p-\gamma}.
\]
Therefore, by Bernstein's inequality, for any \(x > 0\),
\[
\pprob[\bigg]{\mathscr{H}_{\shift,\varepsilon} \cap \biggl\{ \abs[\Big]{\frac{1}{4\diln\shift^{\nicefrac{3}{2}}} \int_0^{\lasttime} \hat{\testfunc}(t) e^{k\AMPN(t)+l\AMPD(t)} \sumdiff(t) \frac{\sqrt{\diln}}{(1 - \diln t)^{\nicefrac{5}{2}}} \diff{\sBM}(t)} \geq x \biggr\}} \leq 2\exp\Bigl( - \frac{\diln \shift^{1-2p+\gamma} x^2}{C_\varepsilon^2 \norm{\testfunc}_\infty^2 \norm{\sumdiff}_\infty^2} \Bigr).
\]
Again, since \(p - \nicefrac{\gamma}{2} < \nicefrac{1}{2}\), \(\shift^{1-2p+\gamma} \to \infty\) as \(\shift\to\infty\), so for any \(x > 0\) and any \(\shift\) large enough, the right-hand side is bounded by \(\nicefrac{\varepsilon}{2}\). As \(\pprob{\mathscr{H}_{\shift,\varepsilon}} \geq 1 - \nicefrac{\varepsilon}{2}\) by definition, it follows that~\eqref{eq.Airytosine.vaguely.integraldiscretized.diff} converges to \(0\) in probability as \(\shift\to\infty\). 

As all three terms of \eqref{eq.Airytosine.vaguely.integraldiscretized} converge to \(0\) in probability as \(\shift\to\infty\), the second term of the shifted Airy system's coefficient matrix does converge vaguely to \(0\) as \(\shift\to\infty\), and this concludes the proof.
\end{proof}

\subsection{Convergence of transfer matrices}

A rather direct consequence of the vague convergence of canonical systems proved in Theorem~\ref{thm.Airytosine.vaguely} is that the corresponding transfer matrices converge compactly in probability, which proves~Corollary~\ref{cor.TMconvinlaw} given in the introduction. More precisely, we can prove the compact convergence of the transfer matrices on our probability space constructed in Section~\ref{sec.coupling}.

\begin{theorem}
\label{thm.Airytosine.TMconvergence}
\setshift{E_n}
Let \(\mathcal{I} \defeq [0,1)\) if \(\beta \leq 2\) and \(\mathcal{I} \defeq [0,1]\) if \(\beta > 2\). Let \(\sAirymat* \defeq \timechange' (\sAirymat\circ\timechange)\) and \(\sinemat* \defeq \sinemat\circ\logtime\), and let \(\sAiryTM*, \sineTM* \colon \mathcal{I} \times \mathbb{C} \to \mathbb{C}^{2\times 2}\) be the transfer matrices of the associated canonical systems, both defined on the probability space from Lemma~\ref{lem.probspace}. Then \(\sAiryTM* \to \sineTM*\) compactly on \(\mathcal{I}\times\mathbb{C}\) in probability as \(n\to\infty\).
\end{theorem}

\begin{proof}
By Theorem~\ref{thm.Airytosine.vaguely} and Proposition~\ref{prop.CS.random.TMconvergence}, it suffices to find for any \(\varepsilon > 0\) functions \(f_\varepsilon, g_\varepsilon \in L^1\loc(\mathcal{I})\) such that \(\tr\sinemat\circ\logtime \leq f_\varepsilon\) and \(\tr \timechange'(\sAirymat\circ\timechange) \leq g_\varepsilon\) with probability at least \(1 - \varepsilon\) for any \(\shift\) large enough.

We start by finding a suitable bound for \(\tr\sinemat\circ\logtime\). Recall that by definition,
\[
\tr \sinemat\circ\logtime
    = \frac{1}{2\Im\HBM\circ\logtime} \Bigl( 1 + \abs{\HBM\circ\logtime}^2 \Bigr)
\]
where \(\HBM\) is a hyperbolic Brownian motion started at \(i\) with variance \(\nicefrac{4}{\beta}\) and driven by the Brownian motion \(W\) from Lemma~\ref{lem.probspace}. Now, by properties of Brownian motion (for instance, applying Proposition~\ref{prop.martingalebound}), for any \(\varepsilon > 0\) there is a \(\tilde{C}_\varepsilon > 0\) and an event with probability at least \(1 - \nicefrac{\varepsilon}{2}\) on which
\(
\abs{\Im W\circ\logtime(t)}
    \leq \tilde{C}_\varepsilon \bigl( 1 + \logtime(t)^{\nicefrac{3}{4}} \bigr)
\)
for any \(t \in [0,1)\), and therefore for any \(\delta > 0\) there is a \(C_\varepsilon > 0\) such that on that event, for any \(t \in [0, 1)\),
\[
\frac{1}{C_\varepsilon} (1 - t)^{\nicefrac{2}{\beta}+\delta} \leq \Im\HBM\circ\logtime(t) \leq C_\varepsilon (1 - t)^{\nicefrac{2}{\beta}-\delta}.
\]
If \(\Im\HBM\circ\logtime\) is bounded in that way, then as long as \(\delta < \nicefrac{2}{\beta}\),
\[
\quadvar{\Re\HBM\circ\logtime}(t)
    \leq \frac{4C_\varepsilon^2}{\beta} \int_0^t (1 - s)^{\nicefrac{2}{\beta}-\delta} \diff{\quadvar{\Re W\circ\logtime}}(s)
    = \frac{4C_\varepsilon^2}{\beta} \int_0^t (1 - s)^{\nicefrac{2}{\beta}-\delta - 1} \diff{s}
    \leq \frac{4C_\varepsilon^2}{2 - \beta\delta}. 
\]
Hence, Bernstein's inequality shows that for any \(x > 0\),
\[
\pprob[\bigg]{\Bigl\{ \abs[\big]{\Im W\circ\logtime(t)}
    \leq \tilde{C}_\varepsilon \bigl( 1 + \logtime(t)^{\nicefrac{3}{4}} \bigr) \Bigr\} \cap \Bigl\{ \sup_{t\in [0,1)} \abs{\Re\HBM\circ\logtime(t)} \geq x\Bigr\}} \leq 2 \exp\Bigl( - \frac{(2 - \beta\delta) x^2}{8C_\varepsilon^2} \Bigr),
\]
and this bound can be made smaller than \(\nicefrac{\varepsilon}{2}\) by taking \(x\) large enough. Combining the above, we get for such an \(x\) that
\[
\bprob[\bigg]{\forall t\in [0,1), \tr\sinemat\circ\logtime(t) \leq \frac{C_\varepsilon ( 1 + C_\varepsilon^2 + x^2)}{2(1 - t)^{\nicefrac{2}{\beta}+\delta}}} \geq 1 - \varepsilon.
\]
This dominating function is always \(L^1\loc[0,1)\), and it can be taken to be integrable on \([0,1]\) when \(\beta > 2\) by choosing \(\delta\) small enough so that \(\nicefrac{2}{\beta} + \delta < 1\). 

When \(\beta > 2\), we have also already seen it in \eqref{eq.tracebound.integrable} that Proposition~\ref{prop.tracebound} gives, for any \(\varepsilon > 0\) and \(\delta > 0\) small enough so that \(\nicefrac{2}{\beta} + \delta < 1\), an event of probability at least \(1 - \varepsilon\) on which for any \(t \in [0,1)\),
\beq{eq.Airytosine.TMconvergence.bound}
\tr\bigl( \timechange'(t) \sAirymat\circ\timechange(t) \bigr)
    \leq e^{2\AMPN(t)} + e^{2\AMPD(t)}
    \leq \frac{2C_\varepsilon}{(1-t)^{\nicefrac{2}{\beta}+\delta}}
\eeq
for a suitable \(C_\varepsilon > 0\). As this upper bound is integrable on \([0,1)\), the proof is complete for \(\beta > 2\).

When \(\beta \leq 2\), we slightly modify the argument. Proposition~\ref{prop.tracebound} shows in the same way that for any \(\varepsilon, \delta > 0\), there is an event of probability at least \(1 - \varepsilon\) on which the bound~\eqref{eq.Airytosine.TMconvergence.bound} holds for any \(t \in [0,\lasttime) = [0, 1-\nicefrac{1}{\sqrt{\shift}})\). This means that for any compact \([0,b] \subset [0,1)\), the bound~\eqref{eq.Airytosine.TMconvergence.bound} holds uniformly on \([0,b]\) for all \(\shift\) large enough so that \(b < 1 - \nicefrac{1}{\sqrt{\shift}}\), and therefore Theorem~\ref{thm.Airytosine.vaguely} and Proposition~\ref{prop.CS.random.TMconvergence} show that the transfer matrices converge compactly in probability on \([0,b] \times \mathbb{C}\). Since \(b\) is arbitrary in \([0,1)\), we can deduce from this the full compact convergence in probability on \([0,1) \times \mathbb{C}\). Indeed, taking a sequence of compact sets \(K_n \subset [0,1) \times \mathbb{C}\) such that \(\bigcup_{n\in\mathbb{N}} K_n = [0,1) \times \mathbb{C}\), a metric on \(\TMLP[0,1)\) is
\[
d(T, T') = \sum_{n=1}^\infty \frac{1}{2^n} \frac{d_n(T, T')}{1+d_n(T, T')}
\qquadtext{where}
d_n(T, T') \defeq \sup_{(t,z)\in K_n} \abs[\big]{T(t,z) - T'(t,z)}.
\]
Now, \(\sAiryTM* \to \sineTM*\) uniformly on \(K_n\) for any \(n\) by the previous argument, so given \(\varepsilon, \zeta > 0\), we can first take \(N\) large enough so that \(\sum_{n=N+1}^\infty 2^{-n} < \nicefrac{\zeta}{2}\), and then take \(\shift_0\) large enough so that for each \(n \leq N\) and each \(\shift \geq \shift_0\), \(d_n\bigl( \sAiryTM*, \sineTM* \bigr) \leq \nicefrac{\zeta}{2N}\) with probability at least \(1 - \nicefrac{\varepsilon}{N}\). Combining, \(d\bigl( \sAiryTM*, \sineTM* \bigr) \leq \zeta\) with probability at least \(1-\varepsilon\) for all \(\shift\) large enough, that is, \(\sAiryTM* \to \sineTM*\) compactly on \(\mathcal{I}\times\mathbb{C}\) in probability.
\end{proof}

\section{Spectral convergence of the canonical systems}
\label{sec.WTconv}

The purpose of this section is to prove the convergence of the canonical systems' Weyl--Titchmarsh functions stated in Theorem~\ref{thm.WTconvinlaw}. To do this, we prove that the convergence of the Weyl--Titchmarsh functions holds in probability on the probability space from Lemma~\ref{lem.probspace}. By results from canonical systems theory, there is really only something to prove when \(\beta > 2\). Indeed, when \(\beta \leq 2\), Theorem~\ref{thm.Airytosine.TMconvergence} directly yields the following, which proves Theorem~\ref{thm.WTconvinlaw} for \(\beta \leq 2\).

\begin{corollary}
\setshift{E_n}
Let \(\sAiryWT\) and \(\sineWT\) be the Weyl--Titchmarsh functions of the shifted Airy system and of the sine system respectively, with both systems defined on the probability space from Lemma~\ref{lem.probspace}. When \(\beta \leq 2\), \(\sAiryWT \to \sineWT\) compactly on \(\UHP\) in probability as \(n\to\infty\), and in particular the corresponding spectral measures also converge in probability.
\end{corollary}

\begin{proof}
By Theorem~\ref{thm.Airytosine.TMconvergence}, the transfer matrices of these systems converge in probability as \(\shift\to\infty\). Hence, by continuity of the map \(\TMLP[0,1) \to \Hol(\UHP,\clUHP)\) which sends a transfer matrix to the corresponding Weyl--Titchmarsh function (see Theorem~\ref{thm.TMtoWT.LP}), \(\sAiryWT \to \sineWT\) compactly in probability. In the same way, the convergence of the spectral measures follows from the continuity of the Herglotz representation map.
\end{proof}

Given this result, for the remainder of the section, we suppose that \(\beta > 2\). In that case, the convergence of the Weyl--Titchmarsh functions will be a consequence of the continuity of the map \(\TMLC[0,1] \times \mathscr{C}(\UHP,\mathbb{C}^2) \to \RiemannSphere^{\UHP}\) sending a transfer matrix and a right boundary condition to a Weyl--Titchmarsh function (i.e., Theorem~\ref{thm.TMtoWT.LC}), so what remains to be proven, essentially, is that the right boundary conditions converge. We first present the heuristic idea behind the convergence of the boundary conditions.

The right boundary condition of the sine system is a vector parallel to \((q,1)\) where \(q \defeq \lim_{t\to\infty} \Re\HBM(t)\). The time-changed shifted Airy system is limit point at \(1 + \nicefrac{1}{\sqrt{\shift}}\), so the correct \enquote{right boundary condition} to attach to its restriction to \((0,1)\) in order to reproduce the same Weyl--Titchmarsh function is simply the value at \(1\) of an integrable solution \(u_z\circ\timechange\) to \(J(u_z\circ\timechange)' = -z \timechange' (\sAirymat\circ\timechange) (u_z\circ\timechange)\) on \((0, 1+\nicefrac{1}{\sqrt{\shift}})\), which therefore varies with \(z \in \UHP\). By construction of the canonical system, such a solution \(u_z\) always has the form
\[
u_z = \sAirySLtoCS^{-1} \begin{pmatrix} \quasi{h_z} \\ h_z \end{pmatrix}
    = \begin{pmatrix} \shift^{\nicefrac{-1}{4}} \sAiryDirichlet & - \shift^{\nicefrac{-1}{4}} \sAiryDirichlet' \\ -\shift^{\nicefrac{1}{4}} \sAiryNeumann & \shift^{\nicefrac{1}{4}} \sAiryNeumann' \end{pmatrix} \begin{pmatrix} h_z' \\ h_z \end{pmatrix}
\]
where \(h_z\) is an integrable solution to \(\sAiryop h_z = zh_z\) and where \(\sAirySLtoCS\) is the matrix \eqref{eq.sAirySLtoCS}. 

To see heuristically how \(u_z\circ\timechange(1)\) can converge to the right boundary condition \((\Re\HBM(\infty), 1)\) of the sine system, we would like to see the process \(-\exp(-\diffamps-i\diffphases)\) appear in the expression of \(u_z\), since that process is the one that converges to \(\HBM\) (as we saw in Propositions~\ref{prop.GBM} and~\ref{prop.ReHBM}). To make this process appear, we use the representations~\eqref{eq.polarcoordsDN} of \(\sAiryDirichlet\) and \(\sAiryNeumann\) in the polar coordinates from Proposition~\ref{prop.polarcoords}. More precisely, we write \(\sAiryDirichlet\) in its polar coordinates, and we express these polar coordinates in terms of those of \(\sAiryNeumann\) plus the differences \(\diffamps\) and \(\diffphases\). This gives
\begin{align*}
\shift^{\nicefrac{-1}{4}} \sAiryDirichlet\circ\timechange(1)
    & = e^{\AMPD(1)} \cos\phaseD(1) \\
    & = e^{-\diffamps(1)} \cos\diffphases(1) e^{\AMPN(1)} \cos\phaseN(1) + e^{-\diffamps(1)} \sin\diffphases(1) e^{\AMPN(1)} \sin\phaseN(1),
\end{align*}
and we retrieve the expression in polar coordinates of \(\shift^{\nicefrac{1}{4}} \sAiryNeumann\circ\timechange(1)\) in the first term. The Wronskian identity \eqref{eq.ampsphasesWronskianidentity}, which gives \(e^{-\diffamps}\sin\diffphases \equiv e^{-2\AMPN}\), also simplifies the prefactor of the second term, and this yields
\[
\shift^{\nicefrac{-1}{4}} \sAiryDirichlet\circ\timechange(1)
    = e^{-\diffamps(1)} \cos\diffphases(1) \shift^{\nicefrac{1}{4}} \sAiryNeumann\circ\timechange(1) + e^{-\AMPN(1)} \sin\phaseN(1).
\]
The same manipulations can be applied to \(\shift^{\nicefrac{-1}{4}} \sAiryDirichlet'\circ\timechange(1)\), and yield
\[
\shift^{\nicefrac{-1}{4}} \sAiryDirichlet'\circ\timechange(1)
    = e^{-\diffamps(1)} \cos\diffphases(1) \shift^{\nicefrac{1}{4}} \sAiryNeumann'\circ\timechange(1) - e^{-\AMPN(1)} \cos\phaseN(1).
\]
It follows that
\begin{multline*}
\begin{pmatrix}
    \shift^{\nicefrac{-1}{4}} \sAiryDirichlet & -\shift^{\nicefrac{-1}{4}} \sAiryDirichlet' \\ -\shift^{\nicefrac{1}{4}} \sAiryNeumann & \shift^{\nicefrac{1}{4}} \sAiryNeumann'
\end{pmatrix} \bigl( \timechange(1) \bigr) \\
    = \shift^{\nicefrac{1}{4}} \begin{pmatrix} -e^{-\diffamps(1)} \cos\diffphases(1) \\ 1 \end{pmatrix} \begin{pmatrix} -\sAiryNeumann\circ\timechange(1) & \sAiryNeumann'\circ\timechange(1) \end{pmatrix} + e^{-\AMPN(1)} \begin{pmatrix} \sin\phaseN(1) & \cos\phaseN(1) \\ 0 & 0 \end{pmatrix},
\end{multline*}
and therefore that
\beq{eq.integrablesolutiondecomposition}
u_z\circ\timechange(1)
    = \shift^{\nicefrac{1}{4}} \wronskian{h_z}{\sAiryNeumann}(\shift-1) \begin{pmatrix} -e^{-\diffamps(1)} \cos\diffphases(1) \\ 1 \end{pmatrix} + e^{-\AMPN(1)} \begin{pmatrix} \sin\phaseN(1) & \cos\phaseN(1) \\ 0 & 0 \end{pmatrix} \begin{pmatrix} h_z'(\shift-1) \\ h_z(\shift-1) \end{pmatrix}
\eeq
where \(\wronskian{f}{g}\) denotes the Wronskian of \(f\) and \(g\), and where we used that \(\timechange(1) = \shift-1\). Note that the vector remaining in the first term should converge to the right boundary condition of the sine system, since \(-\exp(-\diffamps-i\diffphases)\) converges to the hyperbolic Brownian motion by the results of Section~\ref{sec.asymptotics}. The solution \(h_z\) is only defined up to normalization, so this suggests that we should normalize it so that the Wronskian prefactor is set to 1. The problem would then be reduced to showing that the second term of~\eqref{eq.integrablesolutiondecomposition} vanishes in the limit, which is reasonable to expect given the asymptotic behavior of \(e^{-\AMPN}\) from Corollary~\ref{cor.GBM.exp}.

To see how such a normalization on \(h_z\) can be chosen, notice that
\begin{multline*}
\frac{1}{2\sqrt{\shift}} \sAiryop f(\shift + t)
    = - \Bigl( f' - \frac{2}{\sqrt{\beta}}\bigl( \sBM - \sBM(\shift) \bigr) f \Bigr)'(\shift + t) - \frac{2}{\sqrt{\beta}} \bigl( \sBM(\shift + t) - \sBM(\shift) \bigr) \Bigl( f' - \frac{2}{\sqrt{\beta}} \bigl( \sBM - \sBM(\shift) \bigr) f \Bigr)(\shift + t) \\
    + \Bigl( t - \frac{4}{\beta} \bigl( \sBM(\shift + t) - \sBM(\shift) \bigr)^2 \Bigr) f(\shift + t),
\end{multline*}
as the extra \(\sBM(\shift)\) terms cancel out. This shows that if \(f\) solves \(\sAiryop f = zf\), then \(\tilde{f}(t) \defeq f(\shift + t)\) solves \(\unsAiryop \tilde{f} = \frac{z}{2\sqrt{\shift}} \tilde{f}\) where \(\unsAiryop\) has the law of \(\Airyop\), but is defined from the Brownian motion \(t \mapsto \sBM(\shift+t) - \sBM(\shift)\). Therefore, for any \(t \geq -\shift\),
\beq{eq.solutionseqlaw}
\bigl( h_z(\shift + t), \sAiryNeumann(\shift + t) \bigr) = \bigl( a\intsol{\nicefrac{z}{2\sqrt{\shift}}}(t), \unsAiryNeumann(t) \bigr)
\eeq
where \(\unsAiryNeumann\) solves \(\unsAiryop\unsAiryNeumann = 0\) with \(\unsAiryNeumann(-\shift) = 0\) and \(\unsAiryNeumann'(-\shift) = 1\), \(\intsol{\lambda}\) is an integrable solution to \(\unsAiryop\intsol{\lambda} = \lambda\intsol{\lambda}\), and \(a \in \mathbb{C}\setminus\{0\}\) is a (possibly random) scaling factor. Here, the integrable solution \(\intsol{\lambda}\) is determined only up to a multiplicative constant (which even depends on \(\lambda\) in general) so the scaling factor \(a\) is somewhat redundant, but it will be useful in the sequel since we will set it in terms of a fixed \(\intsol{\lambda}\). By basic properties of generalized Sturm--Liouville operators (see e.g.~\cite[Theorem~2.7]{eckhardt_weyl-titchmarsh_2013}), we may (and do) choose \(\intsol{\lambda}\) to be analytic in \(\lambda\). In particular, \(\intsol{\lambda}\) could be chosen to be the stochastic Airy function \(\SAi_{-\lambda}\) from~\cite{lambert_strong_2021}. 

The equivalence of solutions~\eqref{eq.solutionseqlaw} suggests a way to choose \(h_z\). Indeed, if the Wronskian \(\shift^{\nicefrac{1}{4}} \wronskian{\intsol{\nicefrac{z}{2\sqrt{\shift}}}}{\unsAiryNeumann}(-1)\) stays bounded away from zero, we can choose its inverse as the constant \(a\) which normalizes the integrable solution \(\Phi_{z/2\sqrt{\shift}}\). Doing so removes the prefactor of the first term in \eqref{eq.integrablesolutiondecomposition}, and thus with this setup, the right boundary condition converges provided the second term of~\eqref{eq.integrablesolutiondecomposition} vanishes in the limit.

Now, notice that the Wronskian satisfies
\[
\wronskian{\intsol{\nicefrac{z}{2\sqrt{\shift}}}}{\unsAiryNeumann}(-1)
    = \wronskian{\intsol{0}}{\unsAiryNeumann}(-1) + \bigl( \intsol{\nicefrac{z}{2\sqrt{\shift}}} - \intsol{0} \bigr)(-1) \unsAiryNeumann'(-1) - \bigl( \intsol{\nicefrac{z}{2\sqrt{\shift}}}' - \intsol{0}' \bigr)(-1) \unsAiryNeumann(-1)
\]
where all primes denote time derivatives. Provided the factors \(\unsAiryNeumann(-1)\) and \(\unsAiryNeumann'(-1)\) can be controlled (which we will verify later), the last two terms vanish by analyticity of \(\lambda \mapsto \intsol{\lambda}\). Since the first term is equal to \(\intsol{0}(-\shift)\), the solution to the problem essentially lies in understanding how \(\intsol{0}\) behaves towards \(-\infty\). 

The rest of this section is organized as follows. We first derive polar coordinates for \(\intsol{0}\) on the negative real line, in the same way as we did earlier in Section~\ref{sec.polarcoords}. Then, we derive asymptotics for the radial coordinate and we show that the phase becomes uniformly distributed over the unit circle, which proves Theorem~\ref{thm.asymptoticsSAi} stated in the introduction. Finally, we put all of the pieces together and prove the convergence of the Weyl--Titchmarsh functions in Theorem~\ref{thm.Airytosine.WTconvergence}, a stronger version of Theorem~\ref{thm.WTconvinlaw}.

\subsection{Polar coordinates and their asymptotic behavior}

We want to find polar coordinates for a solution \(f\) to \(\Airyop f = 0\) on the negative real line. To do so, we must have \(\Airyop\) defined on the whole of \(\mathbb{R}\), and therefore we now understand \(\Airyop\) as being defined from a \emph{two-sided} Brownian motion. We now prove the following result, which reproduces the idea of Proposition~\ref{prop.polarcoords}, but going backwards in time.

\begin{proposition}
\label{prop.polarcoordsSAi}
Let \(\Airyop\) be defined from a two-sided standard Brownian motion \(B\), and let \(\rBM(t) \defeq B(-t)\). If \(f\) solves \(\Airyop f = 0\) and \(f(-1)\) and \(f'(-1)\) are independent of the restriction of \(\rBM\) to \((1,\infty)\), then for \(t \geq 1\),
\[
f(-t) = t^{\nicefrac{-1}{4}} e^{\ramp(t)} \cos\rphase(t)
\qquadtext{and}
f'(-t) = -t^{\nicefrac{1}{4}} e^{\ramp(t)} \sin\rphase(t)
\]
where \(\ramp\) and \(\rphase\) solve the stochastic differential equations
\begin{align*}
\diff{\ramp}(t)
    & = \frac{1}{t} \biggl( \frac{1}{2\beta} + \Bigl( \frac{1}{4} + \frac{1}{\beta} \Bigr) \cos 2\rphase(t) + \frac{1}{2\beta} \cos 4\rphase(t) \biggr) \diff{t} - \frac{1}{\sqrt{\beta t}} \sin 2\rphase(t) \diff{\rBM}(t), \\
\diff{\rphase}(t)
    & = - \sqrt{t} \diff{t} - \frac{1}{t} \biggl( \Bigl( \frac{1}{4} + \frac{1}{\beta} \Bigr) \sin 2\rphase(t) + \frac{1}{2\beta} \sin 4\rphase(t) \biggr) \diff{t} - \frac{2}{\sqrt{\beta t}} \cos^2 \rphase(t) \diff{\rBM}(t)
\end{align*}
with \(\ramp(1) = \frac{1}{2} \log\bigl( f^2(-1) + {f'}^2(-1) \bigr)\) and \(\rphase(1) = \arctan\bigl( \nicefrac{f'(-1)}{f(-1)} \bigr)\).
\end{proposition}

\begin{proof}
Set \(y(t) \defeq f(-t)\). Reversing time in the equation \(\Airyop f = 0\) shows that
\[
0 = - \Bigl( y' + \frac{2}{\sqrt{\beta}} \rBM y \Bigr)'(t) + \frac{2}{\sqrt{\beta}} \rBM(t) \Bigl( y' + \frac{2}{\sqrt{\beta}} \rBM y \Bigr)(t) - \Bigl( t + \frac{4}{\beta} \rBM^2(t) \Bigr) y(t),
\]
and then applying Itô's formula in the same way as we did in Section~\ref{sec.AirysineCS.defAiry} (which is justified by the independence of \(y(1)\) and \(y'(1)\) on the restriction of \(\rBM\) to \((1, \infty)\)) shows that \(y'\) satisfies
\[
\diff{y'}(t) = - t y(t) \diff{t} - \frac{2}{\sqrt{\beta}} y(t) \diff{\rBM}(t).
\]
Now, define the real-valued processes \(\ramp\) and \(\rphase\) so that \(e^{\ramp(t) + i\rphase(t)} = t^{\nicefrac{1}{4}} y(t) + it^{\nicefrac{-1}{4}} y'(t)\). By Itô's formula,
\begin{multline*}
\diff{\ramp}(t) + i\diff{\rphase}(t)
    = \frac{1}{t^{\nicefrac{1}{4}}y(t) + it^{\nicefrac{-1}{4}}y'(t)} \bigg( \frac{1}{4} t^{\nicefrac{-3}{4}} y(t) \diff{t} - \frac{i}{4} t^{\nicefrac{-5}{4}} y'(t) \diff{t} + t^{\nicefrac{1}{4}} y'(t) \diff{t} + it^{\nicefrac{-1}{4}} \diff{y'(t)} \biggr) \\
    + \frac{1}{2} \frac{1}{\bigl( t^{\nicefrac{1}{4}}y(t) + it^{\nicefrac{-1}{4}}y'(t) \bigr)^2} t^{\nicefrac{-1}{2}} \diff{\quadvar{y'}}(t)
\end{multline*}
so
\begin{align*}
\SwapAboveDisplaySkip
\diff{\ramp}(t) + i\diff{\rphase}(t)
    & = e^{-2\ramp(t)} \biggl( \Bigl( \frac{1}{4} t^{\nicefrac{-1}{2}} y^2(t) - \frac{1}{4} t^{\nicefrac{-3}{2}} {y'}^2(t) \Bigr) \diff{t} - \frac{2}{\sqrt{\beta t}} y(t) y'(t) \diff{\rBM}(t) \\
    &\hspace*{22mm} - i\Bigl( \frac{1}{2t} y(t) y'(t) + {y'}^2(t) + ty^2(t) \Bigr) \diff{t} - \frac{2i}{\sqrt{\beta}} y^2(t) \diff{\rBM}(t) \biggr) \\
    &\hspace*{22mm} + \frac{2}{\beta} e^{-4\ramp(t)} t^{\nicefrac{-1}{2}} y^2(t) \Bigl( t^{\nicefrac{1}{2}} y^2(t) - t^{\nicefrac{-1}{2}} {y'}^2(t) - 2iy(t)y'(t) \Bigr) \diff{t}.
\intertext{As \(e^{\ramp(t)}\cos\rphase(t) = t^{\nicefrac{1}{4}} y(t)\) and \(e^{\ramp(t)} \sin\rphase(t) = t^{\nicefrac{-1}{4}} y'(t)\), this simplifies to}
\diff{\ramp}(t) + i\diff{\rphase}(t)
    & = \frac{1}{4t} \bigl( \cos^2\rphase(t) - \sin^2\rphase(t) \bigr) \diff{t} - \frac{2}{\sqrt{\beta t}} \cos\rphase(t) \sin\rphase(t) \diff{\rBM}(t) \\
    &\hspace*{11mm} - i\Bigl( \frac{1}{2t} \cos\rphase(t) \sin\rphase(t) + \sqrt{t} \Bigr) \diff{t} - \frac{2i}{\sqrt{\beta t}} \cos^2\rphase(t) \diff{\rBM}(t) \\
    &\hspace*{11mm} + \frac{2}{\beta} \Bigl( \frac{1}{t} \cos^4\rphase(t) - \frac{1}{t} \cos^2\rphase(t)\sin^2\rphase(t) - \frac{2i}{t} \cos^3\rphase(t)\sin\rphase(t) \Bigr) \diff{t}.
\end{align*}
The announced stochastic differential equations then follow from taking the real and imaginary parts in the above and simplifying with trigonometric identities. 
\end{proof}

These polar coordinates have essentially the same properties as the polar coordinates \(\AMP\) and \(\phase\) that we used earlier to describe solutions to \(\sAiryop f = 0\). In particular, the growth of \(\rphase\) leads to the following result, analogous to Corollary~\ref{cor.averaging.integrable}.

\begin{lemma}
\label{lem.averagingSAi}
Let \(\rphase\) solve the SDE from Proposition~\ref{prop.polarcoordsSAi}. For any \(k \neq 0\) and \(\varepsilon > 0\), there is a \(C_\varepsilon > 0\) such that
\[
\bprob[\bigg]{\sup_{t\geq 1} \abs[\Big]{\int_1^t \frac{e^{ki\rphase(s)}}{s} \diff{s}} > C_\varepsilon} < \varepsilon.
\]
\end{lemma}

\begin{proof}
By Itô's formula,
\beq{eq.averagingSAi.dekirphase}
\diff{\bigl( e^{ki\rphase} \bigr)}(s)
    = e^{ki\rphase(s)} \Bigl( -ki\sqrt{s} \diff{s} - \frac{3k^2}{4\beta s} \diff{s} - \frac{1}{s} R_k\bigl( \rphase(s) \bigr) \diff{s} - \frac{2ki}{\sqrt{\beta s}} \cos^2\rphase(s) \diff{\rBM(s)} \Bigr)
\eeq
where
\beq{eq.averagingSAi.defRk}
R_k(\xi) \defeq ki\Bigl( \frac{1}{4} + \frac{1}{\beta} \Bigr) \sin 2\xi + \frac{ki}{2\beta} \sin 4\xi + \frac{k^2}{\beta} \cos 2\xi + \frac{k^2}{4\beta} \cos 4\xi.
\eeq
Therefore,
\[
e^{ki\rphase(s)} \diff{s}
    = \frac{i}{k\sqrt{s}} \diff{\bigl( e^{ki\rphase} \bigr)}(s) + \frac{i}{k\sqrt{s}} e^{ki\rphase(s)} \Bigl( \frac{3k^2}{4\beta s} \diff{s} + \frac{1}{s} R_k\bigl( \rphase(s) \bigr) \diff{s} + \frac{2ki}{\sqrt{\beta s}} \cos^2\rphase(s) \diff{\rBM}(s) \Bigr),
\]
and
\beq{eq.averagingSAi.integral}
\begin{multlined}
\int_1^t \frac{e^{ki\rphase(s)}}{s} \diff{s}
    = \frac{i}{k} \int_1^t \frac{1}{s^{\nicefrac{3}{2}}} \diff{\bigl( e^{ki\rphase} \bigr)}(s) + \int_1^t \Bigl( \frac{3ki}{4\beta} + \frac{i}{k} R_k\bigl( \rphase(s) \bigr) \Bigr) \frac{e^{ki\rphase(s)}}{s^{\nicefrac{5}{2}}} \diff{s} \\
    - \frac{2}{\sqrt{\beta}} \int_1^t \frac{e^{ki\rphase(s)}}{s^2} \cos^2\rphase(s) \diff{\rBM}(s).
\end{multlined}
\eeq
Integrating the first term by parts gives
\[
\frac{i}{k} \int_1^t \frac{1}{s^{\nicefrac{3}{2}}} \diff{\bigl( e^{ki\rphase} \bigr)}(s)
= \frac{i}{k} \frac{e^{ki\rphase(s)}}{s^{\nicefrac{3}{2}}} \biggr\rvert_1^t + \frac{3i}{2k} \int_1^t \frac{e^{ki\rphase(s)}}{s^{\nicefrac{5}{2}}} \diff{s}.
\]
The first term is immediately bounded by \(\nicefrac{2}{k}\), and since \(s \mapsto s^{\nicefrac{-5}{2}}\) is integrable on \([1,\infty)\), the remaining integral is also bounded by a constant depending only on \(k\). In the same way, by definition of \(R_k\), the second integral of \eqref{eq.averagingSAi.integral} is bounded by a constant depending only on \(k\) and \(\beta\). As to the remaining stochastic integral in \eqref{eq.averagingSAi.integral}, its bracket is bounded by
\[
\frac{4}{\beta} \int_1^t \frac{1}{s^4} \cos^4\rphase(s) \diff{s}
    \leq \frac{4}{3\beta},
\]
and thus Bernstein's inequality shows that
\[
\bprob[\bigg]{\sup_{t\geq 1} \frac{2}{\beta} \abs[\Big]{\int_1^t \frac{e^{ki\rphase(s)}}{s^2} \cos^2\rphase(s) \diff{\rBM}(s)} > x} \leq 4\exp\Bigl( - \frac{3\beta x^2}{8} \Bigr).
\]
Since the first line of~\eqref{eq.averagingSAi.integral} is bounded by a constant \(C > 0\), taking \(x\) large enough so that \(4\exp\bigl( - \frac{3\beta x^2}{8} \bigr) < \varepsilon\) yields the result with \(C_\varepsilon \defeq C + x\).
\end{proof}

From the above averaging result, it is not hard to see how the radial coordinate \(\ramp\) behaves towards \(-\infty\), which gives a similar description to what we found about the radial coordinates \(\AMP\) of solutions to \(\sAiryop f = 0\) in Proposition~\ref{prop.tracebound}.

\begin{proposition}
\label{prop.ampboundSAi}
Let \(\ramp\) solve the SDE from Proposition~\ref{prop.polarcoordsSAi}. Then for any \(\varepsilon, \delta > 0\), there are \(C, C' > 0\) such that
\[
\bprob[\bigg]{\forall t \geq 1, \abs[\Big]{\ramp(t) - \ramp(1) - \frac{1}{2\beta} \log t} \leq C + C'(\log t)^{\nicefrac{1}{2}+\delta}} \geq 1 - \varepsilon.
\]
\end{proposition}

\begin{proof}
Recall from Proposition~\ref{prop.polarcoordsSAi} that \(\ramp\) solves
\[
\diff{\ramp}(t)
    = \frac{1}{t} \biggl( \frac{1}{2\beta} + \Bigl( \frac{1}{4} + \frac{1}{\beta} \Bigr) \cos 2\rphase(t) + \frac{1}{2\beta} \cos 4\rphase(t) \biggr) \diff{t} - \frac{1}{\sqrt{\beta t}} \sin 2\rphase(t) \diff{\rBM}(t)
\]
where \(\rphase\) solves the other SDE in Proposition~\ref{prop.polarcoordsSAi}. By Lemma~\ref{lem.averagingSAi}, given \(\varepsilon > 0\), there are \(C_2, C_4 > 0\) such that for \(k=2\) and \(k=4\),
\beq{eq.ampboundSAi.deterministicints}
\bprob[\bigg]{\sup_{t\geq 1} \abs[\Big]{\int_1^t \frac{\cos k\rphase(s)}{s}} > C_k} < \frac{\varepsilon}{3}.
\eeq
Then, the martingale
\[
M(t) \defeq \frac{1}{\sqrt{\beta}} \int_1^{1+t} \frac{\sin 2\rphase(s)}{\sqrt{s}} \diff{\rBM}(s)
\]
has quadratic variation
\[
\quadvar{M}(t)
    = \frac{1}{\beta} \int_1^{1+t} \frac{\sin^2 2\rphase(s)}{s} \diff{s}
    \leq \frac{1}{\beta} \log(1 + t).
\]
By Proposition~\ref{prop.martingalebound}, it follows that for any \(\delta > 0\) there is a \(y > 0\) such that
\[
\bprob[\bigg]{\sup_{t\geq 1} \frac{\abs{M(t-1)}}{1 + (\frac{1}{\beta} \log t)^{\nicefrac{1}{2}+\delta}} \geq y} < \frac{\varepsilon}{3},
\]
and the result follows by combining this event with those of~\eqref{eq.ampboundSAi.deterministicints}.
\end{proof}

We now show that the noise drives \(e^{i\rphase}\) to become uniform on the unit circle in the limit. The following result, together with Propositions~\ref{prop.polarcoordsSAi} and~\ref{prop.ampboundSAi}, proves Theorem~\ref{thm.asymptoticsSAi} stated in the introduction.

\begin{proposition}
\label{prop.uniformphase}
Let \(\rphase\) solve the SDE from Proposition~\ref{prop.polarcoordsSAi}. If \(U \sim \uniform\mathbb{S}^1\), then \(e^{i\rphase(t)} \to U\) in law as \(t\to\infty\).
\end{proposition}

\begin{proof}
Recall that if \(U \sim \uniform\mathbb{S}^1\), then \(\expect U^m = 0\) for any integer \(m \neq 0\). Hence, by density of trigonometric polynomials in continuous functions on the unit circle, it suffices to prove that \(\expect e^{mi\rphase(t)} \to 0\) as \(t \to \infty\) for all \(m \neq 0\). This means that if \(\varphi_m \defeq \expect X_m\) with \(X_m(t) \defeq \exp\bigl( mi(\rphase(t) + \frac{2}{3} t^{\nicefrac{3}{2}}) \bigr)\), it suffices to prove that \(\varphi_m(t) \to 0\) as \(t\to\infty\) for all \(m \neq 0\). 

Now, from the expression \eqref{eq.averagingSAi.dekirphase} of the Itô differential of \(e^{mi\rphase}\),
\[
\diff{X_m}(s) = - X_m(s) \Bigl( \frac{3m^2}{4\beta s} + \frac{1}{s} R_m\bigl( \rphase(s) \bigr) \Bigr) \diff{s} - \frac{2mi}{\sqrt{\beta s}} X_m(s) \cos^2\rphase(s) \diff{\rBM}(s),
\]
so
\[
\varphi_m(t) = \expect X_m(1) - \expect \int_1^t X_m(s) \Bigl( \frac{3m^2}{4\beta s} + \frac{1}{s} R_m\bigl( \rphase(s) \bigr) \Bigr) \diff{s},
\]
and it follows that
\[
\varphi_m'(s) = - \frac{3m^2}{4\beta s} \varphi_m(s) - \frac{1}{s} \expect X_m(s) R_m\bigl( \rphase(s) \bigr).
\]
To simplify notation, set \(p \defeq \frac{3m^2}{4\beta}\). Multiplying by \(s^p\) and integrating, we get another expression for \(\varphi_m\):
\beq{eq.uniformphase.varphim}
\varphi_m(t) = t^{-p} \varphi_m(1) - \expect t^{-p} \int_1^t s^{p-1} X_m(s) R_m\bigl( \rphase(s) \bigr) \diff{s}.
\eeq
As \(p > 0\), it is clear that the first term vanishes as \(t\to\infty\). Then, recall that \(R_m(\xi)\) was defined in \eqref{eq.averagingSAi.defRk} as a linear combination of \(e^{\pm 2i\xi}\) and \(e^{\pm 4i\xi}\). Therefore, in order to prove that \(\varphi_m(t) \to 0\), it suffices to prove that
\[
\psi_{m,k}(t) \defeq \frac{1}{t^p} \expect \int_1^t s^{p-1} X_m(s) e^{ki\rphase(s)} \diff{s} \to 0
\quadtext{as}
t\to\infty
\]
for \(k = \pm 2\) and \(k = \pm 4\). 

Using again the expression \eqref{eq.averagingSAi.dekirphase} of the Itô differential of \(e^{ki\rphase}\), we can write
\beq{eq.uniformphase.psimk}
\psi_{m,k}(t)
    = \frac{i}{kt^p} \biggl( \expect \int_1^t s^{p-\nicefrac{3}{2}} X_m(s) \diff{\bigl( e^{ki\rphase} \bigr)}(s) + \expect \int_1^t s^{p-\nicefrac{5}{2}} X_m(s) e^{ki\rphase(s)} \Bigl( \frac{3k^2}{4\beta} + R_k\bigl( \rphase(s) \bigr) \Bigr) \diff{s} \biggr).
\eeq
Now, if \(Y_m(s) \defeq s^{p-\nicefrac{3}{2}} X_m(s)\), then
\[
\diff{Y_m}(s)
    = - \frac{3}{2} s^{p-\nicefrac{5}{2}} X_m(s) \diff{s} - s^{p-\nicefrac{5}{2}} X_m(s) R_m\bigl( \rphase(s) \bigr) \diff{s} - \frac{2mi}{\sqrt{\beta}} s^{p-2} X_m(s) \cos^2\rphase(s) \diff{\rBM}(s),
\]
which also shows that
\[
\diff{\crossvar[\big]{Y_m}{e^{ki\rphase}}}(s)
    = - \frac{4km}{\beta} s^{p-\nicefrac{5}{2}} X_m(s) e^{ki\rphase(s)} \cos^4\rphase(s) \diff{s}.
\]
From this, we can integrate the first term of \(\psi_{m,k}\) by parts:
\begin{multline*}
\expect \int_1^t s^{p-\nicefrac{3}{2}} X_m(s) \diff{\bigl( e^{ki\rphase} \bigr)}(s)
    = \expect s^{p-\nicefrac{3}{2}} X_m(s) e^{ki\rphase(s)} \biggr\rvert_1^t + \frac{3}{2} \expect \int_1^t s^{p-\nicefrac{5}{2}} X_m(s) e^{ki\rphase(s)} \diff{s} \\
    + \expect \int_1^t s^{p-\nicefrac{5}{2}} X_m(s) R_m\bigl( \rphase(s) \bigr) e^{ki\rphase(s)} \diff{s} + \frac{4km}{\beta} \expect \int_1^t s^{p-\nicefrac{5}{2}} X_m(s) e^{ki\rphase(s)} \cos^4\rphase(s) \diff{s}.
\end{multline*}
Finally, because \(X_m\), \(R_k(\rphase)\), \(R_m(\rphase)\) and \(e^{ki\rphase}\) are all bounded processes, substituting the result of this integration by parts in the expression~\eqref{eq.uniformphase.psimk} of \(\psi_{m,k}(t)\) shows that \(\abs{\psi_{m,k}(t)} \lesssim t^{\nicefrac{-3}{2}}\), and thus that \(\psi_{m,k}(t) \to 0\) as \(t\to\infty\), which completes the proof.
\end{proof}

\subsection{Convergence of the Weyl--Titchmarsh functions}

We now have all of the pieces that we need to complete the plan laid out at the beginning of the section to prove the convergence of the Weyl--Titchmarsh functions when \(\beta > 2\). The following result completes the proof of Theorem~\ref{thm.WTconvinlaw}.

\begin{theorem}
\label{thm.Airytosine.WTconvergence}
\setshift{E_n}
Let \(\sAiryWT\) and \(\sineWT\) be the Weyl--Titchmarsh functions of the shifted Airy system and of the sine system respectively, both being defined on the probability space from Lemma~\ref{lem.probspace}. When \(\beta > 2\), \(\sAiryWT \to \sineWT\) compactly on \(\UHP\) in probability as \(n\to\infty\), and in particular the corresponding spectral measures also converge in probability.
\end{theorem}

\begin{proof}
Since the transfer matrices of these systems converge compactly in probability by Theorem~\ref{thm.Airytosine.TMconvergence}, by continuity of the map \(\TMLC[0,1] \times \mathscr{C}(\UHP, \mathbb{C}^2) \to \RiemannSphere^{\UHP}\) that sends a transfer matrix and a right boundary condition to the corresponding Weyl--Titchmarsh function (i.e., Theorem~\ref{thm.TMtoWT.LC}), it suffices to prove that the right boundary conditions of the systems converge compactly in probability. To simplify notation in what follows, we drop the \(n\) subscript from the notation as we did in previous proofs, and we understand a limit as \(\shift\to\infty\) as being taken along the sequence \(\{\shift_n\}_{n\in\mathbb{N}}\).

Fix a compact \(K \subset \UHP\) and take \(z \in K\). As the time-changed shifted Airy system is limit point at \(1 + \nicefrac{1}{\sqrt{\shift}}\), its restriction to \([0,1]\) keeps the same Weyl--Titchmarsh function if it is given the \(z\)-dependent \enquote{boundary condition} \(u_z\circ\timechange(1)\) at \(1\), where \(u_z\circ\timechange\) is an integrable solution on \((0, 1+\nicefrac{1}{\sqrt{\shift}})\). By construction of the canonical system, such a solution has the form \(u_z = \sAirySLtoCS^{-1} \begin{smallpmatrix} \quasi{h_z} \\ h_z \end{smallpmatrix}\) where \(h_z\) is an integrable solution to \(\sAiryop h_z = zh_z\), and therefore as seen in~\eqref{eq.integrablesolutiondecomposition},
\[
u_z\circ\timechange(1)
    = \shift^{\nicefrac{1}{4}} \wronskian{h_z}{\sAiryNeumann}(\shift-1) \begin{pmatrix} -e^{-\diffamps(1)} \cos\diffphases(1) \\ 1 \end{pmatrix} + e^{-\AMPN(1)} \begin{pmatrix} \sin\phaseN(1) & \cos\phaseN(1) \\ 0 & 0 \end{pmatrix} \begin{pmatrix} h_z'(\shift-1) \\ h_z(\shift-1) \end{pmatrix}.
\]
Recall also that the right boundary condition of the sine system is \((\Re\HBM(\infty), 1)\), where \(\HBM\) is a hyperbolic Brownian motion with variance \(\nicefrac{4}{\beta}\) started at \(i\) in the upper half-plane, and driven by the complex Brownian motion \(W\) from Lemma~\ref{lem.probspace}.

Now, the integrable solution \(h_z\) is defined only up to a constant factor. To fix this ambiguity, let \(\intsol{\lambda}\) be an integrable solution to \(\unsAiryop\intsol{\lambda} = \lambda\intsol{\lambda}\) that is analytic in \(\lambda\), and consider the event
\beq{eq.Airytosine.WTconvergence.defG}
\mathscr{G}_{\shift} \defeq \Bigl\{ \forall z \in K, \shift^{\nicefrac{1}{4}} \abs[\big]{\wronskian{\intsol{\nicefrac{z}{2\sqrt{\shift}}}}{\unsAiryNeumann}(-1)} \geq 1 \Bigr\}.
\eeq
Following \eqref{eq.solutionseqlaw}, set \(h_z\) on \(\mathscr{G}_{\shift}\) so that for any \(t \geq -\shift\),
\[
\bigl( h_z(\shift+t), \sAiryNeumann(\shift+t) \bigr)
    = \Bigl( \frac{1}{\shift^{\nicefrac{1}{4}} \wronskian{\intsol{\nicefrac{z}{2\sqrt{\shift}}}}{\unsAiryNeumann}(-1)} \intsol{\nicefrac{z}{2\sqrt{\shift}}}(t), \unsAiryNeumann(t) \Bigr)
\]
where \(\unsAiryNeumann\) solves \(\unsAiryop\unsAiryNeumann = 0\) with \(\unsAiryNeumann(-\shift) = 0\) and \(\unsAiryNeumann'(-\shift) = 1\). With this normalization for \(h_z\), the Wronskian \(\shift^{\nicefrac{1}{4}} \wronskian{h_z}{\sAiryNeumann}(\shift-1) = 1\). On \(\mathscr{G}_{\shift}^\complement\), we cannot always choose \(h_z\) in this way since the prefactor of \(\intsol{\nicefrac{z}{2\sqrt{\shift}}}\) might blow up, but this will not make a difference since we will show that \(\pprob{\mathscr{G}_{\shift}} \to 1\) as \(\shift\to\infty\). To do so, we must control the Wronskian.

Because \(\wronskian{\intsol{0}}{\unsAiryNeumann} \equiv \intsol{0}(-\shift)\), then
\[
\wronskian{\intsol{\nicefrac{z}{2\sqrt{\shift}}}}{\unsAiryNeumann}(-1)
    = \intsol{0}(-\shift) + \bigl( \intsol{\nicefrac{z}{2\sqrt{\shift}}} - \intsol{0} \bigr)(-1) \unsAiryNeumann'(-1) - \bigl( \intsol{\nicefrac{z}{2\sqrt{\shift}}}' - \intsol{0}' \bigr)(-1) \unsAiryNeumann(-1).
\]
To control this, we can write \(\intsol{0}(-\shift) = \shift^{\nicefrac{-1}{4}} e^{\ramp(\shift)} \cos\rphase(\shift)\) where \(\ramp\) and \(\rphase\) are the polar coordinates for \(\intsol{0}\) we get from Proposition~\ref{prop.polarcoordsSAi}. Then by Proposition~\ref{prop.ampboundSAi}, we know that for any \(\varepsilon, \delta > 0\), with probability at least \(1 - \nicefrac{\varepsilon}{6}\),
\[
\ramp(\shift)
    \geq \ramp(1) + \Bigl( \frac{1}{2\beta} - \delta \Bigr) \log\shift
\]
for \(\shift\) large enough. Since \(\ramp(1)\) is a well-defined random variable, it is possible to find \(\zeta > 0\) small enough so that \(e^{\ramp(1)} \geq \zeta\) with probability at least \(1 - \nicefrac{\varepsilon}{6}\), and therefore \(e^{\ramp(\shift)} \geq \zeta \shift^{\nicefrac{1}{2\beta}-\delta}\) with probability at least \(1 - \nicefrac{\varepsilon}{3}\). Then, by Proposition~\ref{prop.uniformphase} and the Skorokhod representation theorem, we know that there is a \(U \sim \uniform[0,2\pi)\) such that \(\rphase(\shift) - 2\pi \floor[\big]{\frac{\rphase(\shift)}{2\pi}} \to U\) in probability as \(\shift\to\infty\). Thus, \(\abs[\big]{\rphase(\shift) - 2\pi \floor[\big]{\frac{\rphase(\shift)}{2\pi}} - U} \leq \zeta\) with probability at least \(1 - \nicefrac{\varepsilon}{6}\) for any \(\shift\) large enough. Taking \(\zeta\) smaller if necessary, we can assume that \(\abs{\cos x} \geq 1 - \nicefrac{x^2}{2}\) for \(-\zeta \leq x \leq \zeta\), and as \(U\) is uniform, we may also assume that \(\abs{\cos U} \geq 2\zeta (1 - \nicefrac{\zeta^2}{2})^{-1}\) with probability at least \(1 - \nicefrac{\varepsilon}{6}\). This shows that
\[
\abs[\big]{\cos\rphase(\shift)}
    \geq \abs[\big]{\cos U} \Bigl( 1 - \frac{\zeta^2}{2} \Bigr) - \zeta
    \geq \zeta
\]
with probability at least \(1 - \nicefrac{\varepsilon}{3}\) for any \(\shift\) large enough. It follows that for any \(\shift\) large enough,
\[
\shift^{\nicefrac{1}{4}} \abs{\intsol{0}(-\shift)}
    = e^{\ramp(\shift)} \abs[\big]{\cos\rphase(\shift)}
    \geq \zeta^2 \shift^{\nicefrac{1}{2\beta}-\delta}
\]
with probability at least \(1 - \nicefrac{2\varepsilon}{3}\). Now, by analyticity of \(\lambda \mapsto \intsol{\lambda}(-1)\),
\beq{eq.Airytosine.WTconvergence.analyticity}
\max\Bigl\{ \abs[\big]{\intsol{\nicefrac{z}{2\sqrt{\shift}}}(-1) - \intsol{0}(-1)}, \abs[\big]{\intsol{\nicefrac{z}{2\sqrt{\shift}}}'(-1) - \intsol{0}'(-1)} \Bigr\} \leq \frac{X \diam K}{2\sqrt{\shift}}
\eeq
for some positive random variable \(X\). Hence, there is a \(C > 0\) such that this maximum is bounded by \(\nicefrac{C}{\sqrt{\shift}}\) with probability at least \(1 - \nicefrac{\varepsilon}{6}\), and this is uniform in \(z \in K\). Finally, since
\[
\abs{\unsAiryNeumann(-1)} = \abs{\sAiryNeumann\circ\timechange(1)} \leq \shift^{\nicefrac{-1}{4}} e^{\AMPN(1)}
\qquadtext{and}
\abs{\unsAiryNeumann'(-1)} = \abs{\sAiryNeumann'\circ\timechange(1)} \leq \shift^{\nicefrac{-1}{4}} e^{\AMPN(1)},
\]
it follows from Proposition~\ref{prop.tracebound} that if \(\delta' > 0\), there is a \(C' > 0\) such that both \(\abs{\unsAiryNeumann(-1)} \leq C' \shift^{\nicefrac{-1}{4}+\nicefrac{1}{2\beta}+\delta'}\) and \(\abs{\unsAiryNeumann'(-1)} \leq C' \shift^{\nicefrac{-1}{4}+\nicefrac{1}{2\beta}+\delta'}\) with probability at least \(1 - \nicefrac{\varepsilon}{6}\). Combining everything, we finally get that for all \(z \in K\) and any \(\shift\) large enough,
\beq{eq.Airytosine.WTconvergence.containsG}
\shift^{\nicefrac{1}{4}} \abs[\big]{\wronskian{\intsol{\nicefrac{z}{2\sqrt{\shift}}}}{\unsAiryNeumann}(-1)}
    \geq \zeta^2 \shift^{\nicefrac{1}{2\beta}-\delta} - 2CC'\shift^{\nicefrac{-1}{2}+\nicefrac{1}{2\beta}+\delta'}
\eeq
with probability at least \(1 - \varepsilon\). If we choose \(\delta < \nicefrac{1}{2\beta}\) and \(\delta' < \nicefrac{1}{2} - \nicefrac{1}{2\beta}\), then the right-hand side here diverges, so it is certainly at least \(1\) for \(\shift\) large enough. In that case, the event on which \eqref{eq.Airytosine.WTconvergence.containsG} holds contains \(\mathscr{G}_{\shift}\), so this shows that \(\pprob{\mathscr{G}_{\shift}} \to 1\) as \(\shift\to\infty\).

We now return to the comparison of the right boundary conditions. On \(\mathscr{G}_{\shift}\), the second entries of the right boundary conditions are equal (to \(1\)), and since \(\pprob{\mathscr{G}_{\shift}} \to 1\) it already follows that \(u_{z,2}\circ\timechange(1) \to 1\) in probability. To control the first entry, set \(\secondtolasttime \defeq (1 - \shift^{\nicefrac{-1}{2}+\alpha}) / \diln\) for \(\alpha \in (0, \nicefrac{1}{2})\). From the representation \eqref{eq.integrablesolutiondecomposition}, we see that on \(\mathscr{G}_{\shift}\),
\begin{subequations}
\label{eq.Airytosine.WTconvergence.firstentry}
\begin{align}
\abs[\big]{u_{z,1}\circ\timechange(1) - \Re\HBM(\infty)}
    & \leq \abs[\big]{-e^{-\diffamps(1)}\cos\diffphases(1) + e^{-\diffamps(\secondtolasttime)}\cos\diffphases(\secondtolasttime)} 
    \label{eq.Airytosine.WTconvergence.firstentry.RealmostHBM} \\
    &\hspace*{22mm} + \abs[\big]{-e^{-\diffamps(\secondtolasttime)}\cos\diffphases(\secondtolasttime) - \Re\HBM\circ\slogtime(\secondtolasttime)} 
    \label{eq.Airytosine.WTconvergence.firstentry.ReHBMconvergence} \\
    &\hspace*{22mm} + \abs[\big]{\Re\HBM\circ\slogtime(\secondtolasttime) - \Re\HBM(\infty)} \label{eq.Airytosine.WTconvergence.firstentry.ReHBM} \\
    &\hspace*{22mm} + e^{-\AMPN(1)} \abs[\big]{\sin\phaseN(1)h_z'(\shift-1) + \cos\phaseN(1)h_z(\shift-1)}.
    \label{eq.Airytosine.WTconvergence.firstentry.error}
\end{align}
\end{subequations}
To conclude, we verify that if \(Y_{\shift,z}\) is any of these four terms, then for any \(\varepsilon, \zeta > 0\), \(\pprob[\big]{\mathscr{G}_{\shift} \cap \{\forall z \in K, \abs{Y_{\shift,z}} > \zeta\}} < \varepsilon\) for any \(\shift\) large enough. 

The second and third lines are obvious to control given what we already know. First, recall that by definition, \(\slogtime(\secondtolasttime) = \bigl( \frac{1}{2} - \alpha \bigr) \log\shift\), so of course \(\Re\HBM\circ\slogtime(\secondtolasttime) \to \Re\HBM(\infty)\) a.s., which is enough to control \eqref{eq.Airytosine.WTconvergence.firstentry.ReHBM}. Then, we know from Proposition~\ref{prop.ReHBM} that there is a \(C > 0\) such that for any \(\shift\) large enough,
\[
\bprob[\bigg]{\abs[\Big]{-e^{-\diffamps(\secondtolasttime)} \cos\diffphases(\secondtolasttime) - \Re\HBM\circ\slogtime(\secondtolasttime)} \geq \shift^{\nicefrac{-\alpha}{5}}} \leq \frac{C}{\log\shift},
\]
which is enough to control \eqref{eq.Airytosine.WTconvergence.firstentry.ReHBMconvergence}. 

Now, to control \eqref{eq.Airytosine.WTconvergence.firstentry.RealmostHBM}, recall from \eqref{eq.dRealmostHBM} that
\beq{eq.Airytosine.WTconvergence.dRealmostHBM}
\begin{aligned}
\diff{\bigl( -e^{-\diffamps} \cos\diffphases \bigr)}(t)
    & = \frac{2}{\sqrt{\beta}} e^{-2\AMPN(t)} \cos 2\phaseN(t) \sqrt{\frac{2\diln}{1 - \diln t}} \diff{\sBM}(t) \\
    &\hspace*{22mm} - e^{-2\AMPN(t)} \biggl( \Bigl( \frac{4}{\beta} - 1 \Bigr) \sin 2\phaseN(t) + \frac{4}{\beta} \sin 4\phaseN(t) \biggr) \frac{\diln}{1 - \diln t} \diff{t}.
\end{aligned}
\eeq
Given \(\varepsilon > 0\), we know from Proposition~\ref{prop.tracebound} that for some \(C, C' > 0\),
\[
\bprob[\bigg]{\forall t \in [0,1], \abs[\Big]{\AMPN(t) + \frac{1}{\beta}\log(1-\diln t)} \leq C + C'\bigl( -\log(1-\diln t) \bigr)^{\nicefrac{3}{4}}} > 1 - \frac{\varepsilon}{2}.
\]
On that event, if \(\delta > 0\) then there is a (different) \(C > 0\) such that for any \(t \in [\secondtolasttime, 1]\),
\[
e^{-\AMPN(t)}
    \leq C(1 - \diln t)^{\nicefrac{1}{\beta} - \delta}
    \leq C \shift^{-(\nicefrac{1}{2}-\alpha)(\nicefrac{1}{\beta}-\delta)},
\]
and thus for both \(k=2\) and \(k=4\),
\[
\abs[\Big]{\int_{\secondtolasttime}^1 e^{-2\AMPN(t)} \sin k\phaseN(t) \frac{\diln}{1 - \diln t} \diff{t}}
    \leq C^2 \shift^{-2(\nicefrac{1}{2}-\alpha)(\nicefrac{1}{\beta}-\delta)} \log\frac{1-\diln\secondtolasttime}{1-\diln}
    = \alpha C^2 \shift^{-2(\nicefrac{1}{2}-\alpha)(\nicefrac{1}{\beta}-\delta)} \log\shift.
\]
Taking \(\delta < \nicefrac{1}{\beta}\), this shows that for any \(\zeta > 0\), the second line in \eqref{eq.Airytosine.WTconvergence.dRealmostHBM} is bounded by \(\nicefrac{\zeta}{2}\) with probability at least \(1 - \nicefrac{\varepsilon}{2}\) for any \(\shift\) large enough. On the same good event, the bracket of the stochastic integral in \eqref{eq.Airytosine.WTconvergence.dRealmostHBM} is bounded by
\[
\frac{8\alpha C^4}{\beta} \shift^{-4(\nicefrac{1}{2}-\alpha)(\nicefrac{1}{\beta}-\delta)} \log\shift,
\]
so by Bernstein's inequality, for any \(\zeta > 0\),
\[
\bprob[\bigg]{\sup_{t\in [\secondtolasttime, 1]} \abs[\Big]{\frac{2}{\sqrt{\beta}} \int_{\secondtolasttime}^t e^{-2\AMPN(s)} \cos 2\phaseN(s) \sqrt{\frac{2\diln}{1-\diln s}} \diff{\sBM}(s)} \geq \frac{\zeta}{2}}
    \leq \frac{\varepsilon}{2} + 2\exp\Bigl( - \frac{\beta \shift^{4(\nicefrac{1}{2}-\alpha)(\nicefrac{1}{\beta}-\delta)} \zeta^2}{64\alpha C^4\log\shift} \Bigr).
\]
Therefore, if \(\shift\) is also large enough so that the right-hand side here is bounded by \(\varepsilon\), we get that
\[
\bprob[\bigg]{\abs[\Big]{-e^{-\diffamps(\secondtolasttime)}\cos\diffphases(\secondtolasttime) + e^{-\diffamps(1)}\cos\diffphases(1)} \geq \zeta} < \varepsilon,
\]
which is good enough control on \eqref{eq.Airytosine.WTconvergence.firstentry.RealmostHBM}. 

Finally, in order to control \eqref{eq.Airytosine.WTconvergence.firstentry.error}, we use again Proposition~\ref{prop.tracebound}, which shows that for any \(\varepsilon, \delta > 0\) there is a \(C > 0\) such that \(e^{-\AMPN(1)} < C\shift^{\nicefrac{-1}{2\beta}+\delta}\) with probability at least \(1 - \nicefrac{\varepsilon}{3}\), and this upper bound vanishes as \(\shift\to\infty\) with \(\delta < \nicefrac{1}{2\beta}\). Then, by definition of \(h_z\) on \(\mathscr{G}_{\shift}\),
\[
\abs{h_z(\shift-1)}
    = \abs[\Big]{\frac{1}{\shift^{\nicefrac{1}{4}} \wronskian{\intsol{\nicefrac{z}{2\sqrt{\shift}}}}{\unsAiryNeumann}(-1)} \intsol{\nicefrac{z}{2\sqrt{\shift}}}(-1)}
    \leq \abs[\big]{\intsol{\nicefrac{z}{2\sqrt{\shift}}}(-1) - \intsol{0}(-1)} + \abs[\big]{\intsol{0}(-1)},
\]
and a similar bound in terms of \(\intsol{0}'\) and \(\intsol{\nicefrac{z}{2\sqrt{\shift}}}'\) holds for \(h_z'(\shift-1)\). As seen in~\eqref{eq.Airytosine.WTconvergence.analyticity}, by analyticity of \(\lambda \mapsto \intsol{\lambda}(-1)\), the first term is bounded by \(\nicefrac{C}{\sqrt{\shift}}\) for a constant \(C > 0\), uniformly in \(z \in K\). Since \(\intsol{0}(-1)\) is a well-defined random variable, it follows that there is an \(M > 0\) large enough so that
\[
\pprob[\big]{\mathscr{G}_{\shift} \cap \bigl\{ \forall z \in K, \abs{h_z(\shift-1)} > M \bigr\}} < \frac{\varepsilon}{3},
\]
and again the same works with \(h_z'(\shift-1)\). Therefore, given \(\zeta > 0\), for any \(\shift\) large enough so that \(2CM\shift^{\nicefrac{-1}{2\beta}+\delta} < \zeta\), we find that
\[
\pprob[\Big]{\mathscr{G}_{\shift} \cap \Bigl\{ \forall z \in K, e^{-\AMPN(1)} \abs[\big]{\sin\phaseN(1)h_z'(\shift-1) + \cos\phaseN(1)h_z(\shift-1)} > \zeta \Bigr\}} < \varepsilon.
\]
This finishes to prove the convergence of the first entry of the boundary conditions, and thus concludes the proof.
\end{proof}

\appendix
\section{Concentration inequalities}
\label{apx.concentration}

The purpose of this appendix is to recall and prove some concentration inequalities for martingales that are used in the main text. The results we need are simple extensions of classical concentration inequalities (Freedman's and Azuma's) to martingales with subgaussian increments. Throughout, given a discrete-time martingale \(M = \{M_n\}_{n\in\mathbb{N}_0}\) with respect to a filtration \(\{\mathscr{F}_n\}_{n\in\mathbb{N}_0}\), we denote by \(\Delta M_n \defeq M_n - M_{n-1}\) the increments of \(M\). We define the \emph{bracket} or \emph{quadratic variation} process of \(M\) as
\[
\quadvar{M}_n \defeq \sum_{j=1}^n \bexpect[\big]{(\Delta M_j)^2 \given \mathscr{F}_{j-1}}.
\]
Moreover, given an adapted process \(V = \{V_n\}_{n\in\mathbb{N}}\), we say that \(M\) is \emph{\(V\)-conditionally subgaussian} if each increment \(\Delta M_n\) is \(V_n\)-subgaussian given \(\mathscr{F}_{n-1}\), that is, if for all \(n\in\mathbb{N}\) and all \(\lambda \in \mathbb{R}\),
\[
\bexpect[\big]{e^{\lambda \Delta M_n} \given \mathscr{F}_{n-1}} \leq e^{\nicefrac{\lambda^2 V_n}{2}} \quad\text{a.s.}
\]
We then define the \emph{subgaussian bracket} \(\subGbracket{M}\) of \(M\) as the smallest nonnegative, nondecreasing predictable process so that \(M\) is \(V\)-conditionally subgaussian for \(V_n \defeq \subGbracket{M}_n - \subGbracket{M}_{n-1}\), and we say that \(M\) is \emph{subgaussian} if \(\subGbracket{M} < \infty\).

The first concentration inequality that we need is an extension of Freedman's inequality \cite{freedman_tail_1975}. Our starting point is the following formulation of Freedman's inequality.

\begin{theorem}[Freedman's inequality]
\label{thm.Freedman}
Let \(M = \{M_n\}_{n\in\mathbb{N}_0}\) be a discrete-time martingale with respect to a filtration \(\{\mathscr{F}_n\}_{n\in\mathbb{N}_0}\) and with \(M_0 = 0\). Suppose that for some \(a > 0\), the increments of \(M\) satisfy \(\abs{\Delta M_n} \leq a\) for all \(n \in \mathbb{N}\). Then for any \(N \in \mathbb{N}\) and any \(x, y > 0\),
\[
\bprob[\bigg]{\sup_{1\leq n\leq N} \abs{M_n} > x, \quadvar{M}_N \leq y} \leq 2 \exp\biggl( - \frac{\log 2}{2} \Bigl( \frac{x}{a} \wedge \frac{x^2}{y} \Bigr) \biggr).
\]
\end{theorem}

\begin{proof}
Since \(\quadvar{M}\) is a non-decreasing process, the event that \(\sup_{1\leq n\leq N} M_n > x\) and \(\quadvar{M}_N \leq y\) is contained in the event that \(M_n > x\) and \(\quadvar{M}_n \leq y\) for some \(n\). Hence, it follows from the original statement of Freedman's inequality~\cite[Theorem~1.6]{freedman_tail_1975} that
\beq{eq.Freedman.nonopt}
\bprob[\bigg]{\sup_{1\leq n\leq N} M_n > x, \quadvar{M}_N \leq y}
    \leq e^{\nicefrac{x}{a}} \biggl( \frac{\nicefrac{y}{a^2}}{\nicefrac{x}{a} + \nicefrac{y}{a^2}} \biggr)^{\nicefrac{x}{a}+\nicefrac{y}{a^2}}
    = \exp\biggl( \frac{y}{a^2} \Bigl( - \Bigl( 1 + \frac{ax}{y} \Bigr) \log\Bigl( 1 + \frac{ax}{y} \Bigr) + \frac{ax}{y} \Bigr) \biggr).
\eeq
To get a simpler expression for the exponent, let
\[
b_p(\xi) \defeq - C \xi^p + (1 + \xi)\log(1 + \xi) - \xi
\]
for \(p = 1\) or \(p = 2\), with \(C > 0\) to be determined. Note that \(b_p'(\xi) = -Cp\xi^{p-1} + \log(1 + \xi)\), so in particular for \(\xi \geq 1\), \(b_1'(\xi) = - C + \log(1 + \xi) \geq - C + \log 2\), and \(b_1\) is increasing on \([1,\infty)\) if \(C < \log 2\). Because \(b_1(1) = - C + 2\log 2 - 1\), this guarantees that \(b_1 \geq 0\) on \([1,\infty)\) if \(C < 2 \log 2 - 1\). Then, using the concavity of \(\log(1 + \cdot)\), we see that if \(\xi \in [0,1]\), then \(b_2'(\xi) \geq -2C\xi + \xi\log 2 \geq 0\) if \(C \leq \frac{1}{2}\log 2\), so \(b_2 \geq 0\) on \([0,1]\) under the same condition, since \(b_2(0) = 0\). 

Combining the condition for \(b_2\) on \([0,1]\) with that for \(b_1\) on \([1,\infty)\), we get that if \(C \leq \frac{1}{2}\log 2 < 2\log 2 - 1\), then \(- (1 + \xi)\log(1 + \xi) + \xi \leq - C(\xi \wedge \xi^2)\) for any \(\xi \geq 0\). Taking the best constant \(C = \frac{1}{2}\log 2\), this shows that the exponent of the tail bound~\eqref{eq.Freedman.nonopt} is bounded by
\[
- \frac{\log 2}{2} \frac{y}{a^2} \Bigl( \frac{ax}{y} \wedge \frac{a^2x^2}{y^2} \Bigr)
    = - \frac{\log 2}{2} \Bigl( \frac{x}{a} \wedge \frac{x^2}{y} \Bigr). 
\]
The theorem now follows by applying the above to both \(M\) and \(-M\), and combining the results with a union bound.
\end{proof}

We now prove the following extension of Freedman’s inequality, in which the condition that increments are bounded is weakened at the price of increasing the tail bound.

\begin{corollary}
\label{cor.FreedmanSubG}
Let \(M = \{M_n\}_{n\in\mathbb{N}_0}\) be a discrete-time martingale with respect to a filtration \(\{\mathscr{F}_n\}_{n\in\mathbb{N}_0}\) and with \(M_0 = 0\). Suppose there are \(K, V > 0\) such that for each \(n\), the increment \(\Delta M_n\) satisfies \(\bexpect[\big]{(\Delta M_n)^2 \given \mathscr{F}_{n-1}} \leq V\) and is \(K\)-subgaussian both conditionally on \(\mathscr{F}_{n-1}\) and unconditionally. Then for any \(N \in \mathbb{N}\) and any \(x > 0\),
\[
\bprob[\bigg]{\sup_{1\leq n\leq N} \abs{M_n} > x}
    \leq \bigl( 2 + N(N+1) \bigr) \exp\biggl( - \frac{1}{2} \min\Bigl\{ \Bigl( \frac{x\log 2}{6\sqrt{K}} \Bigr)^{\nicefrac{2}{3}}, \frac{x^2\log 2}{9NV} \Bigr\} \biggr) + 2\exp\biggl( - \frac{x\log 2}{3N\sqrt{2V}} \exp\biggl( \frac{1}{4} \Bigl( \frac{x\log 2}{6\sqrt{K}} \Bigr)^{\nicefrac{2}{3}} \biggr) \biggr).
\]
\end{corollary}

\begin{proof}
Take \(a > 0\) to be determined later. Notice that \(M\) can be decomposed as
\[
M_n = \sum_{j=1}^n \Bigl( \Delta M_j - \bexpect{\Delta M_j \given \mathscr{F}_{j-1}} \Bigr)
    = M_n^\leq + M_n^>
\]
where \(M_n^\leq \defeq \sum_{j=1}^n \Delta M_j^\leq\) and \(M_n^> \defeq \sum_{j=1}^n \Delta M_j^>\) are martingales with increments
\[
\Delta M_j^\leq \defeq \Delta M_j \charf{\{\abs{\Delta M_j} \leq a\}} - \bexpect[\big]{\Delta M_j \charf{\{\abs{\Delta M_j} \leq a\}} \given \mathscr{F}_{j-1}}
\quadtext{and}
\Delta M_j^> \defeq \Delta M_j \charf{\{\abs{\Delta M_j} > a\}} - \bexpect[\big]{\Delta M_j \charf{\{\abs{\Delta M_j} > a\}} \given \mathscr{F}_{j-1}}.
\]

Now, since \(M^\leq\) has bounded increments and has \(\quadvar{M}_N \leq NV\), Freedman's inequality (Theorem~\ref{thm.Freedman}) shows that
\beq{eq.FreedmanSubG.Freedman}
\bprob[\bigg]{\sup_{1\leq n\leq N} \abs{M^\leq_n} > x}
    \leq 2 \exp\biggl( - \frac{\log 2}{2} \Bigl( \frac{x}{2a} \wedge \frac{x^2}{NV} \Bigr) \biggr).
\eeq
Thus, to control the supremum of \(M\), we only need to to control the supremum of \(M^>\). Note that if \(0 < x \leq a\), then
\[
\bprob[\big]{\abs{\Delta M_j} \charf{\{\abs{\Delta M_j}>a\}} > x}
    = \bprob[\big]{\abs{\Delta M_j} > a},
\]
because for \(\abs{\Delta M_j} \charf{\{\abs{\Delta M_j}>a\}} > x\), the indicator must not vanish, so it must be that \(\abs{\Delta M_j} > a\), in which case it is also true that \(\abs{\Delta M_j} > x\). If \(x > a\), then when \(\abs{\Delta M_j} \charf{\{\abs{\Delta M_j}>a\}} > x\) it must also be true that \(\abs{\Delta M_j} > a\), and we conclude that for any \(x > 0\),
\[
\bprob[\big]{\abs{\Delta M_j} \charf{\{\abs{\Delta M_j}>a\}} > x}
    \leq \bprob[\big]{\abs{\Delta M_j} > a}.
\]
Hence, for any \(x > 0\),
\[
\bprob[\bigg]{\sup_{1\leq n\leq N} \abs[\Big]{\sum_{j=1}^n \Delta M_j \charf{\{\abs{\Delta M_j}>a\}}} > x}
    \leq \sum_{n=1}^N \sum_{j=1}^n \bprob[\bigg]{\abs[\big]{\Delta M_j \charf{\{\abs{\Delta M_j}>a\}}} > \frac{x}{n}}
    \leq \sum_{n=1}^N \sum_{j=1}^n \bprob[\big]{\abs{\Delta M_j} > a},
\]
and as increments are \(K\)-subgaussian,
\beq{eq.FreedmanSubG.tails}
\bprob[\bigg]{\sup_{1\leq n\leq N} \abs[\Big]{\sum_{j=1}^n \Delta M_j \charf{\{\abs{\Delta M_j}>a\}}} > x}
    \leq N(N+1) \exp\Bigl( - \frac{a^2}{2K} \Bigr).
\eeq
It remains to control the mean terms of the increments of \(M^>\). By Cauchy--Schwarz, for each \(j\),
\[
\abs[\big]{\bexpect[\big]{\Delta M_j \charf{\{\abs{\Delta M_j}>a\}} \given \mathscr{F}_{j-1}}}
    \leq \sqrt{\bexpect[\big]{(\Delta M_j)^2 \given \mathscr{F}_{j-1}} \bprob[\big]{\abs{\Delta M_j} > a \given \mathscr{F}_{j-1}}},
\]
and as each increment \(\Delta M_j\) is also \(K\)-subgaussian conditionally on \(\mathscr{F}_{j-1}\), it follows that
\[
\abs[\bigg]{\sum_{j=1}^n \bexpect[\big]{\Delta M_j \charf{\{\abs{\Delta M_j}>a\}} \given \mathscr{F}_{j-1}}}
    \leq n\sqrt{2V} \exp\Bigl( - \frac{a^2}{4K} \Bigr).
\]
With the subgaussian tail bound of bounded random variables, this also gives for any \(x > 0\)
\beq{eq.FreedmanSubG.means}
\bprob[\bigg]{\sup_{1\leq n\leq N} \abs[\Big]{\sum_{j=1}^n \bexpect[\big]{\Delta M_j \charf{\{\abs{\Delta M_j}>a\}} \given \mathscr{F}_{j-1}}} > x}
    \leq 2\exp\biggl( - \frac{x\log 2}{N\sqrt{2V}} \exp\Bigl( \frac{a^2}{4K} \Bigr) \biggr).
\eeq

To conclude, we write
\[
\sup_{1\leq n\leq N} \abs{M_n}
    \leq \sup_{1\leq n\leq N} \abs{M^\leq_n} + \sup_{1\leq n\leq N} \abs[\Big]{\sum_{j=1}^n \Delta M_j \charf{\{\abs{\Delta M_j}>a\}}} + \sup_{1\leq n\leq N} \abs[\Big]{\sum_{j=1}^n \bexpect[\big]{\Delta M_j \charf{\{\abs{\Delta M_j}>a\}} \given \mathscr{F}_{j-1}}}.
\]
Then the tail bounds we got in~\eqref{eq.FreedmanSubG.Freedman}, \eqref{eq.FreedmanSubG.tails} and \eqref{eq.FreedmanSubG.means} give
\[
\bprob[\bigg]{\sup_{1\leq n\leq N} \abs{M_n} > x}
    \leq 2\exp\biggl( - \frac{\log 2}{2} \Bigl( \frac{x}{6a} \wedge \frac{x^2}{9NV} \Bigr) \biggr)
    + N(N+1) \exp\Bigl( - \frac{a^2}{2K} \Bigr)
    + 2\exp\biggl( - \frac{x\log 2}{3N\sqrt{2V}} \exp\Bigl( \frac{a^2}{4K} \Bigr) \biggr).
\]
We now choose \(a\) in such a way that the exponents in the first two tail bounds match, that is, so that
\[
\frac{x \log 2}{12a} = \frac{a^2}{2K}
\iff
a^3 = \frac{Kx \log 2}{6},
\qquadtext{which gives}
\frac{x\log 2}{12a} = \frac{a^2}{2K} = \frac{1}{2} \Bigl( \frac{x\log 2}{6\sqrt{K}} \Bigr)^{\nicefrac{2}{3}}.
\]
Using this value for \(a\) in the tail bounds yields the desired inequality.
\end{proof}

Finally, we prove a version of Azuma's inequality for martingales with subgaussian increments.

\begin{theorem}[Azuma's inequality]
\label{thm.Azuma}
Let \(M = \{M_n\}_{n\in\mathbb{N_0}}\) be a martingale with respect to a filtration \(\{\mathscr{F}_n\}_{n\in\mathbb{N}_0}\) and with \(M_0 = 0\). If \(M\) is subgaussian, then for any \(N\in\mathbb{N}\) and any \(x,y > 0\),
\[
\bprob[\bigg]{\sup_{1\leq n\leq N} \abs{M_n} > x, \subGbracket{M}_N \leq y} \leq 2\exp\Bigl( - \frac{x^2}{2y} \Bigr).
\]
\end{theorem}

\begin{proof}
For \(\lambda > 0\), let
\[
\mathcal{E}_n \defeq \exp\Bigl( \lambda M_n - \frac{\lambda^2}{2} \subGbracket{M}_n \Bigr).
\]
Then \(\{\mathcal{E}_n\}_{n\in\mathbb{N}_0}\) is a supermartingale by definition of the subgaussian bracket of \(M\), and it satisfies \(\mathcal{E}_0 = 1\). Hence, if \(T \defeq \inf\{n\in\mathbb{N} : M_n > x \text{ or } \subGbracket{M}_n > y\}\), the optional stopping theorem shows that \(\expect\mathcal{E}_{T\wedge N} \leq 1\) for any \(N \in \mathbb{N}\).

Now, let \(A_N \defeq \{T\leq N\} \cap \{\subGbracket{M}_N \leq y\}\). Note that for \(A_N\) to occur, it must be that \(M_T > x\), because this is the only way that \(T \leq N\) can occur at the same time as \(\subGbracket{M}_N \leq y\). It follows that
\[
1 \geq \expect\mathcal{E}_{T\wedge N}
    \geq \bexpect{\mathcal{E}_{T\wedge N} \given A_N} \pprob{A_N}
    \geq \exp\Bigl( \lambda x - \frac{\lambda^2}{2} y \Bigr) \pprob{A_N},
\]
which gives an upper bound on \(\pprob{A_N}\). Optimizing over \(\lambda\), we see that the best upper bound is achieved with \(\lambda = \nicefrac{x}{y}\). Unraveling the definition of \(A_N\), we get that
\[
\bprob[\bigg]{\sup_{1\leq n\leq N} M_n > x, \subGbracket{M}_N \leq y}
    \leq \exp\Bigl( - \frac{x^2}{2y} \Bigr),
\]
which is half of the desired bound. The full result follows from applying the above to \(-M\) and combining the results for \(M\) and \(-M\) with a union bound.
\end{proof}

\section{Resolvents of canonical systems}
\label{apx.resolvents}

In this section, we give a few results about resolvents of canonical systems. In the case where a canonical system is derived from a Sturm--Liouville operator, we also describe the relationship between the two associated resolvents.

\subsection{Integral representation of a canonical system's resolvent}

Let \(H\) be the coefficient matrix of a canonical system on an interval \((a,b)\). Consider a self-adjoint realization \(\mathcal{S}_H\) of the system's maximal relation \(\mathcal{T}_H\), for example one of the relations described in Theorem~\ref{thm.CS.selfadjointrealizations}. The \emph{resolvent} of \(\mathcal{S}_H\) at a \(z \in \mathbb{C}\) outside of its spectrum is the relation \((\mathcal{S}_H - z)^{-1}\). It can be shown that this relation is in fact always a bounded normal operator on \(L_H^2(a,b)\), although it can have a kernel (see \cite[Theorem~3.1]{remling_spectral_2018}). This resolvent can be described as an integral operator.

\begin{theorem}[Theorem~3.3 of \cite{remling_spectral_2018}]
\label{thm.CS.resolvent}
Let \(\mathcal{S}_H\) be as above and let \(z \in \mathbb{C} \setminus \mathbb{R}\). For \(t \in \{a, b\}\), let \(u_t\) be a non-trivial solution of \(Ju' = -zHu\) that satisfies the boundary or integrability condition near \(t\), and normalize these solutions so that \(u_a^\transpose J u_b \equiv -1\). Then for \(v \in L_H^2(a,b)\),
\[
\bigl( (\mathcal{S}_H - z)^{-1} v \bigr)(t)
    = \int_a^b G(t,s,z) H(s) v(s) \diff{s}
\]
where
\[
G(t,s,z) \defeq u_b(t) u_a^\transpose(s) \charf{(a,t)}(s) + u_a(t) u_b^\transpose(s) \charf{[t,b)}(s).
\]
\end{theorem}

\begin{remark}
It is easy to see that such solutions \(u_a\) and \(u_b\) always exist, and are unique up to a multiplicative constant (see \cite[Lemma~3.2]{remling_spectral_2018}). Moreover, while the normalization condition \(u_a^\transpose J u_b \equiv -1\) might look very strong at first glance, it only concerns the ambiguity in these multiplicative constants. Indeed, the quantity \(u_a^\transpose J u_b\), which can be seen as the Wronskian of \(u_a\) and \(u_b\), is already constant:
\[
(u_a^\transpose J u_b)'
    = - (Ju_a')^\transpose u_b + u_a^\transpose J u_b'
    = z (H u_a)^\transpose u_b - z u_a^\transpose H u_b
    = 0
\]
since \(H\) is symmetric.
\end{remark}

When the system satisfies Hypothesis~\ref{hyp.LCata} and \(z \in \UHP\), the integral kernel \(G\) of the resolvent can be expressed in terms of the system's transfer matrix \(T_H\) and of its Weyl--Titchmarsh function \(m\). Indeed, \(u_a(t) \defeq T_H(t,z) e_0\) and \(u_b(t) \defeq T_H(t,z) \begin{smallpmatrix} m(z) \\ 1 \end{smallpmatrix}\) are solutions of the system by definition of the transfer matrix. If \(b\) is limit circle, then \(m(z) = \projection T_H(b,z)^{-1} u_b(b)\) by definition so \(u_b\) must satisfy the boundary condition at \(b\), and if \(b\) is limit point, then it follows from Weyl theory that \(u_b \in L_H^2(a,b)\)---see \cite[Theorem~3.14(a)]{remling_spectral_2018}. As \(u_a(a) = e_0\) and \(u_a^\transpose J u_b \equiv (u_a^\transpose J u_b)(a) = e_0^\transpose J \begin{smallpmatrix} m(z) \\ 1 \end{smallpmatrix} = -1\), these two solutions can be used to define the kernel \(G\). A direct computation then shows that
\beq{eq.CS.resolventkernel}
G(t,s,z) = T_H(t,z) \begin{pmatrix} m(z) & \charf{[t,b)}(s) \\ \charf{(a,t)}(s) & 0 \end{pmatrix} T_H^\transpose(s,z).
\eeq

\subsection{Remarks on the convergence of canonical systems' resolvents}

Given the integral representation of the resolvent from Theorem~\ref{thm.CS.resolvent} with the kernel \(G\) as written in~\eqref{eq.CS.resolventkernel}, we immediately get the following.

\begin{proposition}
\label{prop.CS.resolventconvergence}
Let \(\mathcal{I} = [a,b)\) or \(\mathcal{I} = [a,b]\). For any \(\testfunc \in \mathscr{C}_c(\mathcal{I}, \mathbb{C}^2)\) and any \(z \in \mathbb{C} \setminus \mathbb{R}\), the mapping \(S_{z, \testfunc}\colon \CS\mathcal{I} \times \TM\mathcal{I} \times \mathbb{C} \to \mathscr{C}(\mathcal{I}, \mathbb{C}^2)\) given by
\[
S_{z, \testfunc}(H, T, m)(t) = T(t,z) \int_a^b \begin{pmatrix} m & \charf{[t,b)}(s) \\ \charf{(a,t)}(s) & 0 \end{pmatrix} T^\transpose(s,z) H(s) \testfunc(s) \diff{s}
\]
is continuous with respect to the topology of compact convergence. In particular, given self-adjoint realizations \(\mathcal{S}_{H_n}\) and \(\mathcal{S}_H\) of canonical systems maximal relations, \((\mathcal{S}_{H_n} - z)^{-1}\testfunc \to (\mathcal{S}_H - z)^{-1}\testfunc\) compactly whenever the associated coefficient matrices, transfer matrices and Weyl--Titchmarsh functions all converge.
\end{proposition}

It is natural to ask whether the above resolvent convergence result can be strengthened. After all, the resolvent of a canonical system is a bounded operator on a Hilbert space, and therefore it would be natural to hope for the resolvents of a convergent sequence of systems to converge in norm or at least in the strong operator topology, for example. However, it is not obvious to extend this result in general, because the Hilbert space \(L_H^2(a,b)\) on which a canonical system's resolvent acts depends on its coefficient matrix \(H\), so it changes along a sequence of systems. To extend the convergence of resolvents to one of operators on a Hilbert space, one would need to isometrically map all of the \(L_H^2(a,b)\) spaces into a given one, and see how the resolvents behave in that space, which is not obvious to do in general. Nevertheless, all \(\mathscr{C}_c(\mathcal{I}, \mathbb{C}^2)\) functions certainly lie in any \(L_H^2(a,b)\) space for \(H \in \CS\mathcal{I}\), hence the above result.

\subsection{Resolvents of Sturm--Liouville operators}

Unsurprisingly, there is a relationship between the resolvent of a generalized Sturm--Liouville operator and that of its associated canonical system. Following \cite[Theorem~7.3]{eckhardt_weyl-titchmarsh_2013}, the resolvent of a generalized Sturm--Liouville operator \(L\) at \(z \in \mathbb{C} \setminus \mathbb{R}\) with appropriate boundary conditions at limit circle endpoints acts on \(g \in L^2\bigl( (a,b), w(t)\diff{t} \bigr)\) as
\beq{eq.SLresolvent}
\bigl( (L - z)^{-1}g \bigr)(t)
    = \int_a^b G_{\mathrm{SL}}(t,s,z) g(s) w(s) \diff{s}
\eeq 
where \(f_a\) and \(f_b\) are solutions to \(Lf = zf\) that satisfy the boundary or integrability conditions at \(a\) and \(b\) respectively, and where the integral kernel is
\[
G_{\mathrm{SL}}(t,s,z) = \frac{1}{\wronskian{f_b}{f_a}} \bigl( f_b(t) f_a(s) \charf{(a,t)}(s) + f_a(t) f_b(s) \charf{[t,b)}(s) \bigr)
\]
with \(\wronskian{f_b}{f_a} \defeq f_b\quasi{f_a} - f_a\quasi{f_b}\) being the (constant) Wronskian of \(f_b\) and \(f_a\). 

Let \(A\colon (a,b) \to \mathbb{R}^{2\times 2}\) denote the matrix which maps solutions of the eigenvalue equation \(Lf = zf\) to solutions of the associated canonical system, as described in Section~\ref{sec.CS.SL}. By construction, the functions \(u_a\) and \(u_b\) defined so that \(Au_a = (\quasi{f_a}, f_a)\) and \(Au_b = (\quasi{f_b}, f_b)\) are solutions of the canonical system and satisfy the boundary or integrability conditions at \(a\) and \(b\) respectively. Using the identity \(M^\transpose J M = J \det M\), which holds for any invertible \(M\), and the fact that \(\det A \equiv \det A_0 = 1\), we see that
\[
u_a^\transpose J u_b
    = u_a^\transpose A^\transpose A^{-\transpose} J A^{-1} A u_b
    = u_a^\transpose A^\transpose J A u_b
    = \begin{pmatrix} \quasi{f_a} & f_a \end{pmatrix} \begin{pmatrix} -f_b \\ \quasi{f_b} \end{pmatrix}
    = \wronskian{f_a}{f_b},
\]
so if we choose \(f_a\) and \(f_b\) to be normalized so that \(\wronskian{f_a}{f_b} \equiv -1\), then \(u_a\) and \(u_b\) can be used to characterize the canonical system's resolvent as in Theorem~\ref{thm.CS.resolvent}. Then, if \(x\) denotes either \(a\) or \(b\), by expanding the definition \(H = w J A^{-1} \begin{smallpmatrix} 0 & 1 \\ 0 & 0 \end{smallpmatrix} A\) of the coefficient matrix, we see that for any \(v \in L_H^2(a,b)\),
\begin{multline*}
u_x^\transpose(s) H(s) v(s)
    = w(s) u_x^\transpose(s) A^\transpose(s) A^{-\transpose}(s) J A^{-1}(s) \begin{pmatrix} 0 & 1 \\ 0 & 0 \end{pmatrix} A(s) v(s) \\
    = w(s) (Au_x)^\transpose(s) J \begin{pmatrix} 0 & 1 \\ 0 & 0 \end{pmatrix} (Av)(s) 
    = w(s) (Au_x)^\transpose(s) \begin{pmatrix} 0 \\ (Av)_2(s) \end{pmatrix}
    = f_x(s) (Av)_2(s) w(s).
\end{multline*}
It follows that if \(\mathcal{S}\) denotes the self-adjoint realization of the canonical system's maximal relation corresponding to the boundary conditions taken here, then
the integral representation of its resolvent at \(z \in \mathbb{C} \setminus \mathbb{R}\) from Theorem~\ref{thm.CS.resolvent} satisfies
\beq{eq.resolventconjugation}
A(t) \bigl( (\mathcal{S} - z)^{-1} v \bigr)(t)
    = \int_a^b \biggl( \begin{pmatrix} \quasi{f_b}(t) \\ f_b(t) \end{pmatrix} f_a(s) \charf{(a,t)}(s) + \begin{pmatrix} \quasi{f_a}(t) \\ f_a(t) \end{pmatrix} f_b(s) \charf{[t,b)}(s) \biggr) (Av)_2(s) w(s) \diff{s}.
\eeq
In particular, the generalized Sturm--Liouville operator's resolvent evaluated at \((Av)_2\) can be extracted as the second entry of this vector.

\section{Proof of Corollary \ref{cor.SAizeta}}
\label{apx.proofSAizeta}

In this section, we detail the proof of Corollary~\ref{cor.SAizeta}. The proof simply amounts to understanding how to express the Weyl--Titchmarsh functions of the Airy and sine systems in terms of the stochastic Airy and zeta functions. 

The Weyl--Titchmarsh function of the shifted and scaled Airy system is \(\sAiryWT(z) = \projection u(0,z)\) where \(u\colon [0,\infty) \times \mathbb{C} \to \mathbb{C}^2\) is defined so that for each \(z \in \mathbb{C}\), \(u(\cdot,z)\) is an integrable solution to \(J\partial_1 u(\cdot,z) = -z \sAirymat u(\cdot,z)\). By construction of the canonical system, \(u(0,z) = \sAirySLtoCS^{-1}(0)\bigl( \partial_1 h(0,z), h(0,z) \bigr)\) where \(h(\cdot,z)\) is an integrable solution to \(\sAiryop h(\cdot,z) = z h(\cdot,z)\) for each \(z \in \mathbb{C}\). This means that \(h(\cdot,z)\) also solves \(\Airyop h(\cdot,z) = (\shift + \frac{z}{2\sqrt{\shift}}) h(\cdot,z)\), and therefore by definition of the stochastic Airy function, \(h(0,z)\) has the same law as \(a\SAi_{-\shift - \nicefrac{z}{2\sqrt{\shift}}}(0)\) for some (possibly random and \(z\)-dependent) \(a \in \mathbb{C} \setminus \{0\}\). Hence,
\beq{eq.WTSAi.apx}
\sAiryWT(z) \eqlaw \frac{\SAi'_{\lambda_{\shift}(z)}(0)}{\sqrt{\shift} \SAi_{\lambda_{\shift}(z)}(0)}
\qquadtext{where}
\lambda_{\shift}(z) = - \shift - \frac{z}{2\sqrt{\shift}},
\eeq
and where the prime denotes a time derivative.

We now recall the definition of the stochastic zeta function \(\zeta_\beta\) introduced by Valkó and Virág in \cite{valko_many_2022}. Let \(b_1\) and \(b_2\) be two independent two-sided standard Brownian motions on \(\mathbb{R}\), and define the processes
\[
y(t) \defeq \exp\Bigl( b_2(t) - \frac{t}{2} \Bigr)
\qquadtext{and}
x(t) \defeq \begin{cases}
    - \int_t^0 y(s) \diff{b_1}(s) & \text{if } s \leq 0, \\
    \int_0^t y(s) \diff{b_1}(s) & \text{if } s \geq 0.
\end{cases}
\]
Then, set
\[
\taumat(t) \defeq \frac{1}{2y(\frac{4}{\beta}\log t)} X_{\tau_\beta}^*\bigl( \tfrac{4}{\beta}\log t \bigr) X_{\tau_\beta}\bigl( \tfrac{4}{\beta}\log t \bigr)
\qquadtext{where}
X_{\tau_\beta} \defeq \begin{pmatrix} 1 & -x \\ 0 & y \end{pmatrix}.
\]
Following \cite[Definition~43]{valko_many_2022}, we now define \(\tau_\beta\) as the random Dirac operator mapping \(u \colon (0,1) \to \mathbb{C}\) to \(\taumat^{-1} J u'\) with boundary conditions \(\mathfrak{u}_0^*J u(0) = \mathfrak{u}_1^*J u(1) = 0\) where again \(\mathfrak{u}_0 \defeq (1,0)\), but now \(\mathfrak{u}_1 \defeq (-q,-1)\) where \(q\) is a standard Cauchy random variable independent of \(b_1\) and \(b_2\).

The stochastic zeta function \(\zeta_\beta\) is the \emph{secular function} of \(\tau_\beta\) as defined in \cite[Definitions~13 and~43]{valko_many_2022}. For our purposes, it suffices to understand \(\zeta_\beta\) through its characterization as the solution of a differential equation: if \(v\colon (0,1] \times \mathbb{C} \to \mathbb{C}^2\) is defined so that for each \(z \in \mathbb{C}\), \(v(\cdot,z)\) is the unique solution to
\[
\taumat^{-1}J \partial_1 v(\cdot,z) = z v(\cdot,z)
\qquadtext{with}
\lim_{t\to 0} v(t,z) = \begin{pmatrix} 1 \\ 0 \end{pmatrix},
\]
then \(\zeta_\beta(z) = (1, -q) v(1,z)\). This can be taken as a definition of \(\zeta_\beta\), since by \cite[Proposition~18]{valko_many_2022}, this differential equation always has a unique solution (see also \cite[Section 5.1]{valko_many_2022}). 

As explained in \cite[Section~4.2]{valko_many_2022}, the operator \(\tau_\beta\) is orthogonally equivalent to the stochastic sine operator. Indeed, if
\[
S \defeq \begin{pmatrix} 1 & 0 \\ 0 & -1 \end{pmatrix}
\qquadtext{and}
Q \defeq \frac{1}{\sqrt{q^2 + 1}} \begin{pmatrix} q & 1 \\ -1 & q \end{pmatrix}
\]
and if \(\rho\) denotes a transformation acting on a function \(f\) on \((0,1]\) by sending it to \(\rho f(t) = f(1-t)\), then the operator \(\rho^{-1} SQ \tau_\beta (SQ)^{-1} \rho\) has the same distribution as the stochastic sine operator by \cite[Proposition~46]{valko_many_2022}. Writing \(\tau_\beta = \taumat^{-1} J \partial\) with \(\partial\) being a derivative operator, this means that the stochastic sine operator \((\sinemat^{-1}\circ\logtime) J \partial\) has the same distribution as \(\rho^{-1} SQ \taumat^{-1} J \partial (SQ)^{-1} \rho\). Now, \(\partial\rho = - \rho\partial\), so this equivalence yields the equivalence of coefficient matrices
\beq{eq.SAizeta.eqlawmats}
\sinemat^{-1}\circ\logtime
    \eqlaw \rho^{-1} SQ \taumat^{-1} J (SQ)^{-1} J \rho.
\eeq 

With \(v\) as above, define \(u \colon [0,1) \times \mathbb{C} \to \mathbb{C}^2\) by setting \(u(t,z) \defeq \rho SQv(t,-z)\). Then by definition of \(v\), for \(t \in [0,1)\) and \(z \in \mathbb{C}\), \(u\) solves
\[
J \partial_1 u(t,z)
    = - J SQ \partial_1 v(1-t, -z)
    = -z J SQ J \taumat(1-t) v(1-t, -z)
    = -z J SQ J \taumat(1-t) (SQ)^{-1} u(t,z),
\]
which shows that \(u(\cdot,z)\) solves the sine canonical system with spectral parameter \(z\). Moreover, by definition of \(v\), as \(t \to 1\),
\[
u(t,z)
    \to SQ \begin{pmatrix} 1 \\ 0 \end{pmatrix}
    = \frac{1}{\sqrt{q^2 + 1}} \begin{pmatrix} q \\ 1 \end{pmatrix}.
\]
By \cite[Lemma~45]{valko_many_2022}, \(q\) is the point to which converges the hyperbolic Brownian motion that drives the sine system in this setup (i.e., with coefficient matrix \(\rho^{-1} J SQ J \taumat (SQ)^{-1} \rho\)). This shows that \(u(\cdot,z)\) satisfies the sine system's boundary condition at \(1\) when \(\beta > 2\). When \(\beta \leq 2\), there is no boundary condition at \(1\), but it is easy to see from the definition of the sine system's norm that an integrable solution must become parallel to \((q, 1)\) as time goes to \(1\). Indeed,
\[
\norm{f}_{\sinemat\circ\logtime}^2
    = \frac{1}{2} \int_0^1 \frac{\norm{\ABM\circ\logtime(t)f(t)}^2}{\Im\HBM\circ\logtime(t)} \diff{t}
\]
and
\[
\ABM\circ\logtime(t) f(t)
    = \begin{pmatrix} f_1(t) - f_2(t) \Re\HBM\circ\logtime(t) \\ f_2(t) \Im\HBM\circ\logtime(t) \end{pmatrix}.
\]
Because \(\Re\HBM(t)\) a.s.\@ does not vanish as \(t\to\infty\) but \(\Im\HBM(t)\) does, \(f(t)\) must become parallel to \((\Re\HBM(\infty), 1)\) as \(t\to 1\) for \(\norm{f}_{\sinemat\circ\logtime}\) to be finite. As an integrable solution exists, the uniqueness statement of \cite[Proposition~18]{valko_many_2022} therefore implies that \(u(\cdot,z)\) is (a multiple of) that integrable solution. Hence, for any \(\beta > 0\), \(u\) can be taken as the solution that defines the system's Weyl--Titchmarsh function, that is, \(\sineWT(z) \eqlaw \projection u(0, z)\). 

Now,
\[
u(0,z) = SQv(1,-z)
    = \frac{1}{\sqrt{q^2+1}} \begin{pmatrix} q & 1 \\ 1 & -q \end{pmatrix} v(1,-z),
\]
so in particular \(\zeta_\beta(-z) = (1 ,-q) v(1, -z) = \sqrt{q^2 + 1}\, u_2(0, z)\). Therefore,
\beq{eq.WTzeta}
\sineWT(z)
    \eqlaw \frac{u_1(0, z)}{u_2(0, z)}
    = \sqrt{q^2 + 1}\, \frac{u_1(0, z)}{\zeta_\beta(-z)}.
\eeq
By general canonical systems theory, the function \(\xi(z) \defeq \sqrt{q^2 + 1}\, u_1(0, -z) = (q, 1) v(1, z)\) is entire.

Thus, Corollary~\ref{cor.SAizeta} simply follows from Theorem~\ref{thm.WTconvinlaw} because of the representations \eqref{eq.WTSAi.apx} and \eqref{eq.WTzeta} of the two Weyl--Titchmarsh functions.

\section{The spectral weights of the Airy canonical system}
\label{apx.spectralweights}

In this section, we show by comparison with finite-dimensional systems that the spectral weights of the Airy canonical system are independent of the eigenvalues and iid.  We do this by proving an intermediate result characterizing a certain random meromorphic function. We denote by \(\Gamma(\alpha, \lambda)\) a Gamma distribution with shape and rate parameters \(\alpha > 0\) and \(\lambda > 0\).

\begin{proposition}
\label{prop.spectralweights}
Let $\{\lambda_j\}_{j \geq 1}$ be the negatives of the eigenvalues of the stochastic Airy operator, i.e., the zeros of $\SAi_{\lambda}(0)$ in $\lambda$.  Let $\{\Gamma_j\}_{j \geq 1}$ be an iid family of $\Gamma(\tfrac{\beta}{2}, \tfrac{4}{\beta})$ (mean-$2$) random variables.  Then
\[
    \sum_{j=1}^\infty \frac{\Gamma_j}{2(\lambda - \lambda_j)^2}
    \eqlaw
    - \partial_\lambda \frac{\partial_t \SAi_\lambda(0)}{\SAi_\lambda(0)},
\]
with the equality in distribution as random analytic functions on the upper half-plane $\mathbb{H}$.
\end{proposition}

From here, we derive that the spectral weights of the Airy operator have the law of these $\Gamma_j$ random variables. 
\begin{corollary}
\label{cor.spectralweights}
The spectral measure of the shifted and scaled Airy canonical system has the distribution 
\[
\sum_{j=1}^\infty {\Gamma_j}\delta_{\Lambda_j}(\lambda),
\]
where $\{\Lambda_j\}_{j\geq 1}$ are the eigenvalues of the canonical system, and $\{\Gamma_j\}_{j\geq 1}$ are iid $\Gamma(\tfrac{\beta}{2}, \tfrac{4}{\beta})$ (mean-$2$) random variables.
\end{corollary}

\begin{proof}
    We can recover the spectral measure $\mu_{\sAirymat}(\diff{z})$ of $\sAiryWT$ by Stieltjes inversion, i.e. 
    \[
    \mu_{\sAirymat}(\diff{z}) = \frac{1}{\pi} \lim_{\varepsilon \to 0^+} \Im\bigl( \sAiryWT(z+i\varepsilon)\bigr) \diff{z},
    \] 
    as a weak limit of measures.

    From \eqref{eq.WTSAi}, the Weyl--Titchmarsh function of the Airy canonical system is 
    \[
    \sAiryWT(z) = \frac{\partial_t \SAi_{\lambda_{\shift}(z)}(0)}{\sqrt{\shift} \SAi_{\lambda_{\shift}(z)}(0)}
    \]
    where $\lambda_{\shift}(z) = - \shift - \frac{z}{2\sqrt{\shift}}$. The function $\frac{\partial_t \SAi_{\lambda_{\shift}(z)}(0)}{\sqrt{\shift} \SAi_{\lambda_{\shift}(z)}(0)}$ has simple poles at $z = -2\sqrt{\shift}( \lambda_j + \shift)$,  where \(\{\lambda_j\}_{j\geq 1}\) are the zeros of \(\lambda \mapsto \SAi_\lambda(0)\). Moreover, by taking limits, we have that the residue at these poles is given by 
    \[
        \lim_{z \to -2\sqrt{\shift}( \lambda_j + \shift)} 
    \bigl(z + 2\sqrt{\shift}( \lambda_j + \shift) \bigr) \frac{\partial_t \SAi_{\lambda_{\shift}(z)}(0)}{ \sqrt{\shift}\SAi_{\lambda_{\shift}(z)}(0)} = - \frac{2 \partial_t\SAi_{\lambda_j}(0)}{ \partial_\lambda(\SAi_{\lambda}(0))\rvert_{\lambda = \lambda_j}}.
    \]
    Now, because \(\lambda \mapsto \SAi_\lambda(0)\) has a simple zero at \(\lambda_j\),
    \begin{multline*}
    \frac{\partial_t \SAi_{\lambda_j}(0)}
    {\partial_\lambda(\SAi_{\lambda}(0))\rvert_{\lambda = \lambda_j}} 
    =
    \lim_{\lambda\to\lambda_j} (\lambda - \lambda_j)^2 \frac{\partial_t \SAi_\lambda(0) \partial_\lambda \SAi_\lambda(0)}{\SAi_\lambda(0)^2} \\
    = \lim_{\lambda\to\lambda_j} (\lambda - \lambda_j)^2 \frac{\partial_t \SAi_\lambda(0) \partial_\lambda \SAi_\lambda(0) - \SAi_\lambda(0) \partial_\lambda\partial_t \SAi_\lambda(0)}{\SAi_\lambda(0)^2}
    = - \lim_{\lambda\to\lambda_j} (\lambda - \lambda_j)^2 \partial_\lambda \frac{\partial_t \SAi_\lambda(0)}{\SAi_\lambda(0)}.
    \end{multline*}
    Hence, if we have realized the $\Gamma_j$'s on the same probability space as the Airy canonical system so that the relation in Proposition \ref{prop.spectralweights} holds almost surely, by Proposition \ref{prop.spectralweights} we have that
    \[
    \frac{\partial_t \SAi_{\lambda_j}(0)}{\partial_\lambda(\SAi_{\lambda}(0))\rvert_{\lambda = \lambda_j}} 
        = \frac{\Gamma_j}{2}.
    \]
    Therefore, as the masses at the poles of $\sAiryWT$ are the negatives of the residues, we find that 
    \[
    \mu_{\sAirymat}(\diff{z}) = \sum_{j=1}^\infty {\Gamma_j}\delta_{-2\sqrt{\shift}( \lambda_j + \shift)}(\diff{z}).
    \qedhere
    \]
\end{proof}

We now turn to the proof of Proposition \ref{prop.spectralweights}. 

\begin{proof}
    We will show the identity holds by using the coupling results of \cite{lambert_strong_2021}, which couple the tridiagonal Dumitriu--Edelman matrix \cite{dumitriu_matrix_2002} to the Airy operator.
    
    Recall that if $\alpha>0$, a $\chi_\alpha$ random variable has density proportional to $x^{\alpha-1} e^{-x^2 / 2} \charf{x>0}$ and $\chi_\alpha^2 \sim$ $\Gamma(\frac{\alpha}{2}, 2)$. In terms of these variables, we define the semi-infinite tridiagonal matrix
    \[
    \renewcommand{\arraystretch}{0.5}
    \setlength{\arraycolsep}{4pt}
        \mathbf{A} = \begin{pmatrix}
        b_1 & a_1 & & \\
        a_1 & b_2 & a_2 & \phantom{\ddots} \\
        & a_2 & b_3 & \ddots \\
        & & \ddots & \ddots
        \end{pmatrix}
    \]
    where \(b_j \sim \normal{0}{2}\) and \(a_j \sim \chi_{\beta j}\). We let $\Phi_N(z) \defeq \det\bigl( [z-(4 N \beta)^{\nicefrac{-1}{2}} \mathbf{A} ]_{N,N} \bigr)$ where $[\mathbf{A}]_{N,N}$ is the top left $N\times N$ submatrix of $\mathbf{A}$.  Then by the Cramer's formula,
    \[
    \frac{\Phi_{N-1}(z)}{\Phi_N(z)} = 
    \Bigl( \bigl[z-(4 N \beta)^{\nicefrac{-1}{2}} \mathbf{A} \bigr]_{N,N}\Bigr)^{-1}_{N,N} = 
    \sum_{j=1}^N \frac{q_j^2}{z - \widetilde{\lambda}^{(N)}_j}
    \]
    where $\{q_j^2\}_{j=1}^N$ are the \emph{spectral weights} and where $\{\widetilde{\lambda}^{(N)}_j\}_{j=1}^N$ are the {eigenvalues} of $\bigl[(4 N \beta)^{\nicefrac{-1}{2}} \mathbf{A}\bigr]_{N,N}$.  By the Dumitriu--Edelman theorem, weights and eigenvalues are independent, and the weights have a $\operatorname{Dirichlet}(\tfrac{\beta}{2}, \tfrac{\beta}{2},\dots, \tfrac{\beta}{2})$ distribution. 

    Following \cite{lambert_strong_2021}, we let for $n \leq N$,
    \[
        w_n(z)\coloneqq \biggl((2 \pi)^{\nicefrac{1}{4}} e^{N z^2} 2^{-n}\bigl(N z^2\bigr)^{\nicefrac{-1}{12}} \sqrt{\frac{n!}{N^n}}\biggr)^{-1},
    \] 
    and we define $\Psi_n(\lambda) = (w_n \Phi_n)(1+\tfrac{\lambda}{2 N^{2/3}})$.  We also let $\lambda_j^{(N)} = 2N^{\nicefrac{2}{3}}(\widetilde{\lambda}^{(N)}_j - 1)$.  Then we have the representation that for $\lambda \in \mathbb{H}$,
    \[
    N \sum_{j=1}^N \frac{q_j^2}{(\lambda - \lambda_j^{(N)})^2} = -N^{\nicefrac{1}{3}}\partial_\lambda \frac{\Psi_{N-1}(\lambda)}{\Psi_N(\lambda)}.
    \]
    On taking $N\to \infty$, the random function on the left-hand side converges weakly and locally uniformly in $\mathbb{H}$ to the random function 
    \[      
    \sum_{j=1}^\infty \frac{\Gamma_j}{2(\lambda - \lambda_j)^2}.
    \]
    By \cite[Theorem~1.1]{lambert_strong_2021}, as \(N\to\infty\) we have the convergence on compacts of 
    \[
        -N^{\nicefrac{1}{3}}\partial_\lambda \frac{\Psi_{N-1}(\lambda)}{\Psi_N(\lambda)}
        = - N^{\nicefrac{1}{3}} \partial_\lambda \frac{\Psi_{N-1}(\lambda) - \Psi_N(\lambda)}{\Psi_N(\lambda)}
        \lawto - \partial_\lambda \frac{\partial_t \SAi_\lambda(0)}{\SAi_\lambda(0)}.
    \qedhere
    \]
\end{proof}

\printbibliography

\end{document}